\documentclass[11pt]{article}
\usepackage[letterpaper, margin=1in]{geometry}
\usepackage{enumitem}
\usepackage{mathtools}
\usepackage{amsthm}
\usepackage{amssymb}
\usepackage{xspace}
\usepackage{bm}
\usepackage{longtable}
\usepackage{booktabs}
\usepackage{graphicx}
\usepackage[left, pagewise]{lineno}
\usepackage{tikz}
  \usetikzlibrary{positioning,fit,calc}
  \tikzstyle{block} = [draw=black, ultra thin, text width=2cm, minimum height=1.3cm, font = {\footnotesize\itshape},align=center]  
  \tikzstyle{arrow} = [thick,->,>=stealth]

  \usepackage[page,toc,titletoc,title]{appendix}
\usepackage[textsize=footnotesize,disable]{todonotes}
\usetikzlibrary{shapes,arrows,positioning,patterns}
\usepackage{comment}
\usepackage{hyperref}
\usepackage{cleveref}
\hypersetup{
    colorlinks=true,
    citecolor=blue,
    linkcolor=blue,
    filecolor=blue,      
    urlcolor=blue,
}

\usepackage[font=footnotesize,labelfont=bf,margin=0.4in]{caption}

\RequirePackage[osf,sc]{mathpazo}
\usepackage{framed}
\usepackage{mdframed}
\usepackage{float}
\usepackage{xcolor}

\newcommand{\ind}[2][]{%
  \mathrel{
    \mathop{
      \vcenter{
        \hbox{\oalign{\noalign{\kern-.3ex}\hfil$\vert$\hfil\cr
              \noalign{\kern-.7ex}
              $\smile$\cr\noalign{\kern-.3ex}}}
      }
    }^{#2}\displaylimits_{#1}
  }
}
\newcommand{\wh}{\widehat}

\renewcommand{\deg}{\mathrm{degeneracy}}
\newcommand{\sz}[1]{\todo[color=pink!40]{sz: #1}}
\newcommand{\szfuture}[1]{\todo[color=green!40]{future: #1}}

\newlength{\leftbarwidth}
\newlength{\leftbarsep}
\DeclareMathOperator{\rk}{rk}
\DeclareMathOperator{\VCdim}{VCdim}
\DeclareMathOperator{\TVCdim}{2VCdim}
\DeclareMathOperator{\fw}{fw}
\DeclareMathOperator{\cw}{cliquewidth}
\DeclareMathOperator{\rw}{rankwidth}
\DeclareMathOperator{\bfw}{bfw}
\newcommand{\fwd}{\dfw}

\DeclareMathOperator{\iw}{iw}
\DeclareMathOperator{\copw}{copwidth}
\DeclareMathOperator{\dfw}{dfw}
\DeclareMathOperator{\adm}{adm}

\DeclareMathOperator{\tww}{tww}
\DeclareMathOperator{\wcol}{wcol}
\DeclareMathOperator{\scol}{scol}

\newcommand\myitem[1]{\hyperref[item:#1]{\emph{\ref*{item:#1}.}{}}\xspace}

\newcommand{\appmark}{$\ast$}
\newtheorem{theorem}{Theorem}[section]
\newtheorem*{theorem*}{Theorem}
\newtheorem{conjecture}[theorem]{Conjecture}
\newtheorem{question}[theorem]{Question}

\newtheorem{corollary}[theorem]{Corollary}
\newtheorem*{corollary*}{Corollary}
\newtheorem{lemma}[theorem]{Lemma}
\newtheorem*{lemma*}{Lemma}
\newtheorem{fact}[theorem]{Fact}

\newtheorem{proposition}[theorem]{Proposition}
\newtheorem*{proposition*}{Proposition}
\newtheorem{definition}[theorem]{Definition}
\newtheorem{claim}{Claim}[section]
\newtheorem{goal}[theorem]{Goal}

\theoremstyle{remark}
\newtheorem{remark}[theorem]{Remark}
\newtheorem{example}[theorem]{Example}

\def\Nesetril{Ne\v{s}et\v{r}il\xspace}
\def\Dvorak{Dvo\v{r}\'{a}k\xspace}
\def\Kral{Kr\'{a}l\xspace}
\def\Gajarsky{Gajarsk\'{y}\xspace}

\newcommand{\cmso}{\mathrm{CMSO}}

\newcommand{\from}{\colon}

\newcommand{\set}[1]{\{#1\}}
\newcommand{\setof}[2]{\set{#1\mid#2}}
\def\phi{\varphi}
\def\cal{\mathcal}
\def\N{\mathbb N}
\def\R{\mathbb R}
\def\epsilon{\varepsilon}
\def\eps{\varepsilon}
\renewcommand{\subset}{\subseteq}
\renewcommand{\setminus}{-}
\renewcommand{\le}{\leqslant}
\renewcommand{\ge}{\geqslant}

\newcommand{\dist}{\mathrm{dist}}

\newcommand\tw{{\rm treewidth}}

\newcommand{\BB}{\cal B}
\newcommand{\CC}{\cal C}
\newcommand{\DD}{\cal D}

\newcommand{\merge}[2]{\stackrel{u\mapsto v}{\longrightarrow}}

\newcommand{\leaves}{\mathrm{Leaves}}

\newcommand{\Oof}{O}

\newcommand{\tup}{\bar}
\newcommand{\tp}{\textnormal{tp}}
\newcommand{\ltp}{\textnormal{ltp}}
\newcommand{\atp}{\textnormal{atp}}

\newcommand{\anonym}[2][]{#2}

\newcommand{\focs}[2][]{#2}

\newcommand{\ERCagreement}{This paper is part of a project that has received funding from the European Research Council (ERC) (grant agreement No 948057 -- {\sc BOBR}).}

\begin{document}
\title{Flip-width: Cops and Robber on dense graphs}
\author{\anonym[Anonymous]
{Szymon Toru\'nczyk\thanks{University of Warsaw, Poland. \ERCagreement}}}
\date{\today}
\maketitle
\anonym{
\begin{picture}(0,0)
  \put(412,-430)
  {\hbox{\includegraphics[width=40px]{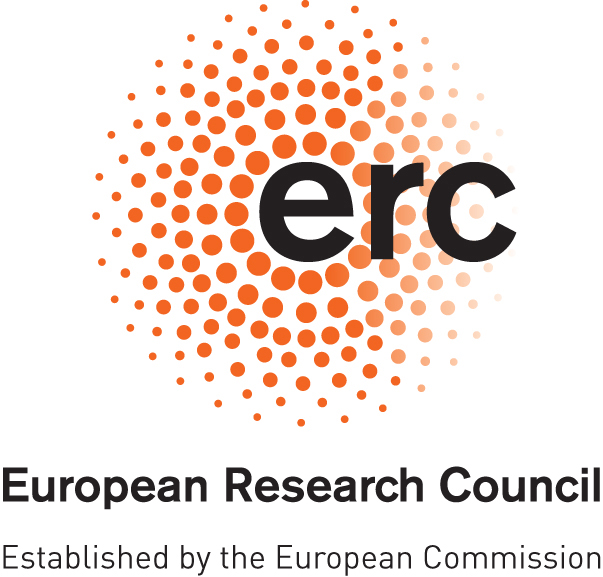}}}
  \put(402,-490)
  {\hbox{\includegraphics[width=60px]{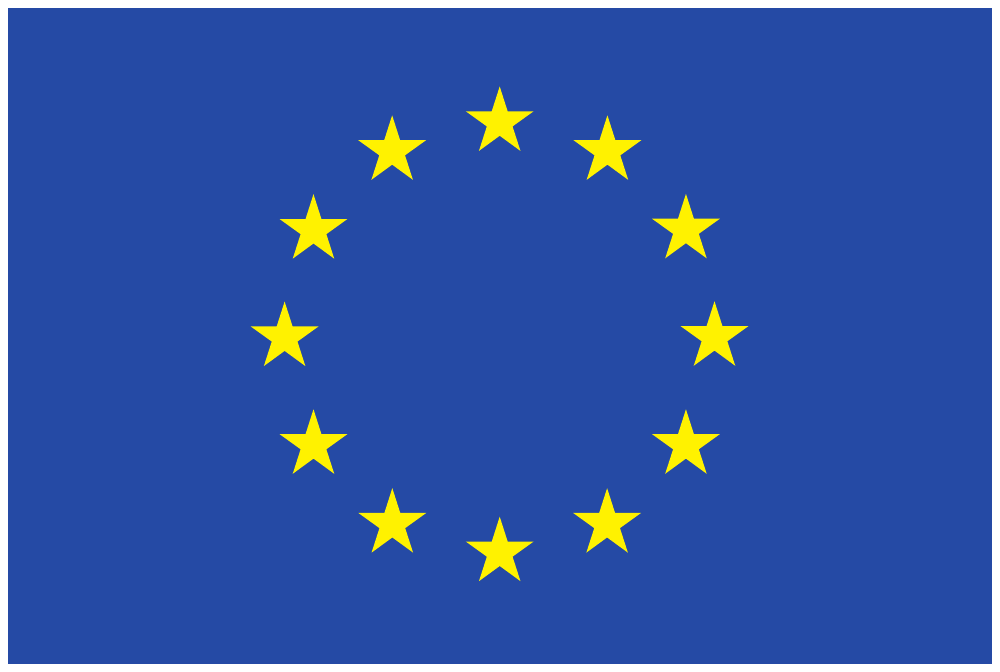}}}
  \end{picture}%
}
\abstract{We define new graph parameters, called \emph{flip-width}, that generalize 
treewidth, degeneracy, and 
generalized coloring numbers for sparse graphs, 
and clique-width  and twin-width for dense graphs. 
The flip-width parameters are defined using 
variants of the Cops and Robber game, in which the robber has speed bounded by a fixed constant $r\in\N\cup\set{\infty}$, and the cops perform flips (or perturbations) of the considered graph.
We then propose a new notion of tameness of a graph class, called \emph{bounded flip-width}, which is a dense counterpart of classes of bounded expansion of \Nesetril and Ossona de~Mendez, and includes classes of bounded twin-width of Bonnet, Kim, Thomass{\'e}, and Watrigant. This unifies Sparsity Theory and Twin-width Theory, for the first time providing a common language for studying the central notions of the two theories, such as weak coloring numbers and twin-width -- corresponding to winning strategies of one player --
or dense shallow minors, rich divisions, or well-linked sets,
corresponding to winning strategies of the other player.
To demonstrate the robustness of the introduced notions,
we prove that boundedness of flip-width is preserved by first-order interpretations, or transductions, generalizing previous results concerning classes of bounded expansion and bounded twin-width. We also show that the considered notions are amenable to algorithms, by providing an algorithm approximating the flip-width of a given graph, which runs in slicewise polynomial time (XP) in the size of the graph.
Finally, we propose a more general notion of tameness, called \emph{almost bounded flip-width}, which is a dense counterpart of nowhere dense classes. We conjecture, and provide evidence, that classes with almost bounded flip-width coincide with monadically dependent (or monadically NIP) classes,
introduced by Shelah in model theory. We also provide evidence that classes of almost bounded flip-width characterise the hereditary graph classes for which the model-checking problem is fixed-parameter tractable,
which is of central importance in structural and algorithmic graph theory.
}

  \thispagestyle{empty}

\newpage

\setcounter{page}{1}


\section{Introduction}\label{sec:intro}
A recent focus of algorithmic and structural 
graph theory, and of finite model theory, 
is to find graph parameters that extend 
the parameters used in the context of sparse graphs,
to the dense setting.
More generally, the goal is to extend the Sparsity theory of Ne\v set\v ril and Ossona de Mendez \cite{sparsity-book} to dense graph classes.
Two central parameters 
used in the sparse setting are treewidth and degeneracy;
 both have found numerous applications in algorithms and combinatorics.
Whereas the dense analogue of treewidth -- clique-width, or rank-width -- is well-understood, there is not even a clear candidate for the dense analogue of degeneracy\footnote{We are only aware of the recent notion 
of graph functionality \cite{functionality} as a possible candidate; see Sec. \ref{sec:fw1-functionality}.}. Recall that a graph 
has degeneracy  at most $d$ if its vertices can be totally ordered
so that every vertex has at most $d$ neighbors before it in the order. 
Generalized coloring numbers are related parameters,
specified by a radius~$r$,
which impose restrictions on neighborhoods of radius $r$, degeneracy being the case of radius $r=1$.
Sparsity theory is a very successful theory
studying classes of sparse graphs 
in which degeneracy and  generalized coloring numbers play a central role. 
The fundamental notions of this theory are two tameness conditions  
for graph classes: \emph{bounded expansion} and \emph{nowhere denseness}.
A class of graphs has 
bounded expansion if and only if 
each generalized coloring number is bounded by a constant (depending on the radius), on all graphs in the class.
Examples include every class with bounded maximum degree,
the class of planar graphs, classes of bounded treewidth, and every class that excludes some graph as a minor or as a topological minor.
The more general nowhere dense classes are characterized  analogously, 
with the constant bound on the generalized coloring numbers of $n$-vertex graphs replaced with $O(n^{\eps})$, for any fixed $\eps>0$.
In particular, nowhere dense classes are sparse -- every $n$-vertex graph in such a class has $O(n^{1+\eps})$ edges for any fixed $\eps>0$, and every nowhere dense class excludes some biclique $K_{t,t}$ as a subgraph.
Classes with bounded expansion and nowhere dense classes can be characterized in many other ways, in terms of their remarkable combinatorial, algorithmic, and logical properties, yielding multiple applications in those areas. 

One of the driving questions in this line of work  \cite{logic-graphs-algorithms}, on the algorithmic side, is to characterize those graph classes for which the model-checking problem for first-order logic is {fixed-parameter tractable}: there is an algorithm
that determines whether a given graph $G$ from the class satisfies a given first-order sentence $\phi$ in time $f(\phi)\cdot |G|^c$, for some constant $c$ and function $f$ that depend only on the class.
Such a characterization is known in the special case of \emph{monotone} graph classes,
that is, graph classes that are closed under removing vertices and edges. In a landmark result, Grohe, Kreutzer and Siebertz~\cite{GroheKS17} proved that a monotone graph class has 
fixed-parameter tractable model-checking if and only if it is nowhere dense
(under a complexity-theoretic assumption, FPT$\neq$AW[$*$]).


There is an ongoing effort to lift Sparsity theory to \emph{hereditary} graph classes, that is, classes that are closed under removing vertices.
For example, the class of cliques is hereditary and combinatorially and logically very simple,
 but lies outside of the realm of Sparsity theory, which is only suited to the study of monotone graph classes. Indeed, all notions studied in Sparsity theory -- generalized coloring numbers, bounded expansion, nowhere denseness, etc. -- are monotone under edge removals.
There are many other known hereditary graph classes that are well-behaved in a similar way to classes with bounded expansion and nowhere dense classes, but are not monotone,
and are not sparse (e.g. contain arbitrarily large cliques).
Those include for instance classes of bounded clique-width, the class of unit interval graphs,
or proper hereditary classes of permutation graphs -- which are all subsumed by the recently introduced classes of bounded twin-width \cite{tww1} (see below) --
as well as
 \emph{structurally nowhere dense}  classes \cite{dreierMS} --  classes of graphs that can be obtained 
from a nowhere dense graph class by redefining the edges using a fixed first-order formula $\phi(x,y)$ -- 
for instance, the edge-complements of graphs from a nowhere dense class, 
or the squares of graphs from a nowhere dense class.

The developments in Structural graph theory -- where many results concerning graph classes of bounded treewidth 
 are extended to the setting of classes of bounded clique-width --  serves as an inspiration in attempts of lifting the results of Sparsity theory from the sparse (monotone) setting to the dense (hereditary) setting.
 It is expected that the fundamental notions of Sparsity theory -- bounded expansion and nowhere denseness -- should extend to more general tameness conditions for graph classes that are possibly dense,
 similarly as treewidth extends to clique-width.
However, currently, even the most fundamental questions
remain unanswered: What is the dense analogue of degeneracy?
 Of generalized coloring numbers?
The pursuit after such notions has been a driving factor, and a major open problem in the area (see \emph{Related Work} below).
To date, no such combinatorial notions, with compelling evidence of their utility,
 have been proposed.

The recent and already very successful Twin-width theory,
developed by Bonnet, Thomass{\'e}, and coauthors
~\cite{tww1,tww2,tww4,tww8},
provides a robust tameness condition for graph classes that are not necessarily sparse,
 classes of \emph{bounded twin-width}.
Those include many studied sparse graph classes, such as classes that exclude a fixed minor, as well as dense  graph classes, such as unit interval graphs, proper hereditary classes of permutation graphs, or posets of bounded width. However, some very simple classes of bounded expansion, such as the class of subcubic graphs (with maximum degree three), have unbounded twin-width. Thus, the scopes of Twin-width theory and Sparsity theory are incomparable. This motivates the quest of finding a unified theory that encompasses both Sparsity theory and Twin-width theory, and provides a common framework for studying the fundamental notions of the two theories.

Both Sparsity theory and Twin-width theory 
have found multiple combinatorial and algorithmic applications\sz{cite}, and the same is expected of a theory unifying the two.
As a concrete application, and a motivation for our pursuit, the sought theory is expected ultimately to resolve
one of the central open problems in structural and algorithmic graph theory --
 of characterizing those hereditary graph classes
for which the model-checking problem is fixed-parameter tractable 
\cite[Sec. 8.2]{logic-graphs-algorithms}
\cite[Sec. 9]{GroheKS17}.

It is conjectured\footnote{The conjecture has been circulating in the community since around 2016.
As far as we know, it has first been stated explicitly during the open problem session of \cite{warwick-problems}. 
There, dependent (or NIP) classes were considered instead of monadically dependent classes, but those two notions coincide for hereditary classes, by a result of Braunfeld and Laskowski~\cite{braunfeld2022existential}.} (see \cite{warwick-problems,new-perspective-journal}) that first-order model-checking
is fixed-parameter tractable on a hereditary graph class $\CC$ 
if and only if $\CC$ is \emph{monadically dependent}
(also called \emph{monadically NIP}).
This notion, formulated in logical terms,
 originates in model theory, and was introduced by Shelah~\cite{Shelah1986} (see also Braunfeld and Laskowski~\cite{Braunfeld2021CharacterizationsOM}) in his momentous classification program of logical theories. Intuitively, a hereditary class $\CC$ 
is monadically dependent if for any fixed first-order formula $\phi(x,y)$,
there is some graph $H$ that cannot be represented in any graph $G\in\CC$,
 using the formula $\phi(x,y)$ to define the edges of $H$ in $G$.

\begin{conjecture}\label{conj:mnip-mc}
  Let $\CC$ be a hereditary class of graphs.
  Then the model-checking problem for first-order logic is fixed-parameter tractable on $\CC$ if and only if~$\CC$ is monadically dependent.
\end{conjecture}
Monadically dependent graph classes include all
the graph classes mentioned above, and 
 are considered (see \cite{AdlerA14,Braunfeld2021CharacterizationsOM,rankwidth-meets-stability,stable-tww-lics}) 
to be the dense counterpart of nowhere dense classes, as expressed e.g. by Conjecture~\ref{conj:mnip-mc}.
For instance, nowhere dense classes are exactly those 
monadically dependent classes that exclude some biclique as a subgraph \cite{rankwidth-meets-stability}.
However, to date, no \emph{combinatorial} characterization of monadically dependent graph classes -- akin to the multiple characterizations of nowhere dense classes -- is known. As a consequence, in general, monadically dependent classes
are currently beyond the reach of algorithmic methods.

\paragraph{Contribution}

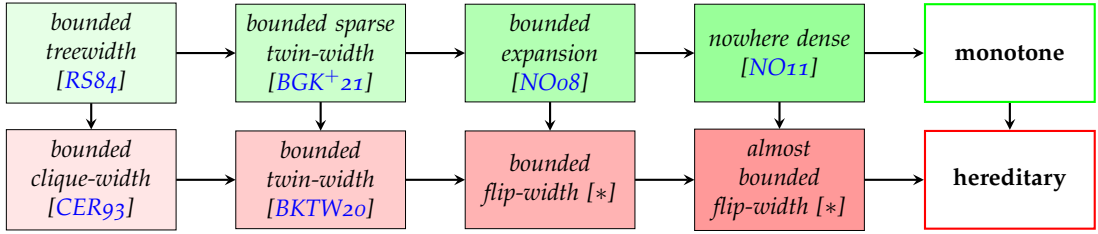
\begin{figure*}
  \begin{centering}
  \begin{tikzpicture}   
    \node[block, fill=green!10] (a) {bounded \mbox{treewidth} \cite{robertson-seymour-tw}};
    \node[block, below=of a, yshift=0.65cm, fill=red!10] (a') {bounded \mbox{clique-width} \cite{clique-width}};
    \node[block, right=of a,  xshift=-0.21cm,fill=green!20] (b) {bounded sparse twin-width \cite{tww2}};
    \node[block, right=of a', xshift=-0.21cm,fill=red!20] (b') {bounded \mbox{twin-width} \cite{tww1}};
    \node[block, right=of b,  xshift=-0.21cm,fill=green!30] (c) {bounded \mbox{expansion} \cite{grad-and-bounded-expansion-Nesetril}};
    \node[block, right=of b',  xshift=-0.21cm,fill=red!30] (c') {bounded \mbox{flip-width} [$*$]};
    \node[block, right=of c,  xshift=-0.21cm,fill=green!40] (d) {\mbox{nowhere dense} \cite{NesetrilM11a}};
    \node[block, right=of c',  xshift=-0.21cm,fill=red!40] (d') {almost bounded flip-width [$*$]%
    };
    \node[block,text width=2cm,right=of d,xshift=-0.21cm,font = {\footnotesize\itshape},draw=green, thick] (e) {\upshape\textbf{monotone}};
    \node[block,text width=2cm,right=of d',xshift=-0.21cm,font = {\footnotesize\itshape},draw=red, thick] (e') {\upshape\textbf{hereditary}};

    \draw [arrow] (a) -- (b);
    \draw [arrow] (b) -- (c);
    \draw [arrow] (c) -- (d);
    \draw [arrow] (d) -- (e);
    \draw [arrow] (a') -- (b');
    \draw [arrow] (b') -- (c');
    \draw [arrow] (c') -- (d');
    \draw [arrow] (d') -- (e');
    \draw [arrow] (a) -- (a');
    \draw [arrow] (b) -- (b');
    \draw [arrow] (c) -- (c');
    \draw [arrow] (d) -- (d');
    \draw [arrow] (e) -- (e');
  \end{tikzpicture}
  \caption{Properties of graph classes, and implications among them. 
  Each property in the lower row, restricted 
  to monotone graph classes, yields the property above it. The properties marked [$*$] are introduced in this paper.
  In this figure, \emph{almost bounded flip-width} could be replaced by \emph{monadically dependent} \cite{Shelah1986}; we conjecture that those properties are equivalent. 
  }\label{fig:diagram}
  \end{centering}
  \end{figure*}
We propose a new family of graph parameters, 
called \emph{flip-width} of radius $r$, for $r\in\N\cup\set{\infty}$, that are based on 
new \emph{flipper games}. Those games are similar to the Cops and Robber game considered by Seymour and Thomas to characterize 
 treewidth, in their classic paper~\cite{seymour-thomas-cops}.
Variants of our game can be used to characterize 
all the aforementioned parameters: treewidth, degeneracy, and generalized coloring numbers used in the context of sparse classes, as well as clique-width and  twin-width in the context of dense graph classes. More importantly, they provide generalizations of degeneracy, and of generalized coloring numbers, to the setting of graphs that are not necessarily sparse, and offer a compelling dense counterpart of classes of bounded expansion -- called classes of bounded \emph{flip-width} -- and of nowhere dense classes -- called classes of \emph{almost bounded flip-width}. 
Classes of bounded flip-width include classes of bounded expansion as well as classes of bounded twin-width, and provide a unified framework for understanding those fundamental notions (see Figure~\ref{fig:diagram}.)
We conjecture, and confirm in many special cases, that classes of \emph{almost} bounded flip-width coincide with monadically dependent classes. If true, this would give a combinatorial, quantitative characterization of monadic dependence,
analogous to the characterization of nowhere dense classes in terms of generalized coloring numbers. Moreover, we verify that classes of almost bounded flip-width include most known examples of hereditary graph classes that are known to have a fixed-parameter tractable model-checking problem
(we could not confirm this only for classes with structurally bounded local cliquewidth).

On a high level, the main contribution of this paper is to lay new foundations of a theory of structurally tractable graph classes, extending 
Sparsity theory to the dense setting, and unifying it with Twin-width theory. 
In this paper, we define the fundamental notions of our theory, and provide overwhelming evidence that they are the sought generalizations of the fundamental notions of Sparsity theory to the dense setting.
This evidence is provided by multiple results, demonstrating that various central concepts of structural graph theory can be  uniformly explained in terms of our notions. 
In the next section, we present an overview of our  
main results. 

\section{Overview}
In this section, we give a high-level overview of the main results of this paper. 
These results are discussed systematically in greater detail, with many illuminating examples and proof sketches, in the following sections. 
Many of the statements recounted in this section 
are simplified versions of the more precise statements given in the next sections. The full statement of each result is referenced in parentheses.
All relevant notions are defined in Section~\ref{sec:prelims}. 

\paragraph{Cop-width}
Our starting point is the -- apparently new -- observation 
that classes with bounded expansion can be characterized in terms of  
a variant of the Cops and Robber game considered 
by Seymour and Thomas~\cite{seymour-thomas-cops}. Recall that in this game,  $k$ cops and a robber are occupying the vertices 
of a graph. In each round, some of the cops move to new positions in helicopters -- that is, not necessarily along edges in the graph -- whereas the robber may traverse 
any path in the graph which avoids the vertices occupied by the cops that remain on ground. The minimal number $k$ of cops needed to capture the robber is equal to one plus the treewidth of the graph.

In Section~\ref{sec:copwidth}
we consider a variant of the Cops and Robber game,
in which the robber runs at speed $r$, for some fixed $r\in\N\cup\set{\infty}$: they may traverse a path of length at most $r$ that does not run through a cop. We call the parameter $r$ the \emph{radius} of the game,
while the number of cops is called the \emph{width} of the game.
The \emph{radius-$r$ cop-width} of $G$, denoted $\copw_r(G)$, is the least number $k$  such that $k$ cops win the  game with radius $r$. 

Thus, we obtain a family of graph parameters, one for each $r\in\N\cup\set{\infty}$.
As we observe, variants of the above game 
characterize the graph parameters mentioned earlier:
treewidth (for $r=\infty$), degeneracy (for $r=1$), and generalized coloring numbers (for  $1\le r<\infty$). Namely, the result of Seymour and Thomas can be phrased as follows.

\begin{theorem}[\cite{seymour-thomas-cops}]
  For every graph $G$,
  \[\copw_\infty(G)=\tw(G)+1.\]
\end{theorem}

On the other extreme, for radius $r=1$, we have the following:
\begin{theorem}[Thm. \ref{thm:degeneracy}]\label{intro:degeneracy}
  For every graph $G$,
  \[\copw_1(G)=\deg(G)+1.\]
\end{theorem}

For higher radii, we prove a correspondence of the cop-width parameters with generalized coloring numbers. Recall that a graph class has bounded expansion if and only if each generalized coloring number is bounded by a constant on all graphs from the class.

\begin{theorem}[Cor.~\ref{cor:copw}]\label{intro:BE}
  A graph class $\CC$ has bounded expansion if and only if
  $\copw_r(\CC)<\infty$ for every fixed $r\in\N$.
\end{theorem}
Here and later on, for a graph parameter $f$  and graph class $\CC$, we denote the supremum of $f(G)$, for $G\in\CC$, by $f(\CC)$.
In Theorem~\ref{intro:degeneracy}, we show that a witness to having degeneracy $d$ -- 
an ordering of the vertices such that every vertex has at most $d$ vertices before it -- yields a winning strategy for the Cops to win the game of width $d+1$ and radius $1$, in which the Cops move to the current position of the robber and all neighbors before it, thus forcing the robber to move upwards in the order.
Conversely, a witness to having degeneracy larger than $d$ --
a set $U$ of vertices such that every vertex in $U$ has more than $d$ neighbors in $U$ -- yields a winning strategy for the robber, allowing him to always remain in the set $U$.
Similarly, to prove Theorem~\ref{intro:BE}, we consider 
the Cops and Robber game for higher radii $1\le r<\infty$, 
and rely on central duality theorems of Sparsity theory,
which describe combinatorial obstructions 
to having a small \emph{weak coloring number}, in terms of a \emph{dense shallow minor} \cite{grad-and-bounded-expansion-Nesetril}.
We show that a witness to having a small weak coloring number yields a winning strategy for the Cops, while a dense shallow minor\sz{address} yields a winning strategy for the Robber in the Cops and Robber game of a fixed, finite radius.
Those observations are analogues to the central duality result of Seymour and Thomas \cite{seymour-thomas-cops}, which 
shows that a witness to having small treewidth yields a winning strategy for the Cops,
while the dual obstruction to small treewidth,  
called a \emph{haven} or \emph{bramble}, yields a winning strategy for the Robber
in the Cops and Robber game of infinite radius.

Parallel to Theorem~\ref{intro:BE},
we obtain a characterization of nowhere dense classes:
\begin{theorem}[Cor. \ref{cor:copw-nd}]A hereditary graph class $\CC$ is nowhere dense if and only if for every fixed radius $r\in \N$ and $\eps>0$, and for all $n$-vertex graphs $G\in\CC$, we have that  $\copw_r(G)\le O(n^{\eps})$.
\end{theorem}

As is made apparent by the results above
(and discussed in Section~\ref{sec:copwidth}) the cop-width parameters capture a substantial part of the fundamental notions of Sparsity theory. However, being 
monotone under edge removals, they are only suited to the study of sparse graphs.

\paragraph{Flip-width}
In  Section~\ref{sec:flipwidth},
we introduce a variant of the Cops and Robber game that is suited for dense graphs, dubbed the \emph{flipper game}.
This is similar to the recent development~\cite{shrubdepth-mfcs, flippers},
where in some contexts, it is shown that vertex removals in sparse graphs correspond to \emph{flips} in dense graphs,
and a different variant of the flipper game is considered
(see \emph{Related work} below).

 \emph{Flipping} a pair of sets $X,Y\subset V(G)$
in a graph $G$ results in the graph $G'$
obtained from $G$ by inverting the adjacency of every pair $(x,y)$ of vertices with $x\in X$ and $y\in Y$:
 every such pair that is adjacent in $G$ becomes non-adjacent in $G'$, and vice-versa.
A \emph{$k$-flip} of a graph $G$ is obtained by partitioning $V(G)$ 
into $k$ parts, and then performing flips between some pairs $X,Y$ of parts of the partition (possibly with $X=Y$).\sz{pic}

To motivate the flipper game, 
observe that in the Cops and Robber game of width $k$, we may think 
of the Cops as \emph{isolating} at most $k$ vertices,
instead of placing the Cops on those vertices,
where isolating a vertex amounts to (temporarily) removing all edges that are incident to it.
Isolating a single vertex $v$ in a graph $G$
can be achieved by performing a $3$-flip of $G$:
partition the vertices into $\set{v}$, the set of neighbors of $v$, and the remaining vertices, and flip the first two sets.
Similarly, a graph $G$ obtained by isolating $k$ vertices $v_1,\ldots,v_k$ is a $(k+2^k)$-flip of $G$, by taking the common refinement of the partitions used for isolating each vertex $v_i$ separately. In the flipper game defined below, we will 
allow one player to perform flips, significantly extending their power comparing to the cops in the Cops and Robber game.

The \emph{flipper game} of radius $r$ and width $k$ on a graph $G$ is a game played by two players, \emph{flipper} and \emph{runner}, proceeds as follows.
Initially, the runner picks any vertex $v_0$ of $G$.
In the $i$th round, the flipper announces a $k$-flip $G_i$ of the original graph $G$.
The runner can move from their previous position $v_{i-1}$ to a new position $v_i$,
by traversing a path of length at most $r$ in the \emph{previous} $k$-flip $G_{i-1}$ of $G$ (where $G_0=G$). 
The game is won by the flipper if the runner's new position $v_i$ is isolated in the announced $k$-flip $G_i$.

The \emph{flip-width} of radius $r$ of a graph $G$, denoted $\fw_r(G)$, is the smallest number $k$ for which the flipper has a winning  strategy in the game described above. 
We say that a graph class $\CC$ has \emph{bounded flip-width} if $\fw_r(\CC)<\infty$ for all $r\in \N$.

Note that if $H$ is an induced subgraph of $G$ then 
$\fw_r(H)\le \fw_r(G)$, for all $r$. In particular, a graph class $\CC$ 
has bounded flip-width if and only if its \emph{hereditary closure} -- consisting of all induced subgraphs of graphs in $\CC$ -- has bounded flip-width.

As we argue below,\sz{move earlier?} the flip-width parameters are the sought dense analogues of the generalized coloring numbers,
while classes of bounded flip-width are the dense analogues of classes of bounded expansion. 
Additionally, we show that variants of the flipper game can be used to uniformly characterize several among the most important graph parameters studied in structural graph theory:
degeneracy, treewidth, and generalized coloring numbers, 
as well as clique-width and twin-width. For the first time, this provides a common perspective on parameters such as degeneracy and twin-width, or generalized coloring numbers and clique-width -- all of which can be seen as representations of winning strategies in the flipper game.
This common perspective also allows to view in a unified way various key combinatorial notions representing obstructions -- namely \emph{havens}  or \emph{brambles}  studied in the context of treewidth \cite{seymour-thomas-cops}, \emph{well-linked sets}  studied in the context of clique-width \cite{OumSeymour-approximating}, \emph{dense shallow minors} studied in the context of bounded expansion classes \cite{grad-and-bounded-expansion-Nesetril}, or \emph{rich divisions} studied in the context of twin-width \cite{tww4}. 
Those notions lie at the core of duality theorems for the related parameters, 
and also have important algorithmic applications. 
As we will see, those notions can be seen as representations of winning strategies for the runner in the flipper game.

We now give a summary of our results concerning flip-width.



\paragraph{Relating flip-width to other notions}
We start with comparing  classes of bounded flip-width with the fundamental notions of Sparsity theory and Twin-width theory.

First, regarding the sparse graph parameters -- degeneracy, treewidth, and generalized coloring numbers -- 
we show that those correspond to 
the parameters $\fw_1$, $\fw_\infty$ and $\fw_r$, for $1\le r<\infty$, 
when considering \emph{weakly sparse} graph classes. A graph class is {weakly sparse} if it excludes some fixed biclique as a subgraph. 

\begin{theorem}\label{intro:wsparse}
  Let $\CC$ be a graph class.
  Then: 
  \begin{enumerate}
    \item (Thm. \ref{thm:fw-deg}) $\CC$ has bounded degeneracy if and only if $\CC$ is weakly sparse and $\fw_1(\CC)<\infty$,
     \item (Cor. \ref{col:wsparse-treewidth}) $\CC$ has bounded treewidth if and only if $\CC$ is weakly sparse and $\fw_\infty(\CC)<\infty$,
    \item (Thm. \ref{thm:wsparse}) $\CC$ has bounded expansion if and only if $\CC$ is weakly sparse and $\CC$ has bounded flip-width.
  \end{enumerate}  
\end{theorem}
This shows that the flipper game of radius $r$ is a sensible generalization of the Cops and Robber game of radius $r$ discussed earlier, since we already  know that the above notions (degeneracy, treewidth, bounded expansion) can be characterized in terms of the latter game.
This fact only gives the rightwards implications in the above statements. To prove the converse implications, we rely on characterizations of $K_{t,t}$-free graphs with small degeneracy/treewidth/expansion in terms of forbidden induced subgraphs.
Such induced subgraphs can be exploited by the runner to win the flipper game.

We then move to studying the flip-width parameters 
in graph classes that are not necessarily sparse.
As a first case study, we verify that 
the flipper game with radius $\infty$ 
corresponds to clique-width, similarly
as the Cops and Robber game corresponds to treewidth.
Namely, we prove:
\begin{theorem}[Thm. \ref{thm:cw}]\label{intro:cw}
  A class $\CC$ of graphs has bounded clique-width 
  if and only if $\fw_\infty(\CC)<\infty$.
\end{theorem}
This gives the first, to our knowledge, game characterization of classes of bounded clique-width,
analogous to the classic characterization of treewidth by Seymour and Thomas.
To prove the rightwards implication, we view 
a clique-width expression (used to construct a graph of bounded clique-width) as a description of a winning strategy for the flipper in the radius-$\infty$ flipper game.
For the converse implication, again we show that obstructions for bounded clique-width -- called \emph{well-linked sets} -- yield winning strategies for the runner.

As clique-width is the dense analogue of treewidth,
Theorem~\ref{intro:cw} is another indication that our flipper game 
is an adequate generalization of the Cops and Robber game for the study of dense graph classes. In Section~\ref{sec:tww}
we provide further evidence of this,
by demonstrating that another crucial graph parameter 
studied in the context of dense graphs, twin-width,
can be characterized in terms of flip-width.
First, we show that classes of bounded flip-width include all classes of bounded twin-width. Those include, 
for example, the class of unit interval graphs,
or every proper hereditary class of permutation graphs\sz{more}.
\begin{theorem}[Thm. \ref{thm:btww}]\label{intro:btww}
  Every class of bounded twin-width has bounded flip-width.
\end{theorem}
The proof of this result views \emph{contraction sequences}, which are recursive decompositions of graphs of bounded twin-width, as strategies for the flipper in the flipper game.

By the results above,
classes of bounded flip-width 
include all classes with bounded expansion as well as classes of bounded twin-width. As those two notions are incomparable, classes of bounded flip-width strictly extend each of them.
In particular, the converse to Theorem~\ref{intro:btww} fails.

To characterize twin-width using flip-width, 
we naturally extend the flip-width parameters to 
arbitrary structures equipped with one or more binary relation, such as
\emph{ordered graphs} -- graphs equipped with a total order.
We prove that for classes of ordered graphs, bounded flip-width and bounded twin-width coincide:
\begin{theorem}[Thm. \ref{thm:tww1}]\label{intro:tww}
  A class of ordered graphs has bounded twin-width 
  if and only if it has bounded flip-width.
\end{theorem}
The rightwards implication follows from Theorem~\ref{intro:btww}.
To prove the converse, we view the combinatorial obstructions to small twin-width, called \emph{rich divisions} \cite{tww4}, as descriptions of winning strategies for the runner in the flipper game.

It is known \cite{tww4} that a class of graphs has bounded twin-width if and only if it can be obtained from a class of ordered graphs of bounded twin-width, by forgetting the order. Thus, classes of bounded twin-width
are exactly classes of ordered graphs of bounded flip-width, with the order forgotten.


\medskip

Reassuming, variants of our flipper game capture degeneracy, treewidth, bounded expansion, clique-width, and twin-width, all of which are of central importance in structural and algorithmic graph theory. Moreover, structural results 
 can be employed to describe winning strategies for the flipper,
while the dual combinatorial obstructions can be used to obtain strategies for the runner.


\paragraph{Closure properties}
Another main contribution of this paper are results showing that classes of bounded flip-width enjoy good closure properties, in particular,
closure under first-order \emph{interpretations} (and more generally, \emph{transductions}, see Theorem~\ref{thm:interpretations}). 
More precisely, for a first-order formula $\phi(x,y)$
and graph $G$, define the graph $\phi(G)$ with vertices $V(G)$ and edges $uv$ with $u\neq w$, such that $\phi(u,v)\lor\phi(v,u)$ holds in $G$. For a graph class $\CC$, let $\phi(\CC):=\setof{\phi(G)}{G\in\CC}$.
For example, if $\phi(x,y)$ is the formula 
$\exists z.E(x,z)\land E(z,y)$ (where $E$ denotes adjacency), then $\phi(G)$ is the square of $G$. In Section~\ref{sec:transductions}, we prove the following.

\begin{theorem}[Thm. \ref{thm:interpretations}]\label{intro:trans}
  Let $\phi(x,y)$ be a first-order formula
  and $\CC$ be a graph class.
  If $\CC$  has bounded flip-width then $\phi(\CC)$ has bounded flip-width.
\end{theorem}

This generalizes a prior analogous result for classes of bounded twin-width~\cite{tww1} and a result concerning classes of bounded expansion.
Also, this result provides further examples of classes of bounded flip-width, for instance, classes of \emph{structurally bounded expansion} \cite{lsd-journal}, that is, classes of the form $\phi(\CC)$, where $\CC$ has bounded expansion and $\phi(x,y)$ is a first-order formula. Moreover, Theorem~\ref{intro:trans} implies that classes of bounded flip-width are monadically dependent,
a notion that conjecturally delimits the
 tractability frontier for the model checking problem (see Conjecture~\ref{conj:mnip-mc}).\sz{delineated by tww}
 
 The proof of Theorem~\ref{intro:trans} uses a tool from finite model theory -- namely,  locality of first-order logic -- that has so far been used mostly in the context of sparse graph classes, but  more recently has also been successfully applied in the context of dense graph classes \cite{tww1,boundedLocalCliquewidth,dreierMS}.

\paragraph{Approximation}
Determining the exact flip-width of radius $r$ of a given graph $G$ 
seems computationally difficult (the naive approach is exponential in the size of $G$). As our next contribution, in Section~\ref{sec:definable},
we obtain a slicewise polynomial (XP) algorithm that approximates the flip-width of a given graph $G$, which runs in time polynomial in the size of $G$, when the flip-width is considered fixed.

\begin{theorem}[Thm.~\ref{thm:apx}]\label{intro:apx}
  There is a constant $C>0$ and an algorithm that, given as inputs an $n$-vertex graph~$G$ and numbers $r,k\in\N$,
  runs in time $n^{O(k)}\cdot 2^{O(2^k)}$,
  and either concludes that $\fw_r(G)\le 2^k$, 
  or concludes that ${\fw_{5r}(G)\ge C \cdot k^{1/3}}$.
\end{theorem}

Note that no XP algorithm approximating twin-width is known. This suggests that flip-width might be easier to approximate than twin-width, even if it defines a more general notion.

To prove Theorem~\ref{intro:apx},
we define a variant of the flipper game, called the \emph{definable flipper game},
in which the partitions used by the Cops
for defining flips are \emph{definable}: 
they partition the vertex set 
according to their adjacency in a set of vertices of bounded size.
The advantage of this game variant is that it has polynomially many (in terms of the graph size) possible configurations, as opposed to the flipper game, which has exponentially many configurations.
Although not every $k$-flip of a graph $G$ can be defined by using a partition definable using a vertex set of bounded size, we prove that nevertheless the graph parameters defined by the definable flipper game can  still be bounded in terms of the flip-width parameters.
The proof 
uses tools from Vapnik-Chervonenkis theory.

\paragraph{Almost bounded flip-width}
As our final contribution, in Section~\ref{sec:subpoly}
we introduce and study classes of almost bounded flip-width, as a candidate dense counterpart of  nowhere dense classes. 
A graph class $\CC$ has \emph{almost bounded flip-width}
if for every $\eps>0$ and $r\in\N$,
we have $\fw_r(G)\le O_{r,\eps}(n^{\eps})$ for every
$n$-vertex graph $G$ which is an induced subgraph of a graph in $\CC$.
Thus, every class of bounded flip-width has almost bounded flip-width, but the converse does not hold.

We provide substantial evidence that classes of almost bounded flip-width coincide with monadically dependent classes (cf. Conjecture~\ref{conj:mnip-mc}),
and also that the model checking problem is 
fixed-parameter tractable for those classes.
This is corroborated by the following results,
which examine this notion in special cases.

First, in the setting of weakly sparse classes, we prove:
\begin{theorem}[Thm.~\ref{thm:nd-are-subpoly}]
  Let $\CC$ be a weakly sparse graph class.
  The following conditions are equivalent:
  \begin{enumerate}
    \item $\CC$ has almost bounded flip-width,
    \item  $\CC$ is nowhere dense,
    \item  $\CC$ is monadically dependent.
  \end{enumerate}
\end{theorem}

The equivalence of the second and third conditions was known; the new result is the equivalence with the first condition.

In the setting of classes of ordered graphs, we prove an analogous statement, with nowhere denseness replaced by bounded twin-width. Moreover, in this case, almost bounded flip-width collapses to bounded flip-width. This collapse is, ultimately, a consequence of the Marcus-Tardos theorem/Stanley-Wilf conjecture \cite{marcus-tardos} from enumerative combinatorics. 
\begin{theorem}[Thm.~\ref{thm:ordered-abfw}]
  Let $\CC$ be a class of ordered graphs.
  The following conditions are equivalent:
  \begin{enumerate}
    \item $\CC$ has almost bounded flip-width,
    \item $\CC$ has bounded flip-width,
    \item  $\CC$ has bounded twin-width,
    \item  $\CC$ is monadically dependent.
  \end{enumerate}
\end{theorem}
The equivalence of  conditions 3,4 was known previously \cite{tww4},
and their equivalence with condition 2 is by
 Theorem~\ref{intro:tww}.
 The implication 2$\rightarrow$3 is immediate.
We prove the implication 1$\rightarrow$3 using another characterisation from \cite{tww4}.

\medskip
In particular, in the two settings considered in the theorems above, the model checking problem is fixed-parameter tractable, by the results of \cite{GroheKS17} and \cite{tww4}, respectively.

\medskip
We now move to a more general setting than that of weakly sparse classes, namely of \emph{stable} classes.
A class $\CC$ is \emph{edge-stable}
if there is a number $k$ 
such that no graph $G\in\CC$ 
contains vertices $a_1,\ldots,a_k,b_1,\ldots,b_k$ 
such that $E(a_i,b_j)\iff i\le j$ for $i,j\in\set{1,\ldots,k}$
(see Fig.~\ref{fig:half-graph}).
Edge-stable, monadically dependent graph classes 
coincide with \emph{monadically stable} classes \cite{BS1985monadic,rankwidth-meets-stability}, an important subfamily of monadically dependent classes.
Examples include all \emph{structurally nowhere dense classes}, that is, classes of the form $\phi(\CC)$, for some first-order formula $\phi(x,y)$ and nowhere dense class $\CC$.
It is conjectured \cite[Conjecture 6.1]{rankwidth-meets-stability} that all monadically stable graph classes are structurally nowhere dense.
By a recent result \cite{dreierMS},
the model-checking problem is fixed-parameter tractable for all structurally nowhere dense classes,
thus significantly extending the result of \cite{GroheKS17} concerning model-checking on nowhere dense classes.
We prove the following.

\begin{theorem}[Thm.~\ref{thm:snd-are-subpoly}]\label{intro:snd}
  Every structurally nowhere dense class 
  has almost bounded flip-width.
\end{theorem}
In the proof, we use the recent description of 
structurally nowhere dense classes \cite{bushes-lics}, which provides certain treelike decompositions for such classes, that can be used to produce a winning strategy for the flipper. As a weak converse, we prove:

\begin{theorem}[Thm.~\ref{thm:subpoly-mstab}]
  Every edge-stable class of almost bounded flip-width is monadically stable.
\end{theorem}
The proof relies on a recent characterization of obstructions to monadic stability \cite{flippers},
which can be used to produce a winning strategy for the runner.

In particular, if all monadically stable classes  are structurally nowhere dense as conjectured in \cite{rankwidth-meets-stability}, 
the two results imply that among edge-stable graph classes, 
almost bounded flip-width coincides with monadic stability (and thus with monadic dependence).
Furthermore, by the result of 
\cite{dreierMS}, 
the conjecture would imply fixed-parameter tractability for all edge-stable classes of almost bounded flip-width.

Those results provide substantial evidence 
indicating that classes of almost bounded flip-width 
form the dense counterpart of nowhere dense classes,
that they coincide with monadically dependent classes,
and that they admit fixed-parameter tractable model-checking.
In fact, almost all\footnote{One exception are 
classes with structurally bounded local clique-width \cite{boundedLocalCliquewidth},
which are tractable and monadically dependent, but we do not know whether they have almost bounded flip-width.} hereditary graph classes 
having fixed-parameter tractable model-checking that we are aware of \cite{GroheKS17,tww1,dreierMS,tww8},
 have almost bounded flip-width. 
 Based on this, we believe that classes of almost bounded flip-width will play a key role in  combinatorial and algorithmic approaches to the analysis of monadically dependent classes.

 Additionally, in Theorem~\ref{thm:abtww-to-abfw} we prove that every hereditary class of \emph{almost bounded twin-width}
 (that is, class whose $n$-vertex graphs have twin-width  $n^{o(1)}$)
 has almost bounded flip-width.

We supplement our results with a discussion, in Section~\ref{sec:discussion}, providing evidence for  various stated conjectures, 
and outline a potential approach towards an algorithmic and combinatorial understanding of the classes introduced in this paper.

\paragraph{Related work}
Our flipper game is inspired by the paper~\cite{flippers}. 
There, another game based on performing flips is introduced,
and is also called \emph{flipper game}.
We consider both games to be variants of a  
broader family of a flipper games, similarly as there are many variants of the Cops and Robber game.
As far as we know, only two variants of the flipper game have been considered so far, but we anticipate that more variants might emerge.
To distinguish the two variants, we may call the variant introduced in this paper the \emph{open flipper game}, and the variant from \cite {flippers} the \emph{confined flipper game of bounded duration}.
The latter game is an analogue of the \emph{splitter game},
 which characterizes nowhere dense classes \cite{GroheKS17}. 
Similarly, the flipper game from \cite{flippers} characterizes monadically stable graph classes, a strict subset of monadically dependent classes.
 It differs from the open flipper game in two major ways. 
 Firstly, the duration of the game is bounded, whereas in our variant, it is unbounded.
%
Furthermore, in each round of the confined flipper game, the fugitive is confined to the radius-$r$ ball around his present position in the current flip of the graph in all future rounds, not just the next round as in the open game.
 The confined flipper game of bounded duration is used in \cite{flippers} to characterize monadically stable classes, which are 
 incomparable with classes of bounded flip-width, 
 and are likely to be strictly contained in classes of almost bounded flip-width (see Section~\ref{sec:subpoly}).

There are multiple variations of the Cops and Robber game, and some of them are similar to the one considered above for sparse graphs.
In some of those variants, the robber has bounded speed, as in our games.
See \cite{https://doi.org/10.1002/jgt.20591,FOMIN20101167,https://doi.org/10.1002/jgt.21791} and \cite{FOMIN2008236} for a survey.
  The cop-width parameters that we study for sparse graphs
  are closely related to the parameters studied in~\cite{richerby-thilikos-lazy-fugitive} and \cite{otherthilikos}, which are  defined in terms of a similar game,
  as we discuss in Section~\ref{sec:copwidth} and in \focs[\cite{flip-width-arxiv}]{Appendix~\ref{sec:copwidth'}}.

  Some attempts at defining graph classes 
that are dense analogues of classes of bounded expansion were 
made in \cite{rankwidth-colorings,lsd-icalp,bounded-linear-rankwidth,rankwidth-meets-stability,stable-tww-lics,transduction-quasiorder}. In~\cite{rankwidth-colorings,bounded-linear-rankwidth,rankwidth-meets-stability}, the property of having low rank\-width covers is proposed as the dense analogue of bounded expansion.
However,  this notion does not include\footnote{An example was provided by Jakub \Gajarsky (private communication)} 
 classes of bounded twin-width, and it is not known whether this notion is closed under transductions. It has been proved recently that those classes are monadically dependent \cite{horizons}.
In \cite[Section 4]{stable-tww-lics} and \cite[Section 8]{transduction-quasiorder}, an attempt at formalizing an abstract notion of a dense analogue of classes with bounded expansion was made, but those papers do not propose any workable combinatorial definition.


\anonym{%
\paragraph{Acknowledgement}
I am very indebted to Édouard Bonnet, Jan Dreier, Jakub \Gajarsky, Rose McCarty,  Micha{\l} Pilipczuk, Wojtek Przybyszewski, 
and Marek Sokołowski,  for numerous stimulating discussions on topics closely related to this paper. Special thanks go to Pierre Ohlmann, who has read a first draft of this paper and helped improve it.
}
\paragraph{Organization}
The organization of this paper is as follows.
Results marked with (\appmark) are proved in the \focs[full 
version of the paper \cite{flip-width-arxiv}]{Appendix}.

Section~\ref{sec:prelims} introduces some standard notation, and recalls some basic notions and results, e.g. from  Sparsity theory and Vapnik-Chervonenkis theory.

In Section~\ref{sec:copwidth} we introduce the cop-width parameters, and demonstrate they characterize degeneracy, treewidth, and generalized coloring numbers. Using those parameters, we characterize classes of bounded expansion and nowhere dense classes.

In Section~\ref{sec:flipwidth} we introduce our two main notions: 
the flip-width parameters, and classes of bounded flip-width. We give many examples, 
and derive some combinatorial properties. We show that radius-$\infty$ flip-width is equivalent to clique-width.

In Section~\ref{sec:wsparse} we study the behavior of the flip-width parameters in the case of $K_{t,t}$-free graphs. We show that radius-one flip-width corresponds to degeneracy, 
and that for higher radii, the flip-width parameters correspond to the generalized coloring numbers. 
Consequently, a weakly sparse graph class has bounded flip-width if and only if it has bounded expansion.

In Section~\ref{sec:tww} we study the relationship between twin-width and flip-width. We prove that every class of bounded twin-width has bounded flip-width. We also prove that for classes of ordered graphs, bounded flip-width and bounded twin-width are equivalent. 

In Section~\ref{sec:transductions} we prove that classes of bounded flip-width are preserved by first-order transductions.

In Section~\ref{sec:definable} we introduce a definable variant of flip-width, and prove its equivalence with flip-width. As a consequence, we get a slicewise polynomial (XP) approximation algorithm for the flip-width parameters.

In Section~\ref{sec:subpoly} we introduce our third main notion, 
classes of almost bounded flip-width. We prove that they contain all structurally nowhere dense classes, and study their relationship with monadically stable and monadically dependent classes.

Finally, in Section~\ref{sec:discussion}, we discuss possible directions of further research, conjectures and questions.

\newpage
\tableofcontents
\newpage

\section{Preliminaries}\label{sec:prelims}
We introduce basic notation in Section~\ref{sec:notation}.
In Section \ref{sec:prelim-sparsity} we recount the fundamental notions of Sparsity theory.
 In Section~\ref{sec:vc} we recall 
 the notion of Vapnik-Chervonenkis dimension of a set system, of a graph, and of a binary relation.
 In Section~\ref{sec:logic}
 we recall basic notions from logic (structures, formulas), 
 and a result characterizing nowhere dense classes in terms 
 of the VC-dimension of certain set systems.

\subsection{Notation}\label{sec:notation}
 $\N=\set{0,\ldots}$ denotes the set of nonnegative integers.
 For two sets $A$ and $B$ their symmetric difference is denoted $A\triangle B:=(A-B)\cup (B-A)$.
 We write $O(n)$ (resp. $\Omega(n)$) to denote a value that is bounded from above (resp. from below) by $c\cdot n+d$,
 for some reals $c,d$ with $c>0$.
 Sometimes we write $O_p(n)$ (resp. $\Omega_p(n)$), where $p$ is a list of parameters, to indicate that the constants $c$ and $d$ above depend on the parameters $p$.
 For a function $f\from \N\to\R$, we write $o(f(n))$ to denote a function $g\from \N\to\R$ with $\lim_{n\rightarrow\infty}\frac{g(n)}{f(n)}=0$.

Graphs are finite, undirected and without self-loops.
The set of vertices of a graph $G$ is denoted $V(G)$,
and the set of edges of $G$ is denoted $E(G)$.  An edge joining $u$ and $v$ is denoted $uv$. In particular, $uv=vu$ and $u\neq v$
for all $uv\in E(G)$. We write $|G|$ for the number of vertices of $G$.
For  a vertex $v$ of a graph $G$  
the (open) \emph{neighborhood} of $v$ in $G$ is
$N_G(v):=
\setof{u}{uv\in E(G)}$, 
denoted $N(v)$ if $G$ is understood from the context.
The set of vertices at distance at most $r$
from $v$ in $G$
is denoted $B_G^r(v)$.

A graph $H$ is a \emph{subgraph} of $G$
if $H$ is obtained by removing vertices and/or edges from $G$, 
and is an \emph{induced subgraph} of $G$ if $H$ is obtained by removing 
vertices from $G$, alongside with the edges incident to them.
The subgraph of $G$ induced by a set of vertices $X\subset V(G)$ is the graph $G[X]$ with vertices $X$ 
and edges $uv\in E(G)$ with $u,v\in X$.
For $X,Y\subset V(G)$, the bipartite graph \emph{semi-induced} by $X$ and $Y$ in a graph $G$ has parts $X$ and $Y$ and edges $xy$ such that $x\in X, y\in Y$ and $xy\in E(G)$.
Note that $X$ and $Y$ need not be disjoint in $G$;
in this case, $G[X,Y]$ contains two copies of every vertex in $X\cap Y$. Two sets $X,Y$ are \emph{complete} in $G$ if $G[X,Y]$ is the complete bipartite graph, and \emph{anti-complete} in $G$ if 
$G[X,Y]$ has no edges, and \emph{homogeneous} if they are either complete or anti-complete.
$K_t$ denotes the complete graph on $t$ vertices, and 
$K_{s,t}$ denotes the complete bipartite graph with parts of sizes $s$ and $t$.

A \emph{graph class} $\CC$ is a set of graphs.
A class $\CC$ is \emph{hereditary} if it is closed under taking induced subgraphs.
The \emph{hereditary closure} of a class $\CC$ 
is the class of all induced subgraphs of graphs from $\CC$.
 $\CC$ is \emph{weakly sparse} if there is some $t\in\N$ such that every $G\in \CC$ excludes the biclique $K_{t,t}$ as a subgraph. 
 A \emph{graph parameter} is a function $f$ assigning reals 
 to graphs, which is invariant under graph isomorphism.
For a graph class $\CC$ and graph parameter $f$,
 denote $f(\CC):=\sup_{G\in\CC}f(G)$, with $f(\CC)=\infty$ if $f$ is unbounded on $\CC$.
 Say that $f$ is \emph{bounded in terms} of $g$
 if there is a function $\alpha\from\R\to\R$ such that $f(G)\le \alpha(g(G))$ for all graphs $G$. Two functions 
   $f$ and $g$ are \emph{functionally equivalent} if each of them is bounded in terms of the other.

   A \emph{$k$-colored graph} is a graph $G$ equipped with a function assigning a color from $\set{1,\ldots,k}$ to each vertex of $G$.
A class $\wh \CC$ of $k$-colored graphs is a \emph{$k$-coloring} of a graph class $\CC$, if for every colored graph in $\wh\CC$, the underlying uncolored graph belongs to $\CC$, and for every graph $G\in\CC$, some $k$-coloring of $G$ belongs to $\wh\CC$.

 \subsection{Sparsity theory}\label{sec:prelim-sparsity}
  We briefly recall the fundamental notions of Sparsity theory. See~\cite{sparsity-book} for more background.

 A graph $G$ is \emph{$d$-degenerate} if there is a total order on the vertices of $G$ such that every vertex has at most $d$ neighbors before it in the order. The \emph{degeneracy} of $G$ is the least $d$ such that $G$ is $d$-degenerate.

 An \emph{exact $r$-subdivision} of a graph $G$
 is the graph obtained by replacing every edge of $G$ 
 by a path of length $r+1$.
 If every edge is replaced by a path of length at most $r+1$,
 the resulting graph is called a \emph{${\le} r$-subdivision} of $G$.
 For a graph $G$, let $\tilde\nabla_r(G)$ denote the maximum average degree, $2|E(H)|/|V(H)|$, of all graphs $H$
 whose ${\le}r$-subdivision is a subgraph of $G$
 (called  \emph{shallow topological minors of $G$ at depth $r$}).
 Note that $\deg(G)\le \tilde\nabla_0(G)\le 2\cdot \deg(G)$.

 \begin{definition}
  A graph class $\CC$ has \emph{bounded expansion}
  if for every $r\ge 1$ we have $\tilde\nabla_r(\CC)<\infty$.
 \end{definition}
 \begin{example}\label{ex:BE-classes}
  Classes of bounded expansion
 include many well-studied sparse classes: the class of planar graphs, 
every class of bounded maximum degree, classes of bounded tree-width, 
classes that exclude a fixed (topological) minor.
\end{example}

The {weak coloring} and  {admissibility} numbers  are two families of graph parameters that generalize
 the degeneracy number to higher radii $r\ge 1$,
 and are defined as follows.
 The \emph{$r$-weak coloring number} of a graph $G$,
 denoted $\wcol_r(G)$, is the smallest number $k$ 
 with the following property:
 There is a total order $<$ on the vertices of $G$
 such that for every vertex~$v$, there are at most $k$ vertices $w$
 with $w<v$
that are \emph{weakly $r$-reachable} from~$v$:
 there is a path of length at most $r$ from $v$ to $w$, in which $w$ is the $<$-smallest vertex. 
 On the other hand, the 
  \emph{$r$-admissibility} of a graph $G$, denoted $\adm_r(G)$, is the smallest number $k$ with the following property:
 There is a 
 total order $<$ on $V(G)$ such that for every $v\in V(G)$ 
 one cannot find more than $k$ paths of length at most $r$ that start at $v$,
 end in some vertex $w<v$, and such that any two of the paths share only $v$ as a common vertex.

Both parameters are of central importance in Sparsity theory.  It is known that $\wcol_r(G)$ is bounded in terms of $\adm_r(G)$, and vice versa~\cite[Lemma 6]{dvorak-admissibility}.
More specifically, we have the following.

\begin{fact}\label{fact:wcol-adm}For all graphs $G$ and $r\ge 1$ we have
\begin{align}\label{eq:wcol-adm}
  \adm_r(G)\le \wcol_r(G)\le O(\adm_r(G))^{r}.  
\end{align}
\end{fact}

The fundamental notion of Sparsity theory  can be characterized using weak coloring numbers as follows:
\begin{fact}[\cite{zhu2009colouring}]\label{fact:be-wcol}
  A class $\CC$ of graphs has bounded expansion if and only if 
$\wcol_r(\CC)<\infty$ for every $r\in\N$.
\end{fact} 
\noindent By~\eqref{eq:wcol-adm}, we could replace $\wcol_r(\CC)$ by $\adm_r(\CC)$ in this characterization.

We use the following inequalities. 

\begin{fact}\label{fact:adm-nabla}
  For all graphs $G$ and $r\ge 1$ we have
\begin{align}\label{eq:adm-nabla}
  \tilde\nabla_{r-1}(G)/2< \adm_r(G)\le 6\left(\tilde\nabla_{r-1}(G)\right)^3.
\end{align}
\end{fact}
\begin{proof}
  The upper bound below is \cite[Theorem 3.1]{doi:10.1137/18M1168753}. We prove
the lower bound. 
Suppose $G$ contains 
an ${\le}(r-1)$-subdivision of a graph with average degree $s$. 
Then $G$ contains 
an ${\le}(r-1)$-subdivision of a graph $H$ with minimum degree larger than $s/2$. 
Let $U\subset V(G)$ consist of the principal vertices, corresponding to the vertices of $H$.
Then for every vertex $u\in U$ there are more than $s/2$ paths of length $r$ that start at $u$, end at vertices of $U$, and are vertex-disjoint apart from $u$. The set $U$ witnesses that $\adm_r(G)>s/2$.
\end{proof}

We now move to nowhere dense classes.
\begin{definition}
  A graph class $\CC$ is \emph{nowhere dense}
  if for every $r\ge 1$ there is some $n\ge 1$ 
  such that for all $G\in\CC$, no ${\le}r$-subdivision of $K_n$ is contained as a subgraph of $G\in\CC$.
 \end{definition}
It is immediate that every class with bounded expansion is nowhere dense,
and there exist nowhere dense classes which do not have bounded expansion.

One of the central results of Sparsity theory is the following characterization of nowhere dense classes (recall that $|G|$ is the number of vertices of $G$).

\begin{fact}
  \label{fact:nd-wcol}
  A hereditary graph class $\CC$ is nowhere dense if and only if 
  for  every $r\ge 1$ and $\eps>0$ there is a constant $n_{r,\eps}$ such that $\wcol_r(G)<|G|^\eps$ for every $G\in\CC$ with $|G|>n_{r,\eps}.$
\end{fact}
The condition in Fact~\ref{fact:nd-wcol}
can be equivalently phrased as follows: for every $r\ge 1$ and $\eps>0$, $\wcol_r(G)\le O_{r,\eps}(|G|^\eps)$ for every $G\in\CC$.
This can be written more concisely as $\wcol_r(G)\le |G|^{o(1)}$.

This fundamental result opens the door for multiple algorithmic applications
of nowhere denseness, 
thanks to the existence of efficient algorithms approximating weak coloring numbers. In particular, the model-checking result~\cite{GroheKS17} relies on  Fact~\ref{fact:nd-wcol}.

\subsection{Vapnik-Chervonenkis dimension}\label{sec:vc}
An important parameter measuring the complexity of graphs, and more generally, of set systems, is the 
\emph{Vapnik-Chervonenkis dimension}, or
\emph{VC-dimension}.

A \emph{set system} is a pair $(X,\cal F)$ with $\cal F\subset 2^X$.
Its \emph{VC-dimension} is the maximal size of a subset $Y\subset X$ such that 
$\setof{Y\cap F}{F\in\cal F}=2^Y$.
We recall the fundamental Sauer-Shelah-Perles lemma~\cite{sauer,shelah-sauer-lemma}.

\begin{lemma}[Sauer-Shelah-Perles lemma]\label{lem:sauer-shelah-perles}
  Let $(X,\cal F)$ be a set system of VC-dimension~$d$.
  Then $|\cal F|\le O(|X|^d)$.
\end{lemma}

The VC-dimension of a graph $G$, denoted $\VCdim(G)$, 
is defined as the VC-dimension of the set system $(V(G),\setof{N(v)}{v\in V(G)})$.
More explicitly, $\VCdim(G)$ is the maximal size of a subset $X\subset V(G)$
such that $\setof{N(v)\cap X}{v\in V(G)}=2^X$.

\newcommand{\cev}[1]{\reflectbox{\ensuremath{\vec{\reflectbox{\ensuremath{#1}}}}}}

For a binary relation $R\subset V\times W$, and elements $a\in V$ and $b\in W$,
denote $\vec R(a):=\setof{w\in W}{(a,w)\in R}$ 
and $\cev R(b):=\setof{v\in V}{(v,b)\in R}$.
The \emph{VC-dimension} of $R$ is the maximum 
of the VC-dimensions of the two set systems 
\[(V,\setof{\vec R(a)}{a\in V})\text{\quad and\quad}(W,\setof{\cev R(b)}{b\in W}).\]

\subsection{Logic}\label{sec:logic}In this paper, we only consider \emph{binary} signatures, that is signatures $\Sigma$
consisting of unary relation symbols, binary relation symbols, and unary function symbols.
Fix a binary signature $\Sigma$.
A $\Sigma$-structure  $B$ consists of a set of vertices $V(B)$,  and is 
equipped with interpretations for each of the symbols in $\Sigma$, as unary relations, binary relations, and unary functions on $V(B)$, respectively. The number of vertices of $B$ is denoted $|B|$.
A graph $G$ is seen as a structure over the signature $\Sigma$ consisting of one binary relation $E$, which is interpreted  as adjacency in $G$.

\emph{Terms} are defined inductively, as either a variable symbol, or a function symbol applied to a term, e.g. $x$ and $f(g(y))$ are terms if $f,g$ are function symbols and $x,y$ are a variable symbols.
A \emph{quantifier-free} formula over the signature $\Sigma$ is a boolean combination 
of \emph{atomic formulas} of the form $U(t(x))$, or $t(x)=t'(y)$, or
$R(t(x),t'(y))$, where $U\in \Sigma$ is a unary relation,  $R\in \Sigma$ is a binary relation, and $t(x)$ and $t'(y)$ are terms.

A \emph{first-order formula} is built inductively:
every  atomic formula is a first-order formula, 
and if $\phi,\psi$ are first-order formulas and $x$ is a variable, then so are $\phi\lor\psi$,  $\neg\phi$, and $\exists x.\phi$. 
The notation $\phi\land\psi$, $\phi\rightarrow \psi$, and $\forall x.\psi$ is used as syntactic sugar.
The \emph{quantifier rank} of a formula $\phi$ 
is the maximal nesting of quantifiers in it.

We may write $\phi(x_1,\ldots,x_k)$ to indicate that the free-variables of $\phi$ are contained in $\set{x_1,\ldots,x_k}$.
For a  formula $\phi(x_1,\ldots,x_k)$, structure $A$ and elements $a_1,\ldots,a_k$ of $A$, we write $A\models\phi(a_1,\ldots,a_k)$ to denote that the valuation of the variables $x_1\mapsto a_1,\ldots,x_k\mapsto a_k$ satisfies the formula $\phi$ in $A$.

Say that a $\Sigma$-formula $\phi(x,y)$ is \emph{symmetric} if $G\models \phi(u,v)\leftrightarrow\phi(v,u)$ holds for every $\Sigma$-structure $G$ and elements $u,v$ of $G$.

The \emph{Gaifman graph} of  a $\Sigma$-structure $B$ is the graph with vertices $V(B)$ and  edges $uv$ with $u\neq v$ such that $(u,v)\in R$
or $(v,u)\in R$ 
for some binary relation $R\in\Sigma$,
or $f(u)=v$ or $f(v)=u$ for some unary function $f\in\Sigma$.

\paragraph{Dependence}
Let $\phi(\tup x;\tup y)$ be a first-order formula, whose set of free variables is partitioned into two disjoint sets $\tup x$ and $\tup y$.
For a structure $H$ define the binary relation $R_H^\phi\subset 
V(H)^{\tup x}\times V(H)^{\tup y}$ as 
\[R_H^\phi=\setof{(\tup u,\tup v) \in V(H)^{\tup x}\times V(H)^{\tup y}}{H\models \phi(\tup u;\tup v)}.\]

Say that a class $\CC$ of $\Sigma$-structures is 
\emph{dependent}, or \emph{NIP} \cite{Shelah1971StabilityTF, Adler2008-ADLAIT} 
if for every first-order formula $\phi(\tup x;\tup y)$ 
  there is some $k_\phi\ge 1$  such that for every $H\in \CC$ 
  the binary relation $R_H^\phi$ has VC-dimension at most~$k_\phi$.
  This is equivalent to saying that for every first-order formula $\phi(\tup x;\tup y)$, there is some bipartite graph $G_\phi$, such that for all $H\in\CC$,  the bipartite graph with parts $V(H)^{\tup x}$, $V(H)^{\tup y}$ and edges $(\tup u,\tup v)$ such that $H\models \phi(\tup u;\tup v)$, does not contain $G_\phi$ as an induced subgraph.

The following fact is a due to Podewski and Ziegler~\cite{podewski1978stable}
(see also~\cite{AdlerA14} and \cite{number-of-types}).

\begin{fact}\label{fact:nd-nip}
  Every nowhere dense graph class $\CC$ is dependent. Conversely, 
  every monotone, dependent graph class $\CC$ is nowhere dense.
\end{fact}

Hereditary, dependent graph classes  are much more general than 
nowhere dense classes. The study of those classes is the main motivation of this paper.
A closely related notion, of a \emph{monadically dependent} class,
   is discussed in Section~\ref{sec:subpoly}.
  Monadically dependent classes are dependent, but in general, the converse implication does not hold. However, as shown by Braunfeld and Laskowski~\cite{braunfeld2022existential},
  for hereditary classes, the two notions coincide.

\section{Cop-width}\label{sec:copwidth}
We start with defining and analyzing the Cops and Robber game with finite radius,
and  the related cop-width parameters.
To define the game, we invoke the original description 
of the Cops and Robber game by 
 Seymour and Thomas~\cite{seymour-thomas-cops}:
``The robber stands on a vertex of the graph, and can at any time run at great speed to any other vertex along a path of the graph. He is not permitted to run through a cop, however. There are $k$ cops, each of whom at any time either stands on a vertex or is in a helicopter (that is, is temporarily removed from the game). The objective of the player controlling the movement of the cops is to land a cop via helicopters on the vertex occupied by the robber, and the robber's objective is to elude capture. (The point of the helicopters is that cops are not constrained to move along paths of the graph -- they move from vertex to vertex arbitrarily.)  The robber can see the helicopter approaching its landing spot and may run to a new vertex before the helicopter actually lands.'' 

Seymour and Thomas proved that the least number of cops needed to catch a robber on a graph $G$ 
is equal to one plus the treewidth of $G$. 
To this end, they proved a \emph{min-max theorem}:
either the cops have a winning strategy of a particularly simple, \emph{monotone} form, which can be described by a tree decomposition of the graph, or otherwise, the robber has a winning strategy of a particularly simple form, called a \emph{haven}.

In our variant of the game the robber runs at speed $r$, for some fixed $r\in\N\cup\set{\infty}$. That is, in each round, after the cops have taken off in their helicopters to their new positions, which are known to the robber, and before the helicopters have landed, the robber may traverse a path of length at most $r$ that does not run through a cop that remains on the ground (he may also stay put). 
We call this game the \emph{Cops and Robber game} with radius $r$ and width $k$, if there are $k$ 
cops, and the robber can run at speed~$r$.

\begin{definition}
  The \emph{radius-$r$ cop-width} of $G$, denoted $\copw_r(G)$, is the least number $k$  such that the cops win the Cops and Robber game with radius $r$ and width~$k$.  
\end{definition}

Note that there is a graph parameter called the \emph{cop-number} of a graph~\cite{cop-number},
which is equal to the number of cops needed to catch the robber in a game where the cops and the robber move at speed one in each turn. It  could seem that $\copw_1(G)$ is upper bounded by the cop-number of $G$.
 However, there is a crucial difference with our notion:
the cops do not announce their moves in advance. So for instance, 
every graph with a universal vertex (a vertex adjacent to all other vertices) has cop-number equal to one (such graphs are called \emph{cop-win graphs}). On the other hand, $\copw_1(G)$ can be arbitrarily large on such graphs. Indeed, it is easy to see that for each $r\in\N\cup\set\infty$, the parameter
$\copw_r(G)$ is monotone with respect to the subgraph relation:
if $H$ is a subgraph of $G$ then $\copw_r(H)\le\copw_r(G)$.

\begin{example}\label{ex:bd-deg}
  Let $G$ be a graph with maximum degree $d$.
  Then $\copw_r(G)< d^{r+1}$, for all $r\in\N$.
  To see this, consider the following strategy for the cops: if the robber is initially placed on a vertex $v$, direct the cops to all (less than $d^{r+1})$ vertices that are at distance at mo  st $r$ from $v$.
  Before the cops land at those locations, the robber traverses a path of length at most $v$ to one of those vertices, and is caught by the landing cop. 
\end{example}
\begin{example}\label{ex:trees}
Let $T$ be a rooted tree. Then $G$ has tree-width at most one,
so by the Seymour-Thomas result, $\copw_\infty(G)\le 2$.
The strategy with two cops is as follows.
In the first round, land one cop on the root of the tree, leaving the other cop in his helicopter.
In each subsequent round, direct the cop that is further from the robber, to the child of the other cop's position
which is closest to the robber's position. With this strategy, in round $i$ the robber will be at distance at least $i$ from the root, so they will be caught after at most as many rounds as the height of $T$.
\end{example}

We start with comparing the most fundamental parameter of Sparsity theory, degeneracy, with the cop-width parameter for radius $1$.
\begin{theorem}\label{thm:degeneracy}
  For every graph $G$,
  \[\copw_1(G)=\deg(G)+1.\]
\end{theorem}

The simple proof of this fact relies on the fundamental duality result 
concerning degeneracy: for every $d\in\N$, 
every graph $G$ is either $d$-degenerate, or it has a subgraph $H$ in which every vertex 
has degree larger than $d$.
This duality theorem is nothing else than a min-max theorem 
for the radius-$1$ Cops and Robber game, quite analogous to the min-max theorem in the eponymous paper of Seymour and Thomas.
Indeed, a $d$-degeneracy order can now be viewed as a compact representation of a winning strategy for the cops involving $d+1$ cops, in the Cops and Robber game with radius $1$:
when robber is on a vertex $v$, place $d+1$ cops 
on $v$ and the neighbors of $v$ before $v$. Then the robber needs to move rightwards in the order, and eventually loses. This proves $\copw_1(G)\le\deg(G)+1$.
Dually, a subgraph $H$ of $G$ whose all vertices have degree larger than $d$, can be seen as a haven for the robber: they can forever evade $d$-cops by always moving to an unoccupied vertex of $H$ (or remaining in place). This proves $\copw_1(G)\ge \deg(G)+1$, thus proving Theorem~\ref{thm:degeneracy}. 

Next we observe that for higher radii $r$, the parameter 
$\copw_r(G)$ is closely related to the generalized coloring numbers: the {weak coloring} number $\wcol_r(G)$ and the {admissibility} numbers $\adm_r(G)$
(see Section~\ref{sec:prelim-sparsity}).
Recall (see Fact~\ref{fact:wcol-adm}) that the two parameters are functionally equivalent, and that $\adm_r(G)\ge \tilde\nabla_{r-1}(G)$.

We prove the following:
\begin{theorem}\label{thm:adm-wcol}
  For $r\in\N$,
  \[\adm_r(G)+1\le \copw_r(G)\le \wcol_{2r}(G)+1.\]     
\end{theorem}

In particular, by Fact~\ref{fact:be-wcol}, this gives the first, arguably very natural, characterization of classes with bounded expansion, in terms of a game:
\begin{corollary}\label{cor:copw}
A graph class $\CC$ has bounded expansion if and only if for every $r\in\N$ we have
that ${\copw_r(\CC)<\infty}$.
\end{corollary}
Thus, all the classes $\CC$ from Example~\ref{ex:BE-classes} satisfy $\copw_r(\CC)<\infty$, for all $r\in\N$.
Similarly, by Fact~\ref{fact:nd-wcol} we get a new characterization of nowhere dense classes.
\begin{corollary}\label{cor:copw-nd}
  A hereditary graph class $\CC$ is nowhere dense if and only if 
  for every $r\in\N$ and $\eps>0$ we have that 
  ${\copw_r(G)\le O_{r,\eps}(|G|^\eps)}$ for $G\in\CC$.
  \end{corollary}

\begin{proof}[Proof of Thm.~\ref{thm:adm-wcol}]
  Fix a total order on $G$ that minimises the maximum size of a weak reachability set with radius $2r$, so that every vertex $v\in V(G)$ weakly $2r$-reaches at most $\wcol_{2r}(G)$ vertices.
The following yields a winning strategy for the cops:
if the robber is at a vertex $v$, then  the cops are placed on $v$ and the vertices $w<v$ that are $2r$-weakly reachable from $v$. 
To see that this strategy is winning for the cops, 
consider the paths $\pi_1,\pi_2,\ldots$ in $G$, where $\pi_i$ is the path of length at most $r$ along which robber traversed from his $i$th position $v_i$ to his $(i+1)$-st position $v_{i+1}$. If $m_i$ denotes the $<$-minimal vertex of the path $\pi_i$, then we have that $m_{i+1}>m_i$. Otherwise,
$m_{i+1}$ is $2r$-weakly reachable from $v_i$ (as witnessed by the path $\pi_i;\pi_{i+1}$, truncated at $m_{i+1}$), 
so at the time the robber traversed the path $\pi_{i+1}$, the vertex $m_{i+1}$ that was occupied by a cop, which is impossible. Therefore, we have $m_1<m_2<\ldots$, so the cops win after at most $|G|$ rounds. This gives the upper bound in Theorem~\ref{thm:adm-wcol}.

For the lower bound, we use the well-known (and straightforward) min-max characterization of admissibility:
a graph $G$ has $r$-admissibility number at least $d$ 
 if there is a set of vertices $X\subset V(G)$ such that 
 for every vertex $v\in X$ there is a set of $d$ paths of length at most $r$ that start at $v$ and end in some vertex of~$X\setminus{v}$,
 such that any two paths only share $v$ as a common vertex.

 A set $X\subset V(G)$ that witnesses that  $\adm_r(G)\ge d$ can be used as a haven for the robber, to elude $d$ cops forever, similarly as in the case of degeneracy.
 If the robber is occupying a vertex $v$ in $X$ and the cops are moving to a set $S$ of at most $d$ new positions, then either $v\notin S$, or there is some path from $v$ to a vertex in $X-S$ of length $\le r$, and the robber moves along this path. This proves $\copw_r(G)\ge d+1$.
\end{proof}

We believe that the proof of Theorem~\ref{thm:adm-wcol}
sheds a new light on the fundamental notions of Sparsity theory:
the total order that appears in the definition of the weak coloring number
can be viewed as a compact representation of a (very particular) winning strategy for the cops in the Cops and Robber game,
whereas an obstruction to admissibility -- as a winning strategy for the robber.
 Therefore, 
the equivalence of the weak coloring numbers and the admissibility numbers can again be seen as a min-max theorem for the Cops and Robber game with finite radius.

\medskip
Note that Theorem~\ref{thm:adm-wcol} does not give an exact min-max 
theorem, as 
there is a gap between the upper and lower bounds.
We can get a family of parameters based on another variant of the  cops and 
robber game, which does admit an exact min-max theorem.
In this game, in each round first the cops move to some $k$ vertices of the graph, and then the robber 
moves via a path of length at most $r$, that does not run through a cop, and loses if no such path exists. See Appendix \ref{sec:copwidth'}
for more details.
This family of parameters again characterizes classes with bounded 
expansion,
and does admit an exact min-max theorem that exhibits a duality 
between total orders describing winning strategies for the cops,
and havens for the robber.
Those parameters essentially appear\footnote{The paper~\cite{richerby-thilikos-lazy-fugitive} considers a variant of the game in which the robber is \emph{lazy}, that is, does not move unless
a cop is placed at his location, whereas the~\cite{otherthilikos} considers a variant where the cops occupy edges instead of vertices, and the robber never remains put.} in the work~\cite{richerby-thilikos-lazy-fugitive, otherthilikos}.
However, the paper does not relate those notions with generalized coloring numbers, and with classes of bounded expansion.
Curiously, the limit version of those parameters, for radius $r=\infty$,
does not correspond to treewidth, but to a notion called \emph{$\infty$-admissibility}~\cite{dvorakTop}. Classes with bounded $\infty$-admissibility are characterized as clique-sums of graphs of almost bounded degree~\cite[Cor. 5 and Thm. 6]{dvorakTop}.
Despite the appealing properties of the parameters based on this variant of the cops and robber game, they seem to be less suited for our purposes, of generalizing to dense graphs.

\medskip
To summarize, our new cop-width parameters exactly characterize treewidth (for $r=\infty$), degeneracy (for $r=1$), classes of bounded expansion, and nowhere dense classes. This captures an appreciable fragment of the theory of sparse graphs, while offering a new perspective on the fundamental graph parameters used for measuring sparsity, and the dualities between them.
We have 
\begin{multline*}
  \deg(G)+1=\copw_1(G)\le \copw_2(G)\\\le\ldots \le \copw_\infty(G)=\tw(G)+1,
\end{multline*}
so for a class $\CC$ of graphs, if any of those parameters is bounded by a constant, then $\CC$ has bounded degeneracy by Theorem~\ref{thm:degeneracy}. 
So those parameters are only well suited to the study of sparse graphs: every class $\CC$ for which either of those parameters is bounded, is sparse, in the sense of having a bound on the edge density
$|E(G)|/|V(G)|$ for all graphs $G$ in the class.
 Moreover,  $\copw_r$ is monotone with respect to the subgraph relation: if $H$ is a subgraph of $G$,
then $\copw_r(H)\le \copw_r(G)$.

\section{Flip-width}\label{sec:flipwidth}
To lift the Cops and Robber game to the setting of dense graphs, 
we enhance the power of the cops.
Now, instead of placing cops on at most $k$ vertices of the graph,
which can be alternatively seen as removing at most $k$ vertices, or isolating them,
one player can perform \emph{flips} on subsets of the graph $G$.
For a fixed graph $G$, applying a \emph{flip} between a pair of sets of vertices $A,B\subset V(G)$ results in the graph
obtained from $G$ by inverting the adjacency between any pair of vertices $a,b$ with $a\in A$ and $b\in B$. For example, 
applying a flip between $V(G)$ and $V(G)$ in $G$ results in the complement of $G$.
And if $v$ is a vertex of $G$, then applying a flip between $\set v$ and the neighborhood $N(v)$ of $v$ has the same effect as \emph{isolating} $v$, that is, removing all the edges adjacent to $v$ in $G$. If $G$ is a graph and $\cal P$ is a partition of its vertex set, then call a graph $G'$ a \emph{$\cal P$-flip} of $G$
if $G'$ can be obtained 
from $G$ by performing a sequence of flips between pairs of parts $A,B\in \cal P$ (possibly with $A=B$). 
Since flips are involutive and commute with each other, such a sequence of flips can be specified by a set of at most $|\cal P|+1\choose 2$ unordered pairs of elements of $\cal P$.
Finally, call $G'$ a \emph{$k$-flip} of $G$, if $G'$ is a $\cal P$-flip of $G$, for some partition $\cal P$ of $V(G)$ with $|\cal P|\le k$.

\begin{remark}
There are many other, functionally equivalent,
ways to measure the complexity of a $k$-flip (also called a \emph{perturbation}) $G'$ of $G$.
For example, say that $G'$ is a \emph{$k$-sequential-flip} of $G$ if $G'$ is obtained from $G$ 
by applying a sequence of flips between $k$ pairs of arbitrary subsets of $V(G)$.
If $G'$ is a $k$-sequential-flip of $G$ then $G'$ is a $2^{2k}$-flip of $G$, and 
conversely, if $G'$ is a $k$-flip of $G$, then $G'$ is a $k+1\choose 2$-sequential-flip of $G$.
We could also require that flips are only applied to pairs of the form $(A,A)$;
this would lead to a functionally equivalent parameter, as flipping a pair $(A,B)$ 
can be obtained by flipping three pairs: $(A\cup B,A\cup B)$, $(A,A)$ and $(B,B)$.
Other, functionally equivalent measures 
of the complexity of a flip $G'$ of a graph $G$
can be defined by considering the graph $G'\triangle G$ with vertices $V(G)$ and edges $E(G')\triangle E(G)$.
Note that $G'$ is a $k$-flip of $G$ if and only if $G'\triangle G$ is a $k$-flip of the edgeless graph on $V(G)$. This is equivalent to $G\triangle G'$ 
having \emph{neighborhood diversity}~\cite{lampis} $k$.
The rank of the adjacency matrix of $G\triangle G'$ over a fixed finite field, see~\cite{nguyenoum,dingkotlov} leads to a further, functionally equivalent complexity measure of a flip.
\end{remark}


\paragraph{Flipper game}
We now come to the central notions of this paper.
The \emph{flipper game} with radius $r\in\N\cup\set{\infty}$ 
and width $k\in\N$, $k\ge 1$, is played on a graph $G$ by two players, \emph{flipper} and \emph{runner}.
In each round $i$ of the game,  
 a $k$-flip $G_i$ of $G$ is  declared by the flipper, and the new position $v_i\in V(G)$ is selected by the runner, as follows. Initially, $G_0=G$ and $v_0$ is a vertex of $G$ chosen by the runner. In round $i>0$,
the flipper announces a new $k$-flip $G_i$ of $G$, that will be put into effect momentarily.
The runner, knowing $G_i$, moves to a new vertex $v_i$ by following a path of length at most $r$ from $v_{i-1}$ to $v_i$ in the \emph{previous} graph $G_{i-1}$. The game terminates when the runner is trapped, that is,  when $v_i$ is isolated in $G_{i}$.

\begin{definition}Fix $r\in\N\cup\set\infty$.
  The \emph{radius-$r$ flip-width} of a graph $G$,
  denoted $\fw_r(G)$,
  is the smallest number $k\in\N$ such that the flipper has a winning strategy in the flipper game of radius $r$ and width $k$ on $G$.
\end{definition}

\begin{definition}
  A class $\CC$ of graphs has \emph{bounded flip-width}
  if $\fw_r(\CC)<\infty$ for every $r\in\N$.
  More explicitly: for every radius $r\in\N$ there is some~$c_r\in\N$ such that $\fw_r(G)<c_r$ for all $G\in \CC$.     
\end{definition}

\begin{remark}
The cop-width parameters considered in Section~\ref{sec:copwidth} are functionally equivalent
(more precisely, each parameter can be bounded from above by a linear function of the other) to parameters defined 
by a variant of the flipper game, call it the \emph{isolation game},
which is played as the flipper game, but each graph $G_i$ announced by the flipper/cops is obtained from $G$ by isolating at most $k$ vertices in $G$
(see \focs[\cite{flip-width-arxiv}]{Lemma~\ref{lem:cw-iw}}).
The difference between the isolation game and the Cops and Robber game 
is that in the Cops and Robber game, the runner/robber can move through 
a vertex from which a cop has just departed by a helicopter, 
while in the isolation game, they cannot.
\end{remark}

We now argue that the flip-width parameters are bounded in terms of the corresponding cop-width parameters.
\begin{lemma}For every $r\in\N\cup\set{\infty}$ and graph $G$, we have
  \begin{align}\label{eq:fw-cw-trivial}
    \fw_r(G)\le \copw_r(G)+2^{\copw_r(G)}.  
  \end{align}
\end{lemma}
\begin{proof}
The main observation is that 
isolating a set $S$ of at most $k$ vertices in $G$ can be achieved by performing a $(k+2^k)$-flip: consider the partition $\cal P_S$ that partitions $S$ into singletons and $V(G)-S$ according to the neighborhood in $S$, and flip $\set s$ with every class of the partition that is complete to $\set s$. Note that $|\cal P_S|\le k+2^k$.
Now, if the cops have a winning strategy in the Cops and Robber game of radius $r$ and width $k$ on a graph $G$,
we can use this strategy in the flipper game of radius $r$ and width $k+2^k$,
as follows:
whenever the cops announce a new set $S$ of positions of the cops in the Cops and Robber game, in the flipper game, the flipper announces the graph $G'$ with the vertices in $S$ 
isolated, which is a $(k+2^k)$-flip of $G$. It is easy to verify that if the cops win in the Cops and Robber game, then, playing according to the above strategy, the flipper also wins in the flipper game. Inequality~\eqref{eq:fw-cw-trivial} follows.
\end{proof}

The following gives an improved bound. Note that 
if $\copw_1(G)\le t$ then $G$ excludes $K_{t,t}$ as a subgraph, by Theorem~\ref{thm:degeneracy} and the fact that $K_{t,t}$ is not $(t-1)$-degenerate. From this, 
one can bound the size of the partition $\cal P_S$ considered above by $k^{t}$ (see \focs[\cite{flip-width-arxiv}]{Lemma~\ref{lem:Ktt-complexity}}), and obtain:
\begin{theorem}\label{thm:copw-fw}
  Fix $r\in\N\cup\set\infty$.
  Let $G$ be a graph and let $t$ be the smallest number such that $G$ excludes $K_{t,t}$ as a subgraph; in particular $t\le \copw_1(G)\le \copw_r(G)$.
  If $t\ge 3$ then \[\fw_r(G)\le\copw_r(G)^t,\]
  and if $t=2$ then 
  \[\fw_r(G)\le O(\copw_r(G)^t).\]
  In particular, every class $\CC$ with bounded expansion has bounded flip-width.
\end{theorem}

Whereas the cop-width parameters are monotone with respect to the 
subgraph relation, 
the flip-width parameters are monotone with respect to the induced subgraph relation. This is expressed by the following, immediate lemma.
\begin{lemma}\label{lem:hereditary}
Fix $r\in\N\cup\set\infty$.
If $H$ is an induced subgraph of $G$, then $\fw_r(H)\le \fw_r(G)$.  
In particular, a class $\CC$ has bounded flip-width 
if and only if its hereditary closure has bounded flip-width.
\end{lemma}

\subsection{Examples}\label{sec:examples}
We start by giving some example classes of bounded flip-width.
By Theorem~\ref{thm:copw-fw}, every class of bounded expansion has bounded flip-width, 
see Example~\ref{ex:BE-classes} for some specific classes.
Unlike cop-width, flip-width is not limited to sparse graphs, and is geared towards the  
study of dense graphs. 
\begin{example}\label{ex:complement}
  If $\bar G$ is the complement of $G$ and $r\in\N\cup\set\infty$, then 
  $\fw_r(\bar G)= \fw_r(G)$, since 
  any $k$-flip of $G$ is also a $k$-flip of $\bar G$ (just complement the set of flipped pairs).    
  Therefore, if $\CC$ is a class with bounded expansion,
  then the class ${\,\bar {\!\CC}}:=\setof{\bar G}{G\in \CC}$, has bounded flip-width, and if $\CC$ has bounded treewidth, then 
  $\fw_\infty({\,\bar {\!\CC}})=\fw_\infty(\CC)<\infty$.
  In particular, as every edgeless graph $G$ has $\fw_\infty(G)=1$, it follows that 
  every clique $G$ also has $\fw_\infty(G)=1$.
\end{example}

\begin{example}\label{ex:half-graphs}
  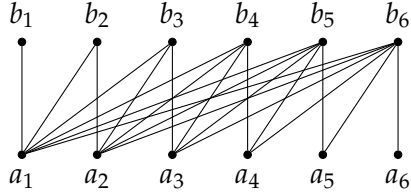
\begin{figure}[t!]
    \centering

      \begin{tikzpicture}
       
    \tikzstyle{vertex} = [draw, circle, fill, inner sep=1pt]

      \foreach \i in {1,2,3,4,5,6} {
        \node[vertex] (a\i) at (\i,0) [label=below:$a_\i$] {};
        \node[vertex] (b\i) at (\i,1.5) [label=above:$b_\i$] {};
      }
      \foreach \i/\j in {1/1,2/2,3/3,4/4,5/5,6/6,1/2,1/3,1/4,1/5,1/6,2/3,2/4,2/5,2/6,3/4,3/5,3/6,4/5,4/6,5/6} {
        \draw (a\i) -- (b\j);
      }
      \end{tikzpicture}
    \caption{A half-graph  of order $6$.}
    \label{fig:half-graph}
  \end{figure}
  Consider the \emph{half-graph} $H_n$
  of order $n$, as depicted in Figure~\ref{fig:half-graph}.  
  We show that $\fw_\infty(H_n)\le 4$.
  Observe that applying a flip between $\set{a_1,\ldots,a_{i-1}}$ and $\set{b_{i},\ldots,b_{n}}$ breaks the half-graph into two connected components (each being a half-graph).
  Additionally flipping $\set{a_{i}}$ and $\set{b_{i},\ldots,b_n}$ makes $a_{i}$ and $b_{i}$ isolated.
  The strategy of the flipper is to perform the above two flips in the $i$th round, thus pushing the runner rightwards in each round.
  To perform those flips, partition $V(H_n)$ into four parts: $\set{1,\ldots,a_{i-1}},\set {a_i},\set{b_i,\ldots,b_n}$, and the rest.
\end{example}
\begin{example}\label{ex:comp-trees}
The comparability graph $G$ of a rooted tree $T$
  is the graph with vertices $V(T)$, where two vertices are adjacent if and only if 
  one is an ancestor of the other in $T$.
  Generalizing half-graphs, those graphs also have $\fw_\infty(G)\le 4$.
The strategy of the flipper is, in round $i$, to 
consider  the node $u_i$ at depth $i$ in the tree, that is the ancestor of the current position of the runner, and to isolate $u_i$, and remove all edges between the descendants of $u_i$ and the ancestors of $u_i$.
This can be done by partitioning $V(G)$ into four parts: $\set{u_i}$,
the ancestors of $u_i$, the descendants of $u_i$, and the rest.
\end{example}

\begin{example}\label{ex:disjoint}
  Fix $r\in\N\cup\set\infty$. If $G_1,\ldots,G_m$ are graphs and $G$ is their disjoint union, 
  then the following inequality holds:
  \[\fw_r(G)\le\max_{1\le i\le m}(\fw_r(G_i))+1.\]
  The $+1$ comes from the fact that a partition $\cal P$ of $V(G_i)$ into $k$ parts induces a partition of $V(G)$ into $k+1$ parts, namely 
  the $k$ parts of $\cal P$, and $V(G)\setminus V(G_i)$.
  Thus, if the flipper is hiding in the graph $G_i$, the runner may translate a winning strategy on $G_i$ of width $k:=\fw_r(G_i)$,
  into a winning strategy on $G$ of width $k+1$.

  It follows that if $\CC$ has bounded flip-width, then the class of disjoint unions of graphs from $\CC$ has bounded flip-width.
\end{example}

\szfuture{how does $\fw_r$ behabe wrt strong/cartesian/tensor/lex products?}

 A \emph{modular partition} of a graph $G$ 
is a partition $\cal P$ of $V(G)$ such that 
any two distinct parts are homogeneous in $G$.
The quotient graph $G/\cal P$ has as vertices the parts of $\cal P$, and as edges pairs of distinct parts that are complete in $G$. 
An extension of the idea in Example~\ref{ex:disjoint} yields the following.

\begin{lemma}[\appmark]\label{lem:modular-partition}
  Let $G$ be a graph and $\cal P$ be its modular partition.
  Then 
  \[\fw_r(G)\le \max\left(\fw_r(G/\cal P),\max_{A\in\cal P}{\fw_r(G[A])+2}\right).\]
\end{lemma}
The strategy for the flipper on $G$ first follows the strategy on $G/\cal P$, where each $k$-flip of $G/\cal P$ is lifted naturally to a $k$-flip of $G$.
Once a part $A\in \cal P$ is isolated in the game on $G/\cal P$,  the strategy on $G[A]$ is used. The $+2$ in the statement is due to the fact that every partition $\cal Q$ of $A$ into $k$ parts
 induces a partition of $G$ into $k+2$ parts:
 the $k$ parts of $\cal Q$, the (common) set of neighbors of 
 vertices in $A$ outside of $A$, and the rest. See Appendix \ref{app:modular-and-subst}
 for details.
 In particular, if $G$ is the lexicographic product of two graphs $H$ and $K$
 (obtained by blowing up each node of $H$ to a copy of $K$)
 then $\fw_r(G)\le \max(\fw_r(H),\fw_r(K)+2)$.



The \emph{substitution closure} of a class of graphs $\CC$ is the smallest 
class $\CC^*$ containing $\CC$ such that
if $G$ is a graph with a modular partition $\cal P$ into modules $A$ satisfying $G[A]\in \CC^*$, and $G/\cal P\in \CC$, then $G\in \CC^*$.
Intuitively, a graph in $\CC^*$ can be obtained from a single vertex by repeatedly blowing up vertices to graphs from $\CC$.

Using similar ideas as in Lemma~\ref{lem:modular-partition}, we prove Lemma~\ref{lem:substitution}.
\begin{lemma}[\appmark]\label{lem:substitution}
  For every $r\in\N\cup\set{\infty}$ 
  and graph class $\CC$, we have 
  \[\fw_r(\CC^*)\le \fw_r(\CC)+2.\]
  In particular, if $\CC$ has bounded flip-width, then $\CC^*$ has bounded flip-width.
\end{lemma}

\begin{example}\label{ex:subst-cubic}
  The substitution closure $\CC^*$ of the class $\CC$ of subcubic graphs (graphs with maximum degree $3$) has bounded flip-width. This follows from Example~\ref{ex:bd-deg}, Theorem~\ref{thm:copw-fw}, and Lemma~\ref{lem:substitution}.
  This class not among the classes 
  that were previously known to be tame:
  $\CC^*$ has unbounded twin-width as already $\CC$  has unbounded twin-width \cite{tww2}, and $\CC^*$ is not edge-stable, 
  in particular it is not monadically stable nor structurally nowhere dense
  (see Section~\ref{sec:subpoly} for definitions).
\end{example}

More examples are given in the following sections:
they include classes of bounded clique-width 
(see Section~\ref{sec:cw}),
classes of bounded twin-width (see Section~\ref{sec:tww}), and 
interpretations of classes of bounded expansion (see Section~\ref{sec:transductions}).

\subsection{Hideouts}\label{sec:hideouts}
We have already seen several examples of 
classes of bounded flip-width. To give examples 
of classes of unbounded flip-width, we need a tool for proving lower bounds.
We now introduce the notion of a hideout, which is a set of vertices allowing the runner to evade the flipper indefinitely, thus allowing to prove lower bounds on $\fw_r(G)$.
In Section~\ref{sec:basic-comb} we use this notion to
 prove some combinatorial properties 
of graphs with bounded flip-width. In particular, it will follow easily that every graph of radius-one flip-width at most $k$ has a pair of vertices whose neighborhoods differ in at most $2k$ vertices.

Although we will use hideouts on several occasions to prove 
lower bounds on flip-width, we do not know whether the existence 
of hideouts is a necessary condition for having large flip-width.
This is stated as Question~\ref{q:hideouts} in Section~\ref{sec:discussion}.

 \begin{definition}\label{def:hideout}
  Fix $k,d\ge 1$ and $r\in\N\cup\set{\infty}$.
A $(r,k,d)$-\emph{hideout} in a graph $G$ 
is a set of vertices $U\subset V(G)$ with $|U|>d$, satisfying the following property. For every $k$-flip $G'$ of $G$,
\begin{align}\label{eq:hideout}
  |\left\{v\in U: |B^r_{G'}(v)\cap U|\le d\right\}|\le d,  
\end{align}
that is, there are 
at most $d$ vertices $v\in U$ such that
there are at most $d$ vertices $u\in U$ that 
 are connected with $v$ by a path of length at most $r$ in $G'$.
\end{definition}

We will show that in the flipper game with radius $r$ and width $k$, a runner can hide infinitely in a $(r,k,d)$-hideout.
Intuitively, when a $k$-flip $G'$ of $G$ is announced,
 the runner will want to avoid all vertices $v\in U$ with $|B^r_{G'}(v)\cap U|\le d$.
The condition in a hideout guarantees that there are at most $d$ such vertices.
This will allow the runner to always move  to some vertex with $|B^r_{G'}(v)\cap U|> d$.

\newcommand{\prv}{\textit{P}}
\newcommand{\nxt}{\textit{N}}

\begin{lemma}\label{lem:hideouts}
  Fix $k\ge 1$ and $r\in\N\cup\set{\infty}$.
  If a graph $G$ has a $(r,k,d)$-hideout $U$ for some $d\ge 1$, then ${\fw_r(G)>k}$.
\end{lemma}
\begin{proof}Let $U\subset V(G)$ be a $(r,k,d)$-hideout.
  We describe a strategy for the runner
  in the flipper game on $G$ with radius $r$ and width $k$, which allows  to elude the flipper indefinitely.
  The strategy is as follows:
  when the flipper announces a $k$-flip $G'$ of $G$,
   the runner moves to some vertex $v\in U$ such that  $|B^r_{G'}(v)\cap U|> d$.
  In the first move, pick any $v\in U$ with $|B^r_{G}(v)\cap U|> d$.
  Such a vertex exists, by \eqref{eq:hideout} applied to $G'=G$, since $|U|>d$.
  
We show it is always possible to make a move as described in the strategy.
 Suppose at some point in the game, 
the current position $v$ of the runner
is such that 
\begin{align}\label{eq:inv-ineq}
  |B^r_{\prv}(v)\cap U|&> d.
\end{align}
where $\prv$ is the previous $k$-flip of $G$ announced by the flipper (in the first round, $\prv=G$), and that the flipper now announces the next $k$-flip $\nxt$ of $G$. Since $U$ is a $(r,k,d)$-hideout, 
the set $X\subset U$ of vertices $w\in U$ 
such that $|B^r_{\nxt}(w)\cap U|\le d$
satisfies $|X|\le d$. 
By~\eqref{eq:inv-ineq},
 $B^r_{\prv}(v)$ contains at least one vertex $v'\in U-X$.
The runner moves from $v$ to $v'$ 
along a path of length at most $r$ in $\prv$.
As $v'\in U-X$, the invariant is maintained.
Therefore, playing according to this strategy, the runner can elude the flipper indefinitely, so $\fw_r(G)>k$.
\end{proof}

\subsection{Flip-width with infinite radius}\label{sec:cw}
As a simple case study, we first analyze the parameter $\fw_\infty$. Recall that its sparse analogue, 
$\copw_\infty$, corresponds to treewidth.
We show that $\fw_\infty$ is functionally equivalent to the clique-width  and rank-width parameters.
Those parameters extend treewidth to the setting of dense graphs, and are recalled later below.
To the best of our knowledge, our result is the first characterization of graph classes of bounded clique-width, in terms of a game\anonym{\footnote{The radius-$\infty$ flipper game arose in a private discussion in 2018 with Michał Pilipczuk.}}, analogous to the game characterization of treewidth.
\begin{theorem}[\appmark]\label{thm:cw}
  For every graph $G$,
  we have 
  \[\rw(G)\le 3\fw_\infty(G)+1\le O(2^{\rw(G)}).\]
  In particular,
  a graph class $\CC$ has bounded rank-width if and only if 
  ${\fw_\infty(\CC)<\infty}$.
\end{theorem}

The \emph{rank-width} parameter is functionally equivalent to clique-width, and is often more convenient to work with. We will only use rank-width.  Rank-width and clique-width are functionally equivalent, as expressed below \cite[Prop. 6.3]{OumSeymour-approximating}:
\[\rw(G)\le \cw(G)< 2^{\rw(G)+1}.\]

A graph $G$ 
has rank-width at most $k$ 
if there is a tree $T$ whose leaves 
are the vertices of $G$,
and inner nodes have degree at most $3$,
such that for every edge $e$ of the tree,
the bipartition $A\uplus B$ of the leaves of $T$
into the leaves on either side of $e$,
has \emph{cut-rank} at most $k$.
The cut-rank of a bipartition $A\uplus B$
of the vertex set of a graph $G$, denoted 
$\rk_G(A,B)$, is defined as the rank, over the two-element field, of the $0,1$-matrix with rows $A$, columns $B$, where the entry at row $a\in A$ and column $b\in B$ is $1$ if $ab\in E(G)$ and $0$ otherwise.


\medskip
The upper bound in Theorem~\ref{thm:cw} is a generalization of the bound in Example~\ref{ex:comp-trees}, where it is shown that comparability graphs of trees have $\fw_\infty(G)\le 4$.
We briefly sketch the argument now.

Note that a $0,1$-matrix of rank at most $k$
over the two-element field has at most $2^k$ distinct rows and at most $2^k$ distinct columns.
It follows  that if $G$ is a graph and 
$V(G)=A\uplus B$ is a bipartition of its vertex sets 
with cut-rank $\rk_G(A,B)\le k$,
then $A$ and $B$ can be partitioned as 
$A=A_1\uplus\cdots\uplus A_p$ and $B=B_1\uplus\cdots \uplus B_q$ with $p,q\le 2^k$, so that 
 $A_i$ and $B_j$ are complete in $G$.
 This implies that there is a $2^{k+1}$-flip $G'$ of $G$ which has no edges with one endpoint in $A$ and one endpoint in $B$.

Therefore, if $G$ is a graph of rank-width $k$,
 then there is a subcubic tree $T$ with leaves $V(G)$
 such that for every edge $e$ of $T$,
if $V(G)=A\uplus B$ is the bi-partition induced by $e$
(into the leaves on either side of $e$),
then there is a $2^{k+1}$-flip $G'$ of $G$ 
which has no edges with one endpoint in $A$ and one endpoint in $B$. Moreover, 
for every inner node $v$ of $T$ (of degree at most three), there is a $O(2^k)$-flip $G'$ of $G$ 
such that for any two vertices $a,b\in V(G)$, if $a$ and $b$ are in the same connected component of $G'$, then $a$ and $b$ are connected in $T$ by a path that avoids $v$.
This flip can be used by the flipper  in their winning strategy, to restrain the runner to the leaves of smaller and smaller subtrees of $T$,
similarly as in Example~\ref{ex:comp-trees}. See Appendix \ref{app:cw} for details.

\medskip

The lower bound relies
on the following result characterizing obstructions to having small rank-width.

   A set $U$ of vertices of $G$ is \emph{well-linked} if for every bipartition $A\uplus B$ of $V(G)$,
   the cut-rank of $A\uplus B$ satisfies  $\rk_G(A,B)\ge \min(|A\cap U|,|B\cap U|)$.
   Oum and Seymour \cite[Theorem 5.2]{OumSeymour-approximating} prove the following:
  
\begin{fact}[\cite{OumSeymour-approximating}]\label{fact:well-linked}
  Every graph of rank-width greater than $k$ contains a well-linked set of size $k$.
\end{fact}


Using this, the lower bound in Theorem~\ref{thm:cw}  follows from the next lemma, which is proved in Appendix \ref{app:cw}.
The proof is due to Rose McCarty (private communication).
\begin{lemma}[\appmark]\label{lem:well-linked-hideout}
  Fix a graph $G$ and number $k\in\N$.
  Every well-linked set $U$ with $|U|>3k$ is a $(\infty,k,k)$-hideout. 
\end{lemma}
This implies the lower bound in Theorem~\ref{thm:cw} as follows. Suppose $\rw(G)>3k+1$ for some $k\in\N$.
By Fact~\ref{fact:well-linked}, $G$ contains a well-linked set $U$ of size $3k+1$.
By Lemma~\ref{lem:well-linked-hideout}, $U$
is a $(\infty,k,k)$-hideout.
By Lemma~\ref{lem:hideouts},
$\fw_\infty(G)>k$. By contrapositive, this shows that if $\fw_\infty(G)=k$ then $\rw(G)\le 3k+1=3\fw_\infty(G)+1$.

\subsection{Radius-one flip-width}\label{sec:basic-comb}
We move to the study of finite radii, which are our main focus, 
starting with the first parameter, $\fw_1$. We have seen that its sparse analogue, $\copw_1$, corresponds precisely to degeneracy (plus one), which is a very well-understood parameter, with many good algorithmic and combinatorial properties.
The parameter $\fw_1$ enjoys many useful combinatorial properties, 
relating it to near-twins, and to the VC-dimension.

\paragraph{Near-twins}
We prove a first combinatorial property of graphs with small $\fw_1(G)$, namely that such graphs have near-twins. This has several consequences.
Say that two vertices $u,v$ of a graph $G$ are \emph{$\delta$-near-twins} if $|N(u)\triangle N(v)|\le \delta$, where $\triangle$ denotes the symmetric difference.
We show that every graph $G$ with $\fw_1(G)\le k$ 
has a pair of $2k$-near-twins. More generally, we prove:
\begin{lemma}\label{lem:near-twins}
  Let $b,k\in \N$ and let $G$ be a graph with $\fw_1(G)\le k$ and $|G|>bk$.
  Then $G$ contains a set of at least $b+1$ vertices which are mutual $2 b k$-near-twins.
\end{lemma}
\begin{proof}
  Assume that $G$ has no set containing $b+1$ mutual $2bk$-near-twins. We prove that $V(G)$ 
  is a $(1,k,bk)$-hideout in $G$, which implies that $\fw_1(G)>k$ by Lemma~\ref{lem:hideouts}.

  Let $G'$ be a $k$-flip of $G$ and 
  let $B$ be the set of vertices of degree 
  at most $bk$ in $G'$.
  We show that $|B|\le bk$, proving that $V(G)$ is a $(1,k,bk)$-hideout in $G$.
  
Suppose that $|B|> bk$. 
Let $\cal P$ be a partition with $|\cal P|\le k$ such that  $G'$ is a $\cal P$-flip of $G$.
As $|B|>bk$, there is a set $B_0\subset B$
with $|B_0|>b$, such that $B_0$ is contained in one part of $\cal P$. Any two vertices of $B_0$ are $2bk$-near-twins in $G'$, as they both have degree at most $bk$ in $G'$. Since $B_0$ is contained in a single part of $\cal P$, it follows that any two vertices of $B_0$ are $2bk$-near-twins in $G$, too. But $|B_0|>b$, so this contradicts the assumption. Hence, $|B|\le bk$.
\end{proof}

A bipartite variant of Lemma~\ref{lem:near-twins},
with a very similar proof, is as follows. Recall
that if  $X,Y\subset V(G)$ are two sets of vertices 
of a graph $G$ then $G[X,Y]$ denotes the bipartite graph semi-induced by $X$ and $Y$ in $G$.

\begin{lemma}[\appmark]\label{lem:near-twins-two-sets}
  Let $b,k\in\N$ and let $G$ be a graph with $\fw_1(G)\le k$. Then for every two sets $X,Y\subset V(G)$
  with $|X|,|Y|>bk$,
  there is a subset of $X$
  or a subset of $Y$ consisting of $b+1$ mutual $2bk$-near-twins in $G[X,Y]$.
\end{lemma}

Setting $b:=1$ in Lemma~\ref{lem:near-twins} we get:
\begin{corollary}\label{cor:near-twins}
  Let $G$ be a graph with $\fw_1(G)\le k$ and $|G|>k$.
  Then $G$ has a pair of $2k$-near-twins.
\end{corollary}

We can now verify that there exist graphs $G$ with arbitrarily large $\fw_1(G)$.
It is well known that there exist graphs 
of arbitrarily large girth and minimum degree\sz{cite}.
\begin{corollary}\label{cor:fw1-nontrivial}
  A graph $G$ with girth larger than $4$ and minimum degree larger than $k$
  has $\fw_1(G)> k$. Therefore, there  exist graphs $G$ with arbitrarily large $\fw_1(G)$.
\end{corollary}
\begin{proof}
Any two distinct vertices $u$ and $v$ have at most one common neighbor, so $|N(u)\triangle N(v)|> 2k$, so $G$ has no pair of $2k$-near-twins, hence $\fw_1(G)>k$ by Corollary~\ref{cor:near-twins}.
\end{proof}

\paragraph{VC dimension}

We now show that the VC-dimension of a graph $G$ is bounded 
in terms of $\fw_1(G)$.
Set systems and graphs of bounded VC-dimension have many 
useful properties, and we use one of them later for studying classes of bounded flip-width.
We also consider a related parameter, called \emph{2VC-dimension} \cite{bousquet-thomassee}, and denoted $\TVCdim(G)$.
This is the maximal size of a set $X\subset V(G)$
such that for every two distinct $a,b\in X$ there is a vertex $c\in V(G)$ with $N_G(c)\cap X=\set{a,b}$.
Clearly, $\VCdim(G)\le\TVCdim(G)$. We prove the following.

\begin{theorem}[\appmark]\label{thm:vcdim}For every graph $G$ we have 
  \begin{align}
    \VCdim(G)&\le 8\fw_1(G),\label{eq:vcdim}\\
    \TVCdim(G)&\le 8\fw_2(G)+2  \label{eq:2vcdim}
  \end{align}
\end{theorem}
\szfuture{How tight are the bounds?}
Note that $\TVCdim$ cannot be bounded in terms of $\fw_1$, only in terms of $\fw_2$,
as witnessed by $1$-subdivided cliques, which
are $2$-degenerate, and hence (by Theorem~\ref{thm:copw-fw} and Theorem~\ref{thm:degeneracy}), have bounded $\fw_1$, and clearly have unbounded $\fw_2$.
Furthermore, graphs  of girth larger than $4$ have VC-dimension at most two, but have arbitrarily large $\fw_1$ by Corollary~\ref{cor:fw1-nontrivial}. Hence, $\fw_1$ is not bounded in terms of $\VCdim(G)$.

Inequality~\eqref{eq:vcdim} in Theorem~\ref{thm:vcdim} follows from Lemma~\ref{lem:near-twins-two-sets} (for $b=1$) and the following. 
\begin{lemma}[\appmark]\label{lem:vector-space-twins}
  Let $G$ be a graph with $\VCdim(G)\ge 2^m$, for some $m$.  Then 
  there are two sets $X,Y$ such that the bipartite graph $G[X,Y]$ 
   contains no pair of $(2^{m-1}-1)$-near-twins in either of the parts $X, Y$.
\end{lemma}
In the proof, the sets $X$ and $Y$ are two copies of the $m$-dimensional vector space over the two-element field, with edges connecting vectors with a nonzero dot product.
See Appendix \ref{app:fw-and-VC} for details.

We prove inequality \eqref{eq:2vcdim}
in Corollary~\ref{cor:2vcdim} later.
From Lemma~\ref{lem:vector-space-twins} and Lemma~\ref{lem:near-twins-two-sets} (for $b=1$) we conclude the following.
\begin{corollary}\label{cor:large-vc-subgraph}
  Let $G$ be a graph with $\VCdim(G)\ge d$.
  Then $G$ contains an induced subgraph $H$ with $O(d)$ vertices 
  and with $\fw_1(H)\ge d/8$.  
\end{corollary}
\begin{corollary}\label{cor:bounded-vc}
  If $\CC$ is a hereditary class of graphs 
  such that $\fw_1(G)\le o(|G|)$ for $G\in\CC$,
  then $\VCdim(\CC)<\infty$.
\end{corollary}
\begin{proof}
  If $\VCdim(\CC)=\infty$ then for every $d$ there is a
  graph $G\in\CC$ with $\VCdim(G)\ge d$,
  and by Corollary~\ref{cor:large-vc-subgraph}
  there is $H\in\CC$ with $O(d)$ vertices and $\fw_1(H)=\Omega(d)$.
  Since this holds for all $d\in\N$, 
  it cannot be that $\fw_1(G)\le o(|G|)$ for all $G\in\CC$.
\end{proof}

\subsection{Relationship of radius-one flip-width with other graph parameters}\label{sec:fw1-functionality}
In this section, we compare radius-one flip-width with existing graph parameters.
First, we mention that in Section~\ref{sec:fw-deg}
we show that for  graphs $G$
that exclude a fixed biclique $K_{t,t}$ as a subgraph,
we have that $\fw_1(G)$ is functionally equivalent 
to the degeneracy of $G$. However, unlike degeneracy,
$\fw_1$ is bounded for some graphs that contain arbitrarily large cliques or bicliques. 
Thus, $\fw_1(G)$ may be thought of as a dense extension of degeneracy. Two other graph parameters -- \emph{symmetric difference} and \emph{functionality} -- which could serve a similar role, are discussed below.

For a graph $G$, define the \emph{symmetric difference} of $G$, denoted
$\text{sd}(G)$, as the least number $k\in\N$,
such that every induced subgraph $H$ of $G$ 
contains a pair $u,v$ of $k$-near-twins. 
Classes $\CC$ with $\text{sd}(\CC)<\infty$ 
are called \emph{classes with bounded symmetric difference}.
For example, if $G$ has degeneracy $k$,
then $\text{sd}(G)\le 2k$. In particular, 
every graph class with bounded degeneracy has bounded symmetric difference.

Since $\fw_1$ is monotone under taking induced subgraphs, 
 Corollary~\ref{cor:near-twins} implies that 
 $\text{sd}(G)\le 2\fw_1(G)$ for every graph $G$.
 Hence, every graph class $\CC$ with $\fw_1(\CC)<\infty$ has bounded symmetric difference. We do not know whether the converse implication holds.

 A more general graph parameter, called \emph{graph functionality} and denoted $\text{fun}(G)$, is defined in \cite{functionality}.
 The following is a rephrasing of the original definition.

Say that a vertex $v\in V(G)$ is a \emph{function} 
of a set $S\subset V(G)$ of vertices, with $s\notin S$, if for every $w\in V(G)-(S\cup \set{v})$,
the adjacency of $w$ and $v$ in $G$ depends only 
on $N(w)\cap S$. For instance, 
if $v$ is a $k$-near-twin of $u$,
then $v$ is a function of 
$\set{u}\cup (N(v)\triangle N(u))-\set v$.
Define $\text{fun}(G)$ as the least number $k$ 
such that every induced subgraph $H$ of $G$ 
has a vertex $v$ that is a function of at most $k$ 
other vertices of $H$. 

Clearly, $\text{fun}(G)\le \text{sd}(G)+1$ holds for every graph $G$, so classes of bounded symmetric difference have bounded functionality.
The converse implication is false, as witnessed by the class of permutation graphs \cite[Sec. 2.4]{functionality}, which have bounded functionality and unbounded symmetric difference.

Classes of bounded functionality have bounded VC-dimension, as shown in \cite[Thm. 8]{functionality}. 
However, no effective bound is known 
(see \cite[Open problem 4]{functionality}).

It follows from \cite[Thm. 2]{implicit-representations} that every hereditary graph class of bounded functionality is 
at most \emph{factorial} -- contains 
at most $2^{O(n\log n)}$ labelled graphs on $n$ vertices. 
In fact, every
class that is at most factorial has bounded VC-dimension (see \cite[Sec. 3]{functionality}).

\medskip
Figure~\ref{fig:diagram-fw1}  summarizes the relationships among the discussed properties of graph classes.

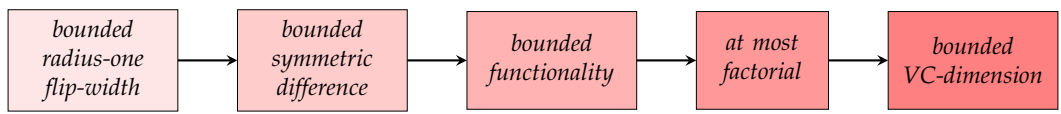
\begin{figure*}[h!]
  \begin{centering}
  \begin{tikzpicture}   
    \node[block, text width=2.cm, fill=red!10] (a') {bounded radius-one flip-width};
    \node[block,text width=2.cm, right=of a', xshift=-0.21cm,fill=red!20] (b') {bounded symmetric difference};
    \node[block, text width=2cm,right=of b',  xshift=-0.21cm,fill=red!30] (c') {bounded \mbox{functionality}};
    \node[block, text width=1.5cm,right=of c',  xshift=-0.21cm,fill=red!40] (d') {at most factorial};
    \node[block,text width=2cm,right=of d',xshift=-0.21cm,fill=red!50] (e') {bounded \mbox{VC-dimension}};

    \draw [arrow] (a') -- (b');
    \draw [arrow] (b') -- (c');
    \draw [arrow] (c') -- (d');
    \draw [arrow] (d') -- (e');
  \end{tikzpicture}
  \caption{Properties of graph classes and implications among them. 
  }\label{fig:diagram-fw1}
  \end{centering}
  \end{figure*}

\section{Flip-width in weakly sparse classes}\label{sec:wsparse}
We have seen in Theorem~\ref{thm:copw-fw}  radius-one flip-width is upper bounded in terms of degeneracy, and 
radius-$r$ flip-width is upper bounded in terms of  generalized coloring numbers.
In this section,
we provide bounds in the other direction,
in weakly sparse classes.
It follows that for weakly sparse classes,
having bounded flip-width is equivalent to having bounded expansion.

\subsection{Radius-one flip-width and degeneracy}\label{sec:fw-deg}
As $\fw_1$ is bounded in terms of $\copw_1$, which is equivalent to degeneracy by Theorem~\ref{thm:degeneracy},  it follows that every class of bounded degeneracy has bounded $\fw_1$. Clearly, every class of degeneracy bounded by $t$ is weakly sparse, as it excludes $K_{t+1,t+1}$ as a subgraph.
We show that for weakly sparse classes, bounded degeneracy is equivalent to having bounded $\fw_1$.

\begin{theorem}\label{thm:fw-deg}
If $G$ is a graph that avoids $K_{t,t}$ as a subgraph, then 
\begin{align}
  \deg(G)/(2t^2)&< \fw_1(G)\le (\deg(G)+1)^{t}.  
\end{align}
  As a consequence, if $\CC$ is a weakly sparse class of graphs 
then $\fw_1(\CC)<\infty$ if and only if $\CC$ has bounded degeneracy. 
\end{theorem}
\begin{proof}The second inequality is by Theorem~\ref{thm:copw-fw} and Theorem~\ref{thm:degeneracy}. To prove the first inequality, we show that $\deg(G)<2 kt^2$, where
 $k:=\fw_1(G)$.

  Recall that $G$ is $d$-degenerate if and only if every induced subgraph of $G$ contains a vertex of degree at most $d$.
Therefore, to prove $\deg(G)<2kt^2$
it is enough to show that 
$G$ contains some vertex of degree less than $2kt^2$
(since the same holds for every induced subgraph $H$ of $G$,
as $\fw_1(H)\le \fw_1(G)\le k$).

  Setting $b:=t-1$ in Lemma~\ref{lem:near-twins}
we get that $G$ contains a set $U$ of $t$ mutual $2k(t-1)$-near-twins.
Pick any $v\in U$. 
Then all $u\in U$ are $2k(t-1)$-near-twins of $v$, 
so $|N(v)\triangle N(u)|\le 2k(t-1)$ for all $u\in U$.

We show that $|N(v)|< 2k(t-1)^2+2t$. Suppose that $|N(v)|\ge 2k(t-1)^2+2t$.
Then the set $W:=\bigcap_{u\in U} N(u)$ has at least $2t$ elements, 
and so $W-U$ has at least $t$ elements.
As every $u\in U$ is adjacent to every $w\in W$ and $U$ and $W$ are disjoint, we have a copy of $K_{t,t}$ as a subgraph of $G$, a contradiction. Hence, $|N(v)|< 2k(t-1)^2+2t<2kt^2$.
\end{proof}




\subsection{Radius-$r$ flip-width and generalized coloring numbers}
\label{sec:adm}
We have seen that every weakly sparse class $\CC$ has bounded degeneracy if and only if $\fw_1(\CC)<\infty$.
It is known that a weakly sparse class 
$\CC$ has bounded clique-width  if and only if $\CC$ has bounded tree-width,
as clique-width and treewidth are functionally equivalent 
in weakly sparse classes~\cite{gurskiwanke}.
This proves the following.
\begin{corollary}\label{col:wsparse-treewidth}
A graph class $\CC$ has bounded treewidth if and only if 
$\CC$ is weakly sparse and 
 ${\fw_\infty(\CC)<\infty}$.
\end{corollary}

A general theme in structural graph theory is that a dense graph parameter  
is often functionally equivalent to its sparse counterpart in weakly sparse classes. 
We now show that for weakly sparse classes, 
 bounded flip-width is indeed equivalent to bounded expansion:
\begin{theorem}\label{thm:wsparse}A class $\CC$ has bounded expansion if and only if $\CC$ is weakly sparse and has bounded flip-width.
\end{theorem}

We prove a more precise result, from which Theorem~\ref{thm:wsparse} follows.
\begin{theorem}\label{thm:adm-fw}For every $r\ge 1$ and graph $G$ we have
 \[\tilde\nabla_{r-1}(G)\le O(r\cdot \fw_r(G)\cdot \deg(G))^{36}.\]
\end{theorem}
We first show how Theorem~\ref{thm:adm-fw} implies Theorem~\ref{thm:wsparse}.
\begin{proof}[Proof of Theorem~\ref{thm:wsparse}]
  For the forward direction, assume $\CC$ has bounded expansion. 
  Then $\CC$ has bounded flip-width by Theorem~\ref{thm:copw-fw}. Also, as remarked before Theorem~\ref{thm:copw-fw}, every class with bounded expansion is weakly sparse.

   Conversely, suppose $\CC$ is weakly sparse
   and has bounded flip-width.
   In particular, $\fw_1(\CC)<\infty$, and by Theorem~\ref{thm:fw-deg}, $\CC$ has degeneracy bounded by some constant $d$. 
   By Theorem~\ref{thm:adm-fw}, for all $r\ge 1$ we have that $\tilde\nabla_{r-1}(\CC)\le 
   O(r\cdot \fw_r(\CC)\cdot d)^{36}<\infty$.
   Hence, $\CC$ has bounded expansion.
\end{proof}

In the remainder of Section~\ref{sec:adm} we prove Theorem~\ref{thm:adm-fw}.
Our proof relies on a result of \Dvorak~\cite{dvorakInducedSubdivisions}, extending a result of K\"uhn and Osthus~\cite{kuhn-osthus},
which we now recall.

An \emph{exact $r$-subdivision} of a graph $G$
is the graph obtained by replacing every edge of $G$ 
by a path of length $r+1$.
If every edge is replaced by a path of length at most $r+1$,
the resulting graph is an \emph{${\le} r$-subdivision} of $G$.
For a graph $G$, let $\tilde\nabla^e_r(G)$ denote the maximum average degree of all graphs $H$ whose \emph{exact} $r$-subdivision is an \emph{induced} subgraph of $G$.

The following is \cite[Lemma 9]{dvorakInducedSubdivisions}.
\begin{lemma}
  \label{lem:dvorak}
  For every $r,k,d\ge 1$ there is a number $s=s(r,k,d)\le O(rdk)^{12}$
  such that for every graph $G$,
  if $\tilde\nabla_0(G)\le d$ and $\tilde\nabla_r^e(G)<k$,
  then $\tilde\nabla_r(G)<\tilde\nabla_{r-1}(G)+s$.
\end{lemma}
An easy induction on $r$ yields the following.
\begin{corollary}\label{cor:dvorak}
  For every $r,d\ge 1$ and $d$-degenerate graph $G$,
  \[\tilde\nabla_r(G)\le O(dr \cdot \sum_{t=1}^r \tilde\nabla_t^e(G))^{12}.\]
\end{corollary}

It is well known that every graph of average degree at least $d$ contains a subgraph with minimum degree at least $d/2$. The following proposition yields Theorem~\ref{thm:adm-fw}, using Corollary~\ref{cor:dvorak} and Fact~\ref{fact:adm-nabla}.

\begin{proposition}[\appmark]\label{prop:subdivisions}
  Fix $r\ge 2,k\ge 1$. 
  Let $G$ be the exact $(r-1)$-subdivision of some graph $H$ 
  with minimum degree at least $2rk$.
  Then $\fw_r(G)>k$.
\end{proposition}
Proposition~\ref{prop:subdivisions} is proved in Appendix~\ref{app:exact},
by showing that the vertices of $G$ that correspond to the vertices of $H$ 
(those of degree larger than $2$) form a $(r,k,k)$-hideout in $G$.
We now show how
 this proves Theorem~\ref{thm:adm-fw}.
\begin{proof}[Proof of Theorem~\ref{thm:adm-fw}]
  Denote $k=\fw_r(G)$.
  Since $\fw_r(G)$ is monotone in $r$, by Proposition~\ref{prop:subdivisions}, for all $t<r$, the graph $G$ does not contain an induced $t$-subdivision of a graph of minimum degree at least $2rk$.
  Hence, $\tilde\nabla^e_t(G)\le 4rk$ for all $t< r$.

  If $G$ is $d$-degenerate then by Corollary~\ref{cor:dvorak} we have:
  \[\tilde\nabla_{r-1}(G)\le
  O(dr\sum_{1\le t<r}\tilde\nabla_t^e(G))^{12}\le  O(dr^3k)^{12}\le O(drk)^{36},\]%
  as required.
\end{proof}
 Proposition~\ref{prop:subdivisions} easily implies the following corollary (see \focs[{\cite{flip-width-arxiv}}]{Appendix~\ref{app:2vc}} for a proof), 
yielding inequality~\eqref{eq:2vcdim} in Theorem~\ref{thm:vcdim}.
\begin{corollary}[\appmark]\label{cor:2vcdim}
  If $G$ is the exact $1$-subdivision of an $n$-clique,
  then $\fw_2(G)>(n-1)/4$. Furthermore, for every graph $G$, $\TVCdim(G)\le 8\fw_2(G)+2$.
\end{corollary}

\medskip

Hence, weakly sparse classes of bounded flip-width,
 being exactly the classes with bounded expansion, 
 are by know very well-understood and characterized in multiple ways. For instance, the model checking problem for first-order logic is fixed-parameter tractable for such classes, by the result of \Dvorak, \Kral, and Thomas~\cite{DvorakKT13-journal}.

 \section{Flip-width of ordered graphs and twin-width}\label{sec:tww}
 Let us return to dense graph classes, which are our main focus.
 As follows from Theorem~\ref{thm:cw}, radius-$\infty$ flip-width is functionally equivalent to clique-width, so classes of bounded clique-width are examples of dense graph classes of bounded flip-width.
In this section we show that classes of bounded twin-width have bounded flip-width,
but the converse does not hold.
To characterize twin-width in terms of flip-width, 
we study flip-width of graphs equipped with a total order.
We start with recalling the definition of twin-width.

 An \emph{uncontraction sequence} of a graph $G$ is a sequence $\cal P_1,\ldots,\cal P_n$ of partitions of $V(G)$
  that starts with the partition $\cal P_1$ with one part, ends with the partition of $V(G)$ into singletons,
 and such that every partition $\cal P_{i+1}$, for $i<n$, is obtained from the previous partition $\cal P_i$ by splitting one of the parts into two.
 The \emph{red graph} of a partition $\cal P$ is the graph whose nodes are the parts of $\cal P$, with \emph{red edges} connecting two distinct nodes $A,B\in\cal P$ if $A$ and $B$ are not homogeneous in $G$.
 A graph has twin-width at most $d$ if it has an uncontraction sequence such that at every time $i$, the red graph of the partition $\cal P_i$ has maximum degree at most $d$.

First, we prove that $\fw_r(G)$ is bounded in terms of $r$ and the twin-width $\tww(G)$ of~$G$. 
For a graph $G$, define the \emph{shatter function} of $G$, denoted $\pi_G\from\N\to\N$, as follows:
$$\pi_G(n):=\max_{X\subset V(G), |X|\le n}|\setof{N(v)\cap X}{v\in V(G)}|.$$

\begin{theorem}[\appmark]\label{thm:btww}Fix $r\in\N$. For every graph $G$ of twin-width $d$ we have:
  \[\fw_r(G)\le \pi_G(d^{O(r)})\le 2^d\cdot d^{O(r)}.\]
   In particular, every class of bounded twin-width has bounded flip-width.
\end{theorem}
In the proof, the flipper uses an uncontraction sequence in order to vanquish the runner by constraining them, in round $i$ to some ball of radius $r$ in the red-graph of the partition~$\cal P_{i}$. In round $i=n$, as every ball of radius $r$ in the red graph of $\cal P_n$ comprises a single vertex of $G$, the runner is trapped.
The second inequality in Theorem~\ref{thm:btww} 
follows immediately from the upper bound~\cite{tww-neighborhood-complexity} on $\pi_G(n)$ 
in a graph of twin-width $d$. See \focs[\cite{flip-width-arxiv}]{Appendix~\ref{app:tww}} for details.

\medskip

Therefore,  bounded twin-width implies bounded flip-width, but the converse does not hold:
 the class of subcubic graphs has bounded expansion, and hence bounded flip-width,
but does not have bounded twin-width~\cite{tww2}. So can twin-width be exactly characterized in terms of flip-width?

The twin-width parameter is defined not only for graphs, but also for structures equipped with one or more binary relations.
As argued in \cite{tww4},  twin-width may -- and perhaps even should -- be seen as a parameter of \emph{ordered} graphs, rather than graphs. 
An ordered graph $G=(V,E,<)$ is equipped with a (symmetric, irreflexive) edge relation $E$ and a total order relation $<$. 
Every graph can be equipped with some total order without increasing the twin-width, so we may assume the order is present (also such an order can be easily computed from an uncontraction sequence, but the problem of finding it efficiently given the graph $G$ only, remains open).
Similarly as twin-width,
the notion of flip-width extends to structures with several binary relations, as defined below.
This, in particular implies a notion of flip-width for ordered graphs,
but in the case of ordered graphs, we also given an alternative, more convenient definition.


\paragraph{Flip-width of binary structures}
Fix a signature $\Sigma$ consisting of several binary relations.
For a $\Sigma$-structure $A$
and partition $\cal P$ of $V(A)$,
a \emph{$\cal P$-flip} of $A$ is specified by a $\Sigma$-structure $F$ with domain $\cal P$. 
Applying the flip $F$ results in a $\Sigma$-structure  
$A'$ with domain $V(A)$ and relations \[R_{A'}\quad:=\quad R_A\quad\triangle \bigcup_{(P,Q)\in R_F}P\times Q,\] for each binary relation symbol $R\in\Sigma$.
A \emph{$k$-flip} of $A$ is any structure $A'$ obtained in this way, for some partition $\cal P$ with $|\cal P|\le k$.
Using this notion of a $k$-flip, the flipper game and the flip-width parameters $\fw_r(A)$ are defined just as for graphs (see \focs[\cite{flip-width-arxiv}]{Appendix~\ref{app:bin-struct}} for details),
where in each round, the runner can traverse at most $r$ edges of the Gaifman graph of the previously announced $k$-flip.
Theorem~\ref{thm:btww} holds also in the case when $G$ is a binary structure, rather than a graph, with the same proof.

\paragraph{Ordered flip-width}
In the case of ordered graphs, it is convenient 
to work with the following variant of flip-width which takes into account 
that one of the binary relations is a total order.
Fix an ordered graph $G=(V,E,<)$. A \emph{$k$-cut-flip} of $G$ is a triple $G'=(V,E',S)$, where $(V,E')$ is a (usual) graph that is a $k$-flip of the graph $(V,E)$, and $S\subset V$ is a set with $|S|\le k$.
Intuitively, the runner will be able 
to instantaneously move between any two points that are not separated 
by an element of $S$.
 Let $\sim_S$ denote the equivalence relation on $S$
such that $u\sim_S v$ if and only if $u=v$ or $u<v$ and there is no $s\in S$ with $u\le s\le v$.
The \emph{weighted graph} associated to $G'=(V,E',S)$ is the graph with vertices $V$ and edges $uv$ such that $uv\in E'$ or $u\sim_S v$,
where each edge $uv$ with $u\sim'v$ has weight $0$,
and the remaining edges have weight $1$.
Fix a radius $r\in\N\cup\set{\infty}$, a width parameter $k$,
and an ordered graph $G=(V,E,<)$.
In the \emph{ordered flipper game} with radius $r$ and width $k$ on an ordered graph $G=(V,E,<)$, 
 in round $i$, the flipper announces
 a $k$-cut-flip $G_i=(V,E',S)$ of the ordered graph $G=(V,E,<)$,
and the runner moves from his previous position $v_{i-1}$ to a new position $v_{i}$ by following a path of total weight at most $r$
in the weighted graph associated with the previous $k$-cut-flip $G_{i-1}$
(in round $i=1$, as the flipper announces $G_1$, the runner picks $v_1\in V$ arbitrarily).
The flipper wins if $v_i$ is isolated in the weighted graph associated to $G_i$, that is, there is no $w\neq v$ with $vw\in E'$ or $v\sim'w$.
The \emph{radius-$r$ ordered flip-width} of an ordered graph $G=(V,E,<)$,
denoted $\fw_r^<(G)$, is 
 the smallest number $k$ such that the flipper has a winning strategy in the  ordered flipper game with radius $r$ and width $k$ on $G$.


For an ordered graph $G=(V,E,<)$, the parameter $\fw_r^<(G)$ relates to the flip-width of $G$, treated as a binary structure, as follows (see \focs[\cite{flip-width-arxiv}]{Appendix~\ref{app:ordered flip-width}}).
\begin{lemma}[\appmark]\label{lem:ordered flip-width}
  Fix $r\in \N\cup\set{\infty}$ and an ordered graph $G=(V,E,<)$.
  Then \[\sqrt{\fw_r(G)+1}\quad\le\quad fw_{r}^<(G)+1\quad\le\quad \fw_{3r+2}(G)+1.\]
 \end{lemma}

\paragraph{Equivalence of flip-width and twin-width for ordered graphs}
We now prove the main result of Section~\ref{sec:tww}, which says 
that for ordered graphs $G=(V,E,<)$, 
the parameters $\tww(G)$ and $\fw^<_1(G)$ are functionally equivalent.
In other words, a class of ordered graphs has bounded twin-width if and only if it has bounded radius-one ordered flip-width.
\begin{theorem}\label{thm:tww1}
  The following conditions are equivalent for a class $\CC$ of ordered graphs:
  \begin{enumerate}
    \item $\CC$ has bounded twin-width,
    \item $\CC$ has bounded flip-width, that is, $\fw_r(\CC)<\infty$ for every $r\in \N$,
    \item $\CC$ has bounded radius-five flip-width (as a class of binary structures), that is, $\fw_5(\CC)<\infty$,
    \item $\CC$ has bounded ordered flip-width, that is, $\fw_r^<(\CC)<\infty$ for every $r\in \N$,
    \item $\CC$ has bounded radius-one ordered flip-width, that is, $\fw_1^<(\CC)<\infty$.   
  \end{enumerate}
\end{theorem}

So for ordered graphs, among the parameters $\fw_r^<$, only the parameters $\fw_\infty^<$
(characterizing bounded clique-width) and $\fw_1^<$ (characterizing bounded twin-width) are relevant, since it follows that $\fw_r^<(G)$  is bounded in terms of $\fw_1^<(G)$ and $r$, for every $r\in\N$ and  ordered graph $G$. 

We obtain the following characterization of twin-width of usual, unordered graphs, in terms of flip-width. Recall that every graph $G$ can be equipped with some total order without increasing the twin-width, and conversely, forgetting a total order of an ordered graph does not increase the twin-width. 

\begin{corollary}\label{cor:tww-char}
  A class $\CC$ of graphs has bounded twin-width if and only if every graph in $\CC$ can be equipped with a total order, so that the resulting class 
  of ordered graphs has bounded flip-width.
\end{corollary}

\begin{proof}[Proof of Theorem~\ref{thm:tww1}]
  The implication 1$\rightarrow$2 is by Theorem~\ref{thm:btww}
(stated for binary structures);
the implications 2$\rightarrow$3 and 4$\rightarrow$5 are immediate;
the implications 2$\rightarrow$4 and 
3$\rightarrow$5 
are by Lemma~\ref{lem:ordered flip-width}. 
It remains to prove the implication 5$\rightarrow$1.
This is done in 
Lemma~\ref{lem:rich} below, which proves that for ordered graphs, twin-width is bounded in terms of $\fw_1^<$.
\end{proof}

We show that ordered graphs of large twin-width have large $\fw_1^<$.
To this end, we use a core result of \cite{tww4},
which states that an ordered graph has large twin-width 
if and only if it contains a \emph{$k$-rich division},
for a large number $k$.
 A $k$-rich division of an ordered graph $G$ is a pair of partitions $\mathcal L,\mathcal R$ of $V(G)$, whose parts are intervals with respect to the order, such that for every interval $A\in \mathcal L$ 
and $k$ intervals $B_1,\ldots,B_k\in \mathcal R$, there are at least $k$ vertices in $A$ with pairwise distinct neighborhoods in $V(G)\setminus (B_1\cup\cdots\cup B_k)$, and symmetrically, for every interval $B\in \mathcal R$ 
and $k$ intervals $A_1,\ldots,A_k\in \mathcal L$, there are at least $k$ vertices in $B$ with pairwise distinct neighborhoods in $V(G)\setminus (A_1\cup\cdots\cup A_k)$.
It is shown in~\cite[Theorem 21]{tww4}
that if $G$ has no $k$-rich division, then $\tww(G)\le 2^{O(k^2)}$.
We now show that a $(k+1)$-rich division can be employed by the runner to evade the flipper in the ordered flipper game of radius $1$ and width $k$.
The following lemma immediately yields the implication 4$\rightarrow$1 
in Theorem~\ref{thm:tww1}, and finishes its proof.

\begin{lemma}\label{lem:rich}
 Let $G$ be an ordered graph with $\fw_1^<(G)\le k$.
 Then $G$ does not have a $(k+1)$-rich division $\cal L,\cal R$.
 In particular, $\tww(G)\le 2^{O(k^2)}$.
\end{lemma}
\begin{proof}
  Suppose $G$ has a $(k+1)$-rich division $\cal L,\cal R$.
We show that $\fw_1^<(G)>k$, by describing a winning strategy for the runner 
in the  ordered flipper game with radius one and width $k$.

The strategy is as follows: in round $i$, when the flipper announces a new $k$-cut-flip $G_i=(V,E_i,S_i)$ of $G$, the runner always moves to any reachable vertex  in one of the parts of $\cal L$ (in even-numbered rounds) or of $\cal R$ (in odd-numbered rounds) that does not contain any element of $S_i$. We show that the runner can always reach such a vertex $v_i$ from his previous position $v_{i-1}$, by following a path of weight $1$ in the previous $k$-cut-flip $G_{i-1}$ of $G$.

By inductive assumption, suppose that $v_{i-1}$ belongs to a part $A$ of $\cal L$ or $\cal R$ that does not contain any element of $S_{i-1}$.
Suppose, by symmetry, that $A$ is a part of $\cal L$.
Since $|S_i|\le k$, there are at most $k$ parts of $\cal R$ that contain some element of $S_i$; denote those parts by $B_1,\ldots,B_k$.
By the condition of a $(k+1)$-rich division 
there are  $k+1$ vertices in $A$ 
with pairwise distinct neighborhoods in $B:=V(G)-(B_1\cup\cdots\cup B_k)$. 
Then in the previous $k$-cut-flip $G_{i-1}$ of $G$, there is an edge joining some vertex  $a\in A$ with some vertex $b\in B$.
Indeed, let $\cal P$ denote the partition of $V(G)$  
underlying the flip $G_{i-1}$ of $G$, with $|\cal P|\le k$.
By the pigeonhole principle, two vertices $a_1,a_2$ of 
$A$ with distinct neighborhoods in $B$, ie.  $N_G(a_1)\cap B\neq N_G(a_2)\cap B$
belong to the same part of $\cal P$. Therefore also $N_{G_{i-1}}(a_1)\cap B\neq N_{G_{i-1}}(a_2)\cap B$,
so one of those sets must be nonempty. 

Thus, there is $a\in A$ and $b\in B$ 
such that $ab\in E(G_{i-1})$.
The runner thus moves to $b$, by the path $v_{i-1}-a'-b$ of weight $1$ in $G_{i-1}$, maintaining the invariant. 
\end{proof}


Reassuming, the duality result of \cite{tww4} proving the equivalence of unbounded twin-width and having $k$-rich divisions for all $k$, 
can be seen as a min-max theorem for the flipper game of radius $1$, for ordered graphs.
Moreover, we now see that degeneracy and twin-width are two flip sides of the same coin:  $\fw_1(G)$ corresponds to the degeneracy of $G$ for weakly sparse graphs (see Theorem~\ref{thm:degeneracy} and Theorem~\ref{thm:wsparse}), and $\fw_1^<(G)$ corresponds to the twin-width of $G$ for ordered graphs (see Theorem~\ref{thm:tww1}). 
Similarly, classes of bounded flip-width
coincide with classes of bounded expansion in the weakly sparse case, and with classes of bounded twin-width in the ordered case.
Both of those cases is by now well-understood from an algorithmic, 
combinatorial, and logical perspective. In particular, the model checking problem for first-order logic is fixed-parameter tractable in each of the two special cases.

\section{Closure under transductions}\label{sec:transductions}
As we have seen, classes of bounded flip-width include all classes of bounded expansion and all classes of bounded twin-width, and characterize those notions 
in the weakly sparse and totally ordered settings. 
We argue that classes of bounded flip-width 
enjoy many good closure properties.
 For instance, if two classes $\CC$ and $\DD$ have bounded flip-width, then their union $\CC\cup \DD$ also has bounded flip-width. Other such properties 
include: closure under disjoint unions (see Example~\ref{ex:disjoint}),  and closure under substitution (see Lemma~\ref{lem:substitution}).

An entire family of closure properties is provided by the notion of first-order \emph{interpretations} or \emph{transductions}.
As a very special instance, we saw in Example~\ref{ex:complement} that if $\CC$ has bounded flip-width, then the class of edge-complements of graphs from~$\CC$
also has bounded flip-width. What about, say, the class of squares of graphs from~$\CC$? (The square of a graph $G$ has vertices $V(G)$ and edges $uv$ such that $u$ and $v$ have a common neighbor in $G$.) We show that this class also has bounded flip-width, and a similar  result holds for every operation that can be defined by a first-order formula, as we now describe. We phrase our result in greater generality for colored graphs. This result has multiple corollaries, and generalizes previous results.

\subsection{Preservation of flip-width under transductions}
We start with  defining interpretations and transductions, and then state the main result of this section.

\paragraph{Colored graphs}
Recall that a $c$-colored graph is a graph together with an assignment of colors from $\set{1,\ldots,c}$ to its vertices.
For a $c$-colored graph $G$, its radius-$r$ flip-width $\fw_r(G)$,
for $r\in\N\cup\set\infty$, is defined as the same parameter for the underlying uncolored graph.
Say that a class~$\CC$ of $c$-colored graphs 
has \emph{bounded flip-width} if the underlying class of uncolored graphs has bounded flip-width.
A $c$-colored graph is seen as a structure over the signature consisting of the binary relation $E(x,y)$ denoting adjacency, as well as unary predicates $C_1(x),\ldots,C_c(x)$ denoting the respective colors.

\paragraph{Interpretations} The following notion is a special case of a (simple, domain-preserving) first-order interpretation.
Let $G$ be a $c$-colored graph
and $\phi(x,y)$ be a first-order formula
in the signature of $c$-colored graphs.
Define the graph $\phi(G)$ with vertices $V(G)$ and edges 
$uv$ such that $u\neq v$ and $\phi(u,v)\lor \phi(v,u)$ holds in $G$. 
For a class $\CC$ of $c$-colored graphs, denote $\phi(\CC):=\setof{\phi(G)}{G\in\CC}$. 
The class $\phi(\CC)$ is called an \emph{interpretation} of $\CC$, via $\phi$.
For example, for the formula $\phi(x,y)=\neg E(x,y)$
and a graph $G$, the graph $\phi(G)$ is the complement $\bar G$ of $G$.
And for the formula $\phi(x,y)=\exists z.[E(x,z)\land E(z,y)]$, the graph $\phi(G)$ is the square of $G$.

\paragraph{Transductions}Say that a graph class $\CC$ \emph{transduces} a graph class $\DD$,
or that $\DD$ is a \emph{transduction} of $\CC$,
if there is some $c\ge 1$ and class $\wh\CC$ of $c$-colored graphs which is a $c$-coloring of $\CC$, and some first-order formula $\phi(x,y)$ in the signature of $c$-colored graphs,
such that every graph in $\DD$ is an induced subgraph of some graph in  $\phi(\wh\CC)$ (that is, $\DD$ is contained in the hereditary closure of $\phi(\wh\CC)$).

\begin{example}
  Let $\CC$ be the class of all half-graphs (see Figure~\ref{fig:half-graph}),
  where the half-graph $H_n$ of order $n$ 
  has vertices $a_1,\ldots,a_n$ and $b_1,\ldots,b_n$,
  and edges $a_ib_j$ for $1\le i<j\le n$.
  We show that $\CC$ transduces the class $\DD$ consisting of disjoint unions of cliques.

  We use two colors.\szfuture{pic}
  Let $\wh\CC$ be the class of all $2$-colored half-graphs.
  Consider the formula $\phi(x,y)$ expressing that there is no vertex of color $2$ which is adjacent to one of $x,y$, and not the other:
  \[\phi(x,y)\equiv \forall z.C_2(z)\rightarrow (E(x,z)\leftrightarrow E(y,z)).\]

  We argue that the  hereditary closure of $\phi(\wh C)$ 
  contains $\DD$, implying that $\DD$ transduces in $\CC$.
  
  Let $F\in \DD$ be a disjoint union of cliques. 
Let $1,\ldots,n$ be the vertices of $F$, for some $n\ge 0$, and assume that every connected component of $F$ 
consists of consecutive vertices in the usual order $1<\ldots<n$.
Consider the half-graph $H_n$ with vertices $a_1,\ldots,a_n$ and $b_1,\ldots,b_n$, and color 
a vertex $b_i$ with color $2$ if 
$i$ is the largest element of its connected component 
in $F$, and with color $1$ otherwise.
Now, for all $1\le i,j\le n$, 
$H_n\models \phi(a_i,a_j)$  if and only if $i$ and $j$ are adjacent in $F$.
Hence, $\phi(H_n)[\set{a_1,\ldots,a_n}]$ is isomorphic to $F$, and therefore $F$ is an induced subgraph of some graph in $\phi(\wh \CC)$.
\end{example}

\paragraph{Transductions and flip-width}
We prove the following theorem.
\begin{theorem}[\appmark]\label{thm:interpretations}
  Fix $q\ge 0$. There is a computable function $T_q\from\N\to\N$ with the following property.
  Fix numbers $r,c\ge 1$ and a first-order formula $\phi(x,y)$ of quantifier rank $q$ in the signature of $c$-colored graphs.
  Set $r':=2^q\cdot r$.
  Then 
for every $c$-colored graph $G$ we have 
\begin{align}\label{eq:interpretations}
  \fw_r(\phi(G))\le T_q(\fw_{r'}(G)\cdot c).  
\end{align}
  In particular, if $\CC$ has bounded flip-width, then $\phi(\CC)$ has bounded flip-width.
  Moreover, $T_0(k)=k$ for every $k$.
\end{theorem}
Before discussing the proof of Theorem~\ref{thm:interpretations}, we discuss its consequences.

The following is an immediate consequence of Theorem~\ref{thm:interpretations}, and the fact that radius-$r$ flip-width is 
 monotone with respect to induced subgraphs.
\begin{corollary}\label{cor:transductions}
  If a class $\CC$ has bounded flip-width and transduces a class $\DD$, then $\DD$ has bounded flip-width.
  \end{corollary}

  Since the class of all graphs has unbounded flip-width by  Corollary~\ref{cor:fw1-nontrivial}, we get the following.
  
  \begin{corollary}\label{cor:mnip}
    If $\CC$ is a class of bounded flip-width, then $\CC$ does not transduce the class of all graphs.
  \end{corollary}

  In the language of model theory, Corollary~\ref{cor:mnip}
  says that classes of bounded flip-width are \emph{monadically dependent},
  see Section~\ref{sec:subpoly}.
  In the phrasing of~\cite{stable-tww-lics,transduction-quasiorder}, Corollary~\ref{cor:transductions} says that 
classes of bounded flip-width form a \emph{transduction ideal}.
By Theorem~\ref{thm:wsparse}, the weakly sparse classes in this transduction ideal are exactly the classes of bounded expansion. Therefore, this transduction ideal is contained 
in the \emph{dense analogue} of bounded expansion classes \cite{stable-tww-lics}.

A graph class has \emph{structurally bounded expansion}~\cite{lsd-journal} if it is a transduction  of a class with bounded expansion. Corollary~\ref{cor:transductions} immediately implies that those classes have bounded flip-width.
\begin{corollary}
  Every class of structurally bounded expansion has bounded flip-width.
\end{corollary}

We remark that an extension of  Theorem~\ref{thm:interpretations} and of Corollary~\ref{cor:transductions}
also hold, with the same proof (but a slightly different bound in \eqref{eq:interpretations}), for input structures over a binary signature $\Sigma$, rather than graphs $G$,
and where instead of one formula $\phi(x,y)$
defining a new edge relation of the output graph,
we have a tuple  of formulas $\bar\phi=(\phi_1(x,y),\ldots,\phi_k(x,y))$ defining $k$ binary relations of the output structure $\bar \phi(G)$
with domain $V(G)$ and $i$th relation $R_i$, for $i=1,\ldots,k$, interpreted as \[R_i:=\setof{(a,b)}{a,b\in V(G), G\models \phi_i(a,b)}.\]
We then obtain that if $\CC$ is a class of $\Sigma$-structures of bounded flip-width (see Sec.~\ref{sec:tww} and \focs[\cite{flip-width-arxiv}]{Appendix~\ref{app:bin-struct}}),
then the class $\bar\phi(\CC):=\setof{\bar\phi(G)}{G\in\CC}$ also has bounded flip-width.

This implies (using Corollary~\ref{cor:tww-char}) the result of \cite{tww1},
that transductions preserve classes of bounded twin-width.
\begin{corollary}[\cite{tww1}]
  If $\CC$ has bounded twin-width and 
  $\phi(x,y)$ is a first-order formula, then $\phi(\CC)$ has bounded twin-width.
\end{corollary}

Indeed, let $\CC$ be a class of bounded twin-width and $\phi(x,y)$ a symmetric formula in the signature of graphs. By Corollary~\ref{cor:tww-char} there 
is a class $\wh\CC$ of ordered graphs of bounded flip-width such that $\CC$ is obtained from $\wh\CC$ by forgetting the order.
Let
 $\bar\phi$ be the pair consisting of the formula $\phi(x,y)$ and the formula $x<y$ defining the order.
Then $\bar\phi(\wh\CC)$ is a class of ordered graphs which has bounded flip-width, since $\wh\CC$ has bounded flip-width.
Moreover, the graph class $\phi(\CC)$ is obtained from the class of ordered graphs $\bar\phi(\wh\CC)$ by forgetting the order. Hence, $\phi(\CC)$ has bounded twin-width, by Corollary~\ref{cor:tww-char}.

Together with Theorem \ref{thm:wsparse}, Corollary~\ref{cor:transductions} implies the following consequence\sz{expand} of \cite{lsd-icalp}.
\begin{corollary}
  Every weakly sparse class with structurally bounded expansion has bounded expansion.  
\end{corollary}

\medskip

  The proof of Theorem~\ref{thm:interpretations}, which is sketched below, relies on locality of first-order logic, a central notion for analysing 
  first-order formulas on sparse graphs~\cite{seese_1996}, which, in some form, also plays a key role in understanding first-order formulas on classes of bounded twin-width~\cite{tww1, tww-types-icalp}.
  The function $T_q(k)$, although computable, is astronomical:
\[T_q(k):=\underbrace{2^{2^{.^{.{^{.^{2^m}}}}}}}_{{\rm height\,} q}\]
  where $m$ is the number of distinct $k$-colored graphs with vertex set $\set{1,\ldots,q+1}$. However, for $q=0$ we have $T_0(k)=k$,
  so the bound \eqref{eq:interpretations} 
  becomes $\fw_r(\phi(G))\le c\cdot \fw_r(G)$ 
  in the case when 
  $\phi(x,y)$ is a quantifier-free formula and $G$ is a $c$-colored graph with underlying graph $G_0$.

  \paragraph{$\cmso$ transductions and radius-$\infty$ flip-width}
  For the case of radius $r=\infty$,
  we prove that classes $\CC$ with $\fw_\infty(\CC)<\infty$ are preserved under transductions expressed in the more powerful logic called $\cmso$.  
  We remark that this result is known for classes of bounded clique-width,
  so our result follows from Theorem~\ref{thm:cw}, which states that $\CC$ has bounded clique-width if and only if $\fw_\infty(\CC)<\infty$.
  However, we provide a separate proof here, as it is very analogous
  to the proof of Theorem~\ref{thm:interpretations}, 
  and nicely illustrates the parallels between the bounded flip-width case, and the limit case of bounded radius-$\infty$ flip-width.
  We also show how the closure under $\cmso$ transductions 
  can be used to obtain an alternative proof of Theorem~\ref{thm:cw}.
  Let us first define the logic, and state the result.
  
  \medskip
  \emph{Counting Monadic Second Order Logic} ($\cmso$) is the extension of first-order logic, where apart from first-order quantifiers $\exists x,\forall x$, that range over vertices $x$ of a graph,
  we have second-order quantifers $\exists X,\forall X$, that range over sets $X$ of vertices of the graph. We additionally have the atomic predicate $x\in X$ that allows to check whether a given vertex belongs to a given set, and the divisibility predicates $\textrm{div}_k(X)$, for $k\ge 1$, where $\textrm{div}_k(X)$ holds for a given set of vertices $X$ if and only if $|X|$ is divisible by $k$.
  The usual constructs of first-order logic
  (boolean connectives and relation
  symbols, such as adjacency in a graph) are also included.

   $\cmso$ is able to express non-local properties. 
   For instance, the formula 
   $$
   \phi(x,y)=\neg \exists X.\left[(x\in X)\land(y\not\in Y)\focs[\\]{}\land\forall z. \forall t. \left[(z\in X)\land E(z,t)\rightarrow (t\in X)\right]\right]$$
   expresses that $x$ and $y$ lie in the same connected component.
  Still, $\cmso$ enjoys property similar to locality of first-order logic,
  called compositionality (or a Feferman-Vaught-Mostowski type result), which is an analogue of locality for the radius $r=\infty$.
  Using this instead of locality, 
  with the same proof as in Theorem~\ref{thm:interpretations}, we get the following.
  
  \begin{theorem}[\appmark]\label{thm:interpCMSO}
    Let $\CC$ be a class of $c$-colored graphs of bounded $\infty$-flip-width and let $\phi(x,y)$ be a formula of $\cmso$.
    Then $\phi(\CC)$ has bounded $\infty$-flip-width.
  \end{theorem}

  Say that a class $\CC$ \emph{$\cmso$-transduces} a class $\DD$,
  or that $\DD$ is a \emph{$\cmso$-transduction} of $\CC$,
  if there is a $c$-coloring  $\wh\CC$ of $\CC$, for some $c\ge 1$, and some $\cmso$ formula $\phi(x,y)$ in the signature of $c$-colored graphs,
  such that every graph in $\DD$ is an induced subgraph of some graph in  $\phi(\wh\CC)$ (that is, $\DD$ is contained in the hereditary closure of $\phi(\wh\CC)$). 
  
    \begin{corollary}\label{cor:mnip-cmso}
      If $\CC$ is a class with $\fw_\infty(\CC)<\infty$ that $\cmso$-transduces a class $\DD$, then ${\fw_\infty(\DD)<\infty}$.
    \end{corollary}

    We now show that, together with a result of Courcelle and Oum, Corollary~\ref{cor:mnip-cmso} yields an alternative proof of the backwards implication in  Theorem~\ref{thm:cw},  that classes of bounded radius-$\infty$ flip-width have bounded clique-width.

Courcelle and Oum~\cite[Corollary 7.5]{courcelle-oum-vertex-minors-seese} proved the following result\footnote{In fact, they proved an equivalent statement, but with `all square grids' instead of `all graphs'. However, it is easy to see that the class of all square grids $\cmso$-transduces the class of all graphs, and transductions can be composed.}.
\begin{theorem}\label{thm:courcelle-oum}
  Every graph class of unbounded clique-width
   $\cmso$-transduces the class of all graphs.
\end{theorem}
In particular, as the class of all graphs has unbounded radius-one flip-width
by Corollary~\ref{cor:fw1-nontrivial}, it follows from Corollary~\ref{cor:mnip-cmso}
and Theorem~\ref{thm:courcelle-oum} that every class $\CC$ of unbounded clique-width
has $\fw_\infty(\CC)=\infty$, proving the backward implication in Theorem~\ref{thm:cw}.

    
\subsection{Transferring strategies}\label{sec:transfer}    
  \newcommand{\HHH}{H}
  \newcommand{\GGG}{G}
Before giving the details about the proof of Theorem~\ref{thm:interpretations}, we define a tool that will be also useful in other contexts.
The following notion
 allows to transfer a winning strategy of the flipper from a graph $\GGG$ to a graph $\HHH$. Let $\GGG$ and $\HHH$ be two graphs with $V(\HHH)\subset V(\GGG)$. Fix $r_G,r_H\in\N\cup\set{\infty}$ and $k,\ell\in\N$.
 Suppose furthermore that the following are given:
 \begin{itemize}
  \item a mapping $F$ that maps each $k$-flip $\GGG'$ of 
  $\GGG$ to an $\ell$-flip $\HHH'=F(\GGG')$ of $\HHH$,
\item a strategy of the flipper in the game on the graph $\GGG$ with radius 
$r_G$ and width~$k$.
 \end{itemize} 
 This induces the following strategy of the flipper in the flipper game of radius $r_H$ and width $\ell$ on the graph $\HHH$; call this game the \emph{$\HHH$-game}. Consider a play of this game the $\HHH$-game. Simultaneously, initiate the \emph{$\GGG$-game} on $\GGG$,
with radius $r_G$ and width $k$, in which we will copy runner's moves from the $\HHH$-game.
Whenever the flipper announces a $k$-flip $\GGG'$ of $\GGG$ in the $\GGG$-game,
then in the $\HHH$-game
the flipper announces the $\ell$-flip $\HHH'=F(\GGG')$ of $\HHH$. 
Whenever the runner moves to a vertex $v$ in the $\HHH$-game, 
we also move the runner to $v$ in the $\HHH$-game. 
Note that this move might not be a valid move in the $\HHH$-game.
In this case, the flipper continues the $\GGG$-game by playing arbitrarily (e.g. always announcing $\GGG$ as the next flip) until the end of the game. We say the resulting strategy of the flipper (in the $\GGG$-game) is  \emph{transferred} from the considered strategy in the $\HHH$-game, according to the mapping $F$.

The following lemma 
gives a condition which implies that  if the original strategy on $\GGG$ is winning,
then the transferred strategy on $\HHH$ is winning.

\begin{lemma}\label{lem:strategy-transfer}
  Fix  $r\in\N\cup\set{\infty}$ and $k,\ell,s\ge 1$.
Let $\HHH,\GGG$ be two graphs with $V(\HHH)\subset V(\GGG)$.
Suppose that for all $k\ge 1$ and every $k$-flip $\GGG'$ of $\GGG$ there is some $\ell$-flip $\HHH'=F(\GGG')$ of $\HHH$,
such that
\begin{align}\label{eq:transfer}
   \dist_{\GGG'}(u,v)\le s\qquad\text{  for all $uv\in E(\HHH')$.}
\end{align}
Then transferring a winning strategy of the flipper from the flipper game on $\GGG$ with radius $rs$ and width $k$, according to the mapping $F$, results in a winning strategy of the flipper in the flipper game on $\HHH$ with radius $r$ and width $\ell$.
In particular,
 \[\fw_{rs}(\GGG)\le k\qquad \text{implies} \qquad\fw_{r}(\HHH)\le \ell.\]
\end{lemma}
In the case $r=\infty$, the conclusion reads 
{``$\fw_\infty(\GGG)\le k$ implies $\fw_\infty(\HHH)\le \ell$''}. Moreover we can then replace $\le s$ by $<\infty$ in \eqref{eq:transfer}, since for $s:=|V(\GGG)|$ we have that $\dist_{\GGG'}(u,v)\le s\iff{\dist_{\GGG'}(u,v)<\infty}$.
We state this explicitly below.
\begin{corollary}\label{lem:strategy-transfer-infty}
  Fix  $k,\ell\ge 1$.
Let $\HHH,\GGG$ be two graphs with $V(\HHH)\subset V(\GGG)$.
Suppose that for all $k\ge 1$ and every $k$-flip $\GGG'$ of $\GGG$ there is some $\ell$-flip $\HHH'=F(\GGG')$ of $\HHH$,
such that
\begin{align}\label{eq:transfer-infty}
   \dist_{\GGG'}(u,v)<\infty \qquad\text{  for all $uv\in E(\HHH')$.}
\end{align}
Then \[\fw_{\infty}(\GGG)\le k\qquad \text{implies} \qquad\fw_{\infty}(\HHH)\le \ell.\]
\end{corollary}

\begin{proof}[Proof of Lemma~\ref{lem:strategy-transfer}]



Consider a play in the flipper game on $\HHH$
of radius $r$ and width $\ell$, according to the strategy transferred from $\GGG$, as described above. 
Suppose that, in some round, $\GGG'$ is the announced flip in the $\GGG$-game 
and $\HHH'=F(\GGG')$ is the announced flip in the $\HHH$-game, 
and that $v$ is the new vertex chosen by the runner in the $\HHH$-game.
By \eqref{eq:transfer}, the following holds:
\[B^r_{\HHH'}(v)\subset B^{rs}_{\GGG'}(v).\]
In the case $r=\infty$, each side of the inclusion should be interpreted as the connected component of $v$ in the appropriate graph.

 In particular, in the next round, 
every valid move of runner in the $\HHH$-game will also be a valid move 
of runner in the $\GGG$-game (this is trivially satisfied in the first round),
and if runner is trapped in the $\GGG$-game, that is, $|B^{rs}_{\GGG'}(v)|=1$, then 
also $|B^r_{\HHH'}(v)|=1$, so
he is also trapped in the $\HHH$-game. Thus, this describes a winning strategy 
for the flipper in the flipper game on $\HHH$ with radius $r$ and width $\ell$.
In particular, $\fw_r(\HHH)\le \ell$.
\end{proof}

\subsection{Proof of Theorem~\ref{thm:interpretations}}
  We now sketch the proof of Theorem~\ref{thm:interpretations}.
The details are presented in \focs[\cite{flip-width-arxiv}]{Appendix~\ref{app:transductions}},
along with the proof of Theorem~\ref{thm:interpCMSO}.

  For simplicity, we assume the case $c=1$, that is, when the considered graphs $G$ have no colors. The general case proceeds analogously. 
  
  Our aim is to transfer a winning strategy 
  of the flipper from $G$ to $H=\phi(G)$ (with appropriate radii),
  by applying Lemma~\ref{lem:strategy-transfer}. 
  So for every flip $G'$ of $G$ we need to produce a flip $\phi(G)'$ of $\phi(G)$
  such that adjacent vertices in $\phi(G)'$ are not too far in $G'$.
  We first show how to achieve this in the case when $G'=G$,
  using locality, a well known tool from finite model theory, which we now recall.

  \paragraph{Locality}
Fix a number $s\in\N$.
  Say that a formula $\phi(x,y)$ is \emph{$s$-local},
  if for any graph $G$
  there is a labelling of the vertices of $G$ using a bounded number of labels
  (depending only on $\phi$, and not on $G$)
such that for any two vertices $u,v$ of $G$
with $\dist_G(u,v)>s$, whether or not $G\models \phi(u,v)$ depends only 
on the label of $u$ and the label of $v$.
It is well-known (and follows for instance from Gaifman's locality theorem) 
that every formula $\phi(x,y)$ of first-order logic is 
$s$-local for some radius $s$ depending only on the quantifier rank $q$ of $\phi$.
Namely, one can take $s:=2^q$.
The label assigned to a vertex $v$ of $G$ as above,
is essentially the set of formulas $\alpha(x)$ 
of quantifier rank at most $q$, such that $\alpha(v)$ holds in $G$.
The number of such formulas, up to equivalence, is finite, and is bounded by $T_q(1)$,
where $T_q$ is the function described above.\sz{check} 

\paragraph{Flipping $\phi(G)$}
We now show how to obtain a flip $\phi(G)'$ of $\phi(G)$ such that vertices that are adjacent in $\phi(G)'$ are not too far in $G$.
Let $\cal P$ be the partition of $V(G)=V(\phi(G))$
such that two vertices of $G$ are in the same part if they get the same label.
In particular, $|\cal P|\le T_q(1)$.
Now, in the graph $\phi(G)$ flip a pair of parts $A,B$ of $\cal P$
if and only if there is a pair of vertices $u\in A$ and $v\in B$,
such that $\dist_G(u,v)>s$ and $G\models \phi(u,v)$,
equivalently, $uv\in E(\phi(G))$.
The statement above implies that whether or not we flip $A$ and $B$,
does not depend on the choice of $u\in A$ and $v\in B$
such that $\dist_G(u,v)>s$. This yields a $\cal P$-flip of $\phi(G)$,
which we denote $\phi(G)'$.
Then the following holds for all $u,v\in V(G)$:
\begin{quote}
   $uv\in E(\phi(G)')$ \qquad implies\qquad $\dist_{G}(u,v)\le s$.  
\end{quote}


We now generalize the above reasoning to the case when $G'$ is a $k$-flip of $G$.
Again, the goal is to construct a flip $\phi(G)'$ of $\phi(G)$ such that vertices that are adjacent in $\phi(G)'$ are not too far in $G'$.

We treat $G'$ as a $k$-colored graph, by adding colors that mark parts of the partition that is used to produce the $k$-flip $G'$ of $G$.
The key observation is that we can write a formula $\phi'(x,y)$ that makes use of those colors, and such that 
\[G'\models \phi'(u,v) \qquad \Longleftrightarrow \qquad G\models \phi(u,v)\qquad \text{for all $u,v\in V(G)$.}\] 
This is because we can write a formula $\eps(x,y)$ such that $G'\models \eps(u,v)$ if and only if $G\models E(u,v)$
(the formula $\eps(x,y)$ checks the colors of $x$ and $y$, whether $x$ and $y$ are adjacent in $G'$, and inverts the flip). The formula $\phi'(u,v)$ is obtained by replacing each atomic formula $E(z,t)$ with $\eps(z,t)$. In particular, $\phi'$ has the same quantifier rank as $\phi$, and is therefore also $s$-local.
Applying the same argumentation as above, but this time to the formula $\phi'$ and the $k$-colored graph $G'$,
we obtain a $\cal P$-flip $\phi(G)'$ such that 
the following holds for all $u,v\in V(G)$:
\begin{quote}
   $uv\in E(\phi(G)')$ \qquad implies\qquad $\dist_{G'}(u,v)\le s$.  
\end{quote}
Moreover, the size of $\cal P$ can be bounded by a number $\ell$ depending on the formula $\phi$ and the number $k$.
Since this holds for every $k$-flip $G'$ of $G$,
we can now apply Lemma~\ref{lem:strategy-transfer} and conclude 
that $\fw_{r}(\phi(G))\le \ell$.

\section{Definable flip-width}\label{sec:definable}
Determining the flip-width of radius $r$ of a given graph $G$ 
seems computationally difficult. The space of all configurations in the flipper game of radius $r$ and width $k$ (consisting of a $k$-flip of $G$ and the runner's position) has size exponential in $|G|$,
and the naive algorithm for determining whether $\fw_r(G)\le k$,
which explores the space of all configurations, 
therefore runs in time exponential in $|G|$. We expect that for possible algorithmic applications, an algorithm which approximates $\fw_r(G)$, instead of computing it exactly, should suffice. This is what happens in the case of weak coloring numbers, and of twin-width, where such an approximation algorithm is not known in general, but is known in various special cases.

In this section, we introduce the \emph{definable flipper game} of radius $r$ and width $k$, in which the moves allowed for the flipper are parameterized by tuples of vertices of the graph, rather than by partitions of the vertex set.
Among other things, this reduces (comparing to the flipper game) the number of configurations, and the computational complexity of determining the definable flip-width of a given graph.
The main result of this section is Theorem~\ref{thm:dfw} that 
says that the definable flip-width at radius $r$ can be bounded in terms 
of the flip-width at radius $5r$.
This allows to obtain, in Theorem~\ref{thm:apx}, an algorithm 
for approximating flip-width, which runs in time $|G|^{O(k)}\cdot O_k(1)$,
where $k$ is the flip-width.
We use tools related to VC-dimension to accomplish this.

\paragraph{Atomic types and definable flips}
Let  $S$ be a set of vertices of a graph~$G$.
Consider the partition $\cal P$ of $V(G)$
such that two vertices $u,v\in V(G)$
are in the same part if $N(u)\cap S=N(v)\cap S$.
The equivalence classes of the partition $\cal P$ are called \emph{$S$-types}.
In particular, $|\cal P|\le 2^{|S|}$.

Say that a graph $G'$ is an \emph{$S$-definable flip of $G$}
if $G'$ is a $\cal P$-flip of $G$, where $\cal P$ is the partition of $V(G)$ into $S$-types.
Say that $G'$ is a \emph{$k$-definable flip} of $G$, if $G'$ 
is an $S$-definable flip of $G$ for some $S\subset V(G)$ with $|S|\le k$.

Note that a $k$-definable flip of $G$ is a $2^k$-flip of $G$.
However, there is no function $f$ such every $k$-flip of a graph $G$
is a $f(k)$-definable flip of $G$. For instance, 
the graph $G_n$ obtained from an $n$-clique $K_n$ by adding $n$ 
isolated vertices, is a $2$-flip of $K_{2n}$,
but is not an $(n-1)$-definable flip of $K_{2n}$.

\paragraph{Definable flipper game}
Fix $r\in\N\cup\set{\infty}$.
The \emph{definable flipper game of radius $r$ and width $k$}
is defined in the same way as the flipper game of radius $r$,
but now in each round the flipper is allowed to announce a $k$-definable flip $G'$ of $G$, rather than a $k$-flip of $G$.
\begin{definition}Fix $r\in\N\cup\set{\infty}$.
  The \emph{radius-$r$ definable flip-width} of a graph $G$,
  denoted $\dfw_r(G)$,
  is the smallest number $k$ such that the flipper has a winning strategy in the definable flipper game of radius $r$ and width $k$ on $G$.
\end{definition}

As every $k$-definable flip of $G$ is a $2^k$-flip of $G$,
it follows that for every $r\in\N\cup\set{\infty}$ and graph $G$ we have:
\begin{align}\label{eq:fw-dfw-trivial}
  \fw_r(G)\le 2^{\dfw_r(G)}.  
\end{align}


One advantage of the definable version of the flipper game 
is that it has far fewer configurations than the original flipper game.
As there are only $O(|G|^{k+1}\cdot 2^{4^k})$ configurations
in the definable flipper game of width $k$,
we get that it can be decided 
in time $|G|^{O(k)}\cdot 2^{O(2^k)}$ whether a given graph $G$ 
has $\fwd_r(G)\le k$ (see \focs[\cite{flip-width-arxiv}]{Appendix~\ref{app:dfw}} for a proof).
\begin{lemma}[\appmark]\label{lem:decide-dfw}
  There is an algorithm that, given a graph $G$ 
  and numbers $k\in\N$ and $r\in\N\cup\set{\infty}$, determines whether 
  $\fwd_r(G)\le k$ in time $n^{O(k)}\cdot 2^{O(2^k)}$.
\end{lemma}
\szfuture{$2^{O(2^k)}$ can be reduced to $2^{k^{O(1)}}$ using Sauer-Shelah} 
The following lemma is an immediate consequence of the Sauer-Shelah-Perles lemma 
(Lemma~\ref{lem:sauer-shelah-perles}).
\begin{lemma}[\appmark]\label{lem:fw-dfw-vc}
  Fix $r\in\N\cup\set{\infty}$.
  For every graph $G$ we have:
  \begin{align}
    \fw_r(G)&\le O(\dfw_r(G)^{\VCdim(G)}).   \label{eq:fw-dfw-vc}
  \end{align}
\end{lemma}

The following is the main result of Section~\ref{sec:definable}. It says that the definable flip-width 
can be bounded in terms of the flip-width, at the cost of increasing the radius. Below, $5\cdot\infty$ is interpreted as $\infty$.
\begin{theorem}\label{thm:dfw}
  Fix $r\in\N\cup\set{\infty}$.
  For every graph $G$ we have:
  \begin{align}
    \dfw_r(G)&\le O(\fw_{5r}(G)^3)\label{eq:dfw-fw}.
    \end{align}
\end{theorem}
We prove Theorem~\ref{thm:dfw} below. We first observe some consequences.
The bounds~\eqref{eq:fw-dfw-trivial} and \eqref{eq:dfw-fw}
give the following.
\begin{corollary}
  The following conditions are equivalent for a graph class $\CC$:
\begin{enumerate}\item $\CC$ has bounded flip-width, that is, $\fw_r(\CC)<\infty$ for all $r\ge 1$,
  \item $\dfw_r(\CC)<\infty$ for all $r\ge 1$.
\end{enumerate}
\end{corollary}
Similarly, for the case $r=\infty$, using Theorem~\ref{thm:cw} we get the following characterization 
of classes of bounded clique-width in terms of the definable flipper game with radius~$\infty$:
\begin{corollary}
  The following conditions are equivalent for a graph class $\CC$:
\begin{enumerate}\item $\CC$ has bounded clique-width,
  \item $\dfw_\infty(\CC)<\infty$.
\end{enumerate}
\end{corollary}


Finally, we get an algorithm for approximating the 
 flip-width of a given graph.
The algorithm is an approximation algorithm: unlike 
the algorithm in Lemma~\ref{lem:decide-dfw},
it does not allow to exactly determine whether the radius-$r$ flip-width 
of a given graph $G$ is smaller than a given number $k$.
Rather, it recognises one of two, non-exclusive, cases: whether $\fw_r(G)$ is small comparing to $k$, and whether $\fw_{5r}(G)$ is large comparing to $k$. Note that there is a gap in the radii, $r$ and $5r$. 
The running time of the algorithm is $O_k(1)\cdot n^{O(k)}$, which is called an XP algorithm (parameterized by $k$) in the language of parameterized complexity.

\begin{theorem}\label{thm:apx}
  There is a constant $C>0$ and an algorithm that inputs a graph~$G$ and numbers $r,k\in\N$,
  runs in time $n^{O(k)}\cdot 2^{O(2^k)}$,
  and either concludes that $\fw_r(G)\le 2^k$, 
  or concludes that ${\fw_{5r}(G)\ge C \cdot k^{1/3}}$.
\end{theorem}
\begin{proof}
  The algorithm 
   tests whether $\dfw_r(G)\le k$ in time $n^{O(k)}\cdot 2^{O(2^k)}$,
   using Lemma~\ref{lem:decide-dfw}. If $\dfw_r(G)\le k$
   it concludes that $\fw_r(G)\le 2^k$, by~\eqref{eq:fw-dfw-trivial}.
   If $\dfw_r(G)>k$, it concludes that $\fw_{5r}(G)\ge Ck^{1/3}$
   by~\eqref{eq:dfw-fw}, where $C>0$ is some fixed constant. 
\end{proof}

\paragraph{Proof of Theorem~\ref{thm:dfw}}
We now turn to the proof of Theorem~\ref{thm:dfw}. We use a result concerning graphs of small VC-dimension.
Recall that $\VCdim(G)\le O(\fw_1(G))$ by Theorem~\ref{thm:vcdim}.
The following
result from~\cite{boundedLocalCliquewidth},
 relies on the $(p,q)$-theorem of Alon-Kleitman-Matou\v sek~\cite{Matousek:p-q-theorem} (see \focs[\cite{flip-width-arxiv}]{Appendix~\ref{app:duality} and Appendix~\ref{app:inc}}).

\begin{lemma}[\appmark]\label{lem:incr}
  Fix $k,d\in\N$.
  Let $V$ be a set equipped with:
  \begin{itemize}  
    \item a binary relation $E\subset V\times V$ of VC-dimension at most $d$,
    \item a pseudometric $\dist\from V\times V\to \R_{\ge0}\cup\set{\infty}$
     (that is, a function satisfying the triangle inequality),
    \item 
    and a partition $\cal P$ of size at most $k$,
  \end{itemize}
  such that $E(u,v)$ depends only on the $\cal P$-class of $u$ and the $\cal P$-class of $v$ whenever $\dist(u,v)>1$.
  Then there is a set $S\subset V$ of size $\Oof(dk^2)$, such that $E(u,v)$ depends only on the $S$-types of $u$ and of $v$, whenever $\dist(u,v)>5$.
\end{lemma}

We reformulate Lemma~\ref{lem:incr} in terms of flips, as follows.

\begin{corollary}\label{cor:incr}
  Let $G$ be a graph and $d=\VCdim(G)$,
  and let $G'$ be a $k$-flip of $G$.
  Then there is a $O(dk^2)$-definable flip $H'$ of $G$
  such that 
\begin{align}\label{eq:dfw}
  \dist_{G'}(u,v)\le 5\qquad \text{for all $u,v\in E(H')$}
\end{align}
\end{corollary}
\begin{proof}
  Apply Lemma~\ref{lem:incr} to $V=V(G)$, 
      $E=E(G)$, $\dist\from V\times V\to \N\cup\set{\infty}$ denoting the shortest path metric in $G'$,
      and the partition $\cal P$ with $|\cal P|\le k$ such that $G'$ is a $\cal P$-flip of $G$.
      Note that for any two vertices $u,v\in V$,
      whether or not $uv\in E(G)$ holds, can be determined basing only on the $\cal P$-class of $u$, the $\cal P$-class of $v$, and on the information whether $uv\in E(G')$ holds.
      In particular, $uv\in E(G)$ depends only on 
      the $\cal P$-class of $u$ and the $\cal P$-class of $v$, for all $u,v\in V$ that are not adjacent in $G'$, equivalently, with $\dist(u,v)>1$.
    Hence, the assumption of Lemma~\ref{lem:incr} is satisfied.
    
    Let $S$ with $|S|\le O(dk^2)$ be as in the conclusion of the lemma, so that 
    whether or not $uv\in E(G)$, depends only on 
      the $S$-type of $u$ and the $S$-type of $v$, for all $u,v\in V$ with $\dist(u,v)>5$.
  
      Let $\cal P_S$ denote the partition of $V(G)$ into $S$-types.
      Let $H'$ be the $\cal P_S$-flip of $G$ 
      that flips between two $S$-types $A$ and $B$ 
      if and only if there are some $u\in A,v\in B$
      with $\dist(u,v)>5$ and $uv\in E(G)$.
      The conclusion follows.
\end{proof}

Theorem~\ref{thm:dfw} easily follows from Corollary~\ref{cor:incr}.

\begin{proof}[Proof of Theorem~\ref{thm:dfw}]
Let $G$ be a graph and let $k=\fw_{5r}(G)$. In particular, 
by Theorem~\ref{thm:vcdim}, we have that $\VCdim(G)\le O(k)$.
By Corollary~\ref{cor:incr} applied to $H:=G$,
the assumptions of Lemma~\ref{lem:strategy-transfer} 
are satisfied, where the $\ell$-flip $H'$ of $H=G$, for $\ell=2^{O(k^3)}$
is a $O(k^3)$-definable flip, as provided by Corollary~\ref{cor:incr}.
By transferring the winning strategy of the flipper in the flipper game 
of radius $5r$ and width $k$ on $G$, 
according to this mapping $G'\mapsto H'$, we get a strategy for the flipper in flipper game of radius $r$ on $G$, which uses only $O(dk^2)$-definable flips. 
By Lemma~\ref{lem:strategy-transfer}, this yields a 
winning strategy of the flipper that will use only $O(k^3)$-definable 
flips. Hence, $\dfw_r(G)\le O(k^3)$.
\end{proof}

\section{Almost bounded flip-width}
\label{sec:subpoly}

Recall from Fact~\ref{fact:nd-wcol} that a hereditary graph class $\CC$ is nowhere dense if and only if 
for every $r\ge 1$ and $G\in\CC$ we have $\wcol_r(G)=|G|^{o(1)}$.
  Inspired by this characterization, we extend the notion of bounded flip-width as follows.

\begin{definition}\label{def:subpoly}
  A graph class $\CC$ has \emph{almost bounded flip-width}
  if for every $r\ge 1$ and real $\eps>0$ we have $\fw_r(G)\le O_{\eps,r}(|G|^\eps)$ 
  for every graph $G$ in the hereditary closure of $\CC$.
\end{definition}

Note that we consider all graphs $G$ from the hereditary closure of $\CC$.
Otherwise, the class consisting of every graph $G$ with $2^{|G|}$ isolated vertices added to it, would have almost bounded flip-width, while according to the above definition, it does not. Indeed, we have the following lemma, which is an immediate consequence of Lemma~\ref{cor:large-vc-subgraph}.
\begin{lemma}\label{lem:abfw-vc}
  Every graph class with almost bounded flip-width has bounded VC-dimension.
\end{lemma}

Clearly, every class with bounded flip-width 
 has almost bounded flip-width.  
As we conjecture (see Conjecture~\ref{conj:fw-mnip}), classes of almost bounded flip-width coincide with monadically dependent classes (see definition below), analogously to the characterization of nowhere dense classes in Fact~\ref{fact:nd-wcol}.

In this section, we provide some evidence towards this conjecture. In Theorem~\ref{thm:nd-are-subpoly}, we prove that a weakly sparse class has almost bounded flip-width if and only if it is nowhere dense, if and only if it is monadically dependent. In Theorem~\ref{thm:snd-are-subpoly} we prove that structurally nowhere dense classes have almost bounded flip-width.
In Theorem~\ref{thm:subpoly-mstab} we prove that 
edge-stable classes of almost bounded flip-width are monadically dependent.
In Theorem~\ref{thm:ordered-abfw} we prove that 
classes of ordered graphs of almost bounded flip-width coincide with classes of bounded twin-width,
and with classes of bounded flip-width.
We start with recalling the discussed notions.

\subsection{Monadic dependence and monadic stability}

The following notion, due to Shelah~\cite{Shelah1986}
(see also~\cite{Braunfeld2021CharacterizationsOM}), originates in model theory.
\begin{definition}
  A graph class $\CC$ is \emph{monadically dependent} (or monadically NIP) if and only if it does not transduce the class of all graphs.   
\end{definition}

Monadically dependent classes have recently attracted attention in areas of structural and algorithmic graph theory~\cite{AdlerA14,rankwidth-meets-stability,tww4,stable-tww-lics}, as it is conjectured
(see Conjecture~\ref{conj:mnip-mc})
that monadically dependent classes are precisely those for which model checking first-order logic is fixed-parameter tractable. 

Monadically dependent classes include all nowhere dense classes, and in fact, among weakly sparse classes, they provide an exact characterization:

\begin{fact}[Consequence of \cite{dvorakInducedSubdivisions}+\cite{AdlerA14}]\label{fact:wsparse-nd}
  Let $\CC$ be a weakly sparse graph class.
  Then $\CC$ is nowhere dense if and only if $\CC$ is monadically dependent.
\end{fact}

Monadically dependent classes also include all classes of bounded flip-width, by Corollary~\ref{cor:mnip}.

An 
important subfamily of monadically dependent classes consists of monadically stable classes.
A graph class is \emph{monadically stable} if it does not transduce the class of all half-graphs (see Fig.~\ref{fig:half-graph}). 
A graph class is \emph{edge-stable} if it excludes some half-graph as a semi-induced bipartite graph.
More precisely, there is $k\ge 1$ such that 
there do not exist $G\in\CC$ and vertices $a_1,\ldots,a_k,b_1,\ldots,b_k$ of $G$ such that $a_ib_j\in E(G)\iff i<j$ for all $i,j\in\set{1,\ldots,k}$.

The following result is proved in~\cite[Theorem 1.3]{rankwidth-meets-stability}, see also~\cite[Theorem 3.20]{braunfeld2022existential}.
\begin{fact}\label{fact:mstable}
  Let $\CC$ be an edge-stable graph class. 
  Then $\CC$ is monadically stable if and only if it is monadically dependent.
\end{fact}
Monadically stable classes include all nowhere dense classes~\cite{AdlerA14}, as well as transductions of nowhere dense classes, called \emph{structurally nowhere dense} classes. It is not known whether all monadically stable classes are structurally nowhere dense
(this has been conjectured in \cite[Conjecture 6.1]{rankwidth-meets-stability}).
The class of half-graphs is clearly not monadically stable,
and has bounded (linear) clique-width. Hence, monadically stable classes 
are incomparable with classes of bounded clique-width.
They are also incomparable with classes of bounded flip-width,
as witnessed by the class of half-graphs on one side, and any nowhere dense class 
which does not have bounded expansion on the other side.
\smallskip

Monadic dependence can be defined not only for graph classes, but for arbitrary classes of structures, 
e.g. classes of  ordered graphs. A class of ordered graphs is monadically dependent if it does not transduce the class of all graphs, where 
now the transduction
may involve the edge relation symbol, as well as the total order $<$ (and the color predicates).
The following result is proved in \cite{tww4}.
\begin{fact}\label{fact:tww-mnip}
  Let $\CC$ be a class of ordered graphs.
  Then $\CC$ is monadically dependent if and only if $\CC$ has bounded twin-width.
\end{fact}

\medskip
In this section, we study classes of almost bounded flip-width in the three settings discussed above: of weakly sparse classes, of edge-stable classes, and of classes of ordered graphs.
We show that in the first and last settings, those classes coincide with monadically dependent classes, and that in the edge-stable case they also coincide, assuming all monadically stable classes are structurally nowhere dense.

We conjecture that the property  of having almost bounded flip-width coincides exactly with  monadic dependence.

\begin{conjecture}\label{conj:fw-mnip}
  A graph class has almost bounded flip-width if and only if it is monadically dependent.
\end{conjecture}
Currently, we are able to prove neither of the two implications in this conjecture. However, 
in the rest of  Section~\ref{sec:subpoly}, 
we provide evidence towards this conjecture,
by confirming it in restricted settings.

Another conjecture that is supported by the evidence presented below,
predicts a collapse result:
that a bound $o(n^{1/2})$ on the flip-width parameters 
implies a bound $o(n^{\eps})$, for every fixed $\eps>0$.

\begin{conjecture}\label{conj:collapse}
The following conditions are equivalent for a hereditary graph class $\CC$:
\begin{enumerate}
  \item $\CC$ has almost bounded flip-width, that is, for every fixed $r\in\N$ and $\eps>0$ we have that  $\fw_r(G)= o(|G|^\eps)$ holds for all $G\in\CC$,
  \item for every fixed $r\in\N$ we have that $\fw_r(G)= o(|G|^{1/2})$ holds for all $G\in\CC$.
\end{enumerate}
\end{conjecture}
A similar collapse occurs in the sparse case \cite{sparsity-book}.\sz{cite}

\subsection{Weakly sparse classes of almost bounded flip-width}
Analogously to Theorem~\ref{thm:wsparse},
which characterizes classes with bounded expansion as exactly the weakly sparse classes of bounded flip-width,
we get a characterization of nowhere dense classes
in terms of almost bounded flip-width.

\begin{theorem}\label{thm:nd-are-subpoly}Let $\CC$ be a weakly sparse  graph class. Then the following conditions are equivalent:
  \begin{enumerate}
    \item $\CC$ is nowhere dense,
    \item $\CC$ has almost bounded flip-width,
    \item $\CC$ is monadically dependent.
  \end{enumerate}
\end{theorem}
The equivalence of the first and last condition is by Fact~\ref{fact:wsparse-nd}, so we only prove the equivalence of the first two.
\begin{proof}We first show that every nowhere dense class has almost bounded flip-width.
Every nowhere dense class $\CC$ is weakly sparse, 
so excludes some $K_{t,t}$ as a subgraph.
By Theorem~\ref{thm:copw-fw}, Theorem~\ref{thm:adm-wcol},
and the forward implication in Fact~\ref{fact:nd-wcol},
we have that 
$\fw_r(G)\le O_{r,\eps}(n^{t\eps})$
for every $n$-vertex graph $G\in\CC$.
Since this holds for every $\eps>0$ and $t$ is fixed, the conclusion follows by rescaling $\eps$.

Conversely, suppose that $\CC$ has almost bounded flip-width, and excludes $K_{t,t}$ as a subgraph.
Without loss of generality, we may assume that $\CC$ is hereditary.
We have $\fw_1(G)\le O_\eps(n^\eps)$ 
for every $n$-vertex graph $G\in\CC$,
and by Theorem~\ref{thm:fw-deg},
$\deg(G)\le O_\eps(n^\eps\cdot 2t^2)\le O_\eps(n^\eps)$.
By Theorem~\ref{thm:adm-fw}, for every $r\ge 1$ and $\eps>0$ we have 
\begin{multline*}
\tilde\nabla_{r-1}(G)\le (r\cdot \fw_r(G)\cdot \deg(G))^{O(1)}\focs[\\]{}\le 
(r\cdot O_{r,\eps}(|G|^\eps)\cdot O_\eps(|G|^\eps))^{O(1)}\le 
O_{r,\eps}(|G|^{O(\eps)}),
\end{multline*}
for all $G\in\CC$.
By Facts~\ref{fact:wcol-adm} and \ref{fact:adm-nabla}
we have $\wcol_r(G)\le O_{r,\eps}(|G|^{O(\eps)})$ for every $r\ge 1$ and graph $G\in\CC$.
Therefore, $\CC$ is nowhere dense, by the backwards implication in Fact~\ref{fact:nd-wcol}.
\end{proof}
As there exist nowhere dense classes of unbounded degeneracy, we get the following.
\begin{corollary}
  There is a class that has almost bounded flip-width, but does not have bounded flip-width.
  \end{corollary}
  By taking the substitution closure (see Section~\ref{sec:examples}) of the class
  from the corollary above, we obtain a class 
  which has almost bounded flip-width (by Lemma~\ref{lem:substitution}), but is not nowhere dense (not even edge-stable), and has unbounded flip-width.


\subsection{Structurally nowhere dense classes}
We are unable to determine whether classes of almost bounded flip-width are closed under transductions.
Observe that the bound $\fw_r(\phi(G))\le T_q(\fw_{r'}(G))$ in Theorem~\ref{thm:interpretations} is not polynomial in $\fw_{r'}(G)$. It is, however, linear in the case when $\phi$ is a quantifier-free formula $\phi(x,y)$, so we get the following.
\begin{corollary}\label{cor:qf-interp-subpoly}
  Let $\CC$ be a class of $k$-colored colored graphs of almost bounded flip-width, and let $\phi(x,y)$ be a quantifier-free formula. Then the class $\phi(\CC)$ has almost bounded flip-width.
\end{corollary}

Even though we do not know whether classes of almost bounded flip-width are closed under transductions,
we confirm that structurally nowhere dense classes (transductions of nowhere dense classes) have almost bounded flip-width.
\begin{theorem}\label{thm:snd-are-subpoly}
  Every structurally nowhere dense class has almost bounded flip-width.
\end{theorem}
To prove Theorem~\ref{thm:snd-are-subpoly},
we use the main result of~\cite{bushes-lics},
which essentially implies that 
for every structurally nowhere dense class $\CC$
there is an \emph{almost nowhere dense} class $\BB$
of structures equipped with functions,
and a quantifier-free formula $\phi(x,y)$ 
involving function symbols,
such that $\CC\subset \phi(\BB)$.
This is made precise below.

Say that a class $\CC$ of graphs 
is \emph{almost nowhere dense} if 
for every $r\ge 1$ and $\eps>0$ we have  $\wcol_r(G)\le O_{r,\eps}(|G|^\eps)$, for all $G\in\CC$. Crucially, $\CC$ does not need to be hereditary, otherwise 
this notion would coincide with nowhere denseness by Fact~\ref{fact:nd-wcol}.

Fix a signature $\Sigma$ consisting of unary relation symbols, binary relation symbols, and unary function symbols.
The VC-dimension of a $\Sigma$-structure $B$, denoted $\VCdim(B)$, 
is the maximum
of the VC-dimensions of the binary relations of $B$ (see Section~\ref{sec:vc}).
Here, the functions of $B$ are ignored.

The following result, apart from the `moreover' part, is a straightforward consequence of \cite[Theorem 3]{bushes-lics}.
It says that every structurally nowhere dense class $\CC$
interprets in an almost nowhere dense class $\BB$ of structures, 
via a quantifier-free interpretation using a unary function symbol.
Additionally, every $G\in \CC$ interprets in some $B\in \BB$ with $|B|\le O(|G|)$. Moreover, the binary relations of the binary structures in $\BB$ have bounded VC-dimension, which will be important in the next lemma, for controlling the bounds on the flip-width of $G$.

\begin{theorem}[\appmark]\label{thm:quasi-bushes-vc}
  Let $\CC$ be a structurally nowhere dense graph class.
  There is a signature $\Sigma$ consisting of unary and binary relation symbols and one function symbol,
   a class $\BB$ 
  of $\Sigma$-structures
  which is almost nowhere dense, and a quantifier-free symmetric formula $\phi(x,y)$ with the following property. For every graph $G\in\CC$ 
  there is some $B\in \BB$ with $|B|\le O(|G|)$, such that $G$ is an induced subgraph of $\phi(B)$. Moreover, $\VCdim(\BB)<\infty$.
\end{theorem}
The `moreover' part is shown by analysing the construction 
from \cite{bushes-lics} (see \cite{bushes-arxiv} for the full version),
and observing that the constructed quasi-bushes interpret (via a $d$-dimensional interpretation, for some fixed $d$) in graphs from a nowhere dense classes,
and thus have bounded VC-dimension by the results of \cite{AdlerA14}.

The following lemma is extends the ideas used in the proof of Theorem~\ref{thm:interpretations}, specifically, of the special case of quantifier-free interpretations considered in Corollary~\ref{cor:qf-interp-subpoly}. The lemma considers quantifier-free interpretations that use function symbols, and says that under some technical conditions, such interpretations  map graphs of small cop-width to graphs of small flip-width. 

\begin{lemma}[\appmark]\label{lem:fw-qf-vc}
  Let $\Sigma$ be a signature consisting of unary and binary relation symbols, and unary function symbols. Fix $k,r\ge 0$,
  and a symmetric quantifier-free $\Sigma$-formula  $\phi(x,y)$. There are numbers $p\le O_\phi(k)$ and 
  $r'\le O_\phi(r)$ such that the following holds.
  Let $B$ be a $\Sigma$-structure with $\VCdim(B)\le k$ and $G_B$ be its Gaifman graph. Then 
   \[\fw_r(\phi(B))\le O(\copw_{r'}(G_B))^{p}.\]
\end{lemma}

The key insight is that a bound $k$ on the VC-dimension implies that 
if the cops in the Cops and Robber game on $G$ occupy 
a set $S$ of vertices of a graph $G$, then, by the Sauer-Shelah-Perles lemma (Lemma~\ref{lem:sauer-shelah-perles}),
the partition of $V(G)$ into $S$-types 
has size $O(|S|^k)$, and this partition is used by the flipper in the flipper game.

The two statements above are proved in \focs[\cite{flip-width-arxiv}]{Appendix~\ref{app:snd-subpoly}}.
Theorem~\ref{thm:snd-are-subpoly} follows, as we now show.

\begin{proof}[Proof of Theorem~\ref{thm:snd-are-subpoly}]
  Let $\CC$ be a structurally nowhere dense class.
  Without loss of generality, $\CC$ is hereditary.
  Let $\BB$ and $\phi(x,y)$ be as in Theorem~\ref{thm:quasi-bushes-vc}, and $k\in\N$ be such that $\VCdim(B)<k$ for $B\in\BB$.

Let $G\in\CC$ and $B\in \BB$ be such that 
$G$ is an induced subgraph of $\phi(B)$ and $|B|\le O(|G|)$, and let $G_B$ be the Gaifman graph of $B$. Let $p$ and $r'$ be as in Lemma~\ref{lem:fw-qf-vc}.
Fix $\eps>0$. Then we have:
\begin{multline*}
\fw_r(G)\le \fw_r(\phi(B))\le O(\copw_{r'}(G_B))^{p}\focs[\\]{}\le O_{r',\eps}(|B|^{\eps p})\le O_{r',\eps}(|G|^{\eps p}).
\end{multline*}
Since $r'$ depends only on $r$ and $\phi$, and $p$ is a constant, and $\eps>0$ is arbitrary, this proves that $\CC$ has almost bounded flip-width. 
\end{proof}

\subsection{Edge-stable classes of almost bounded flip-width}
Conjecture~\ref{conj:fw-mnip} predicts 
that a graph class has almost bounded flip-width if and only if it is
monadically dependent. 
Currently, we are able to prove neither of the two implications. However, we prove the forward implication under the assumption that $\CC$ is edge-stable 
(that is, excludes some half-graph as a semi-induced bipartite graph).

\begin{theorem}\label{thm:subpoly-mstab}
  Let $\CC$ be an edge-stable, hereditary graph class 
  such that for all $r\in\N$, \[\fw_r(G)= o(|G|^{1/2})\qquad\text{for $G\in\CC$}.\]
  Then $\CC$ is monadically stable.
  In particular, every edge-stable, hereditary graph class of almost bounded flip-width is monadically stable. 
\end{theorem}

As far as we know, every monadically stable class might be structurally nowhere dense. This is conjectured in~\cite[Conjecture 6.1]{rankwidth-meets-stability}. 
If this were true, then Theorem~\ref{thm:subpoly-mstab} and Theorem~\ref{thm:snd-are-subpoly} would imply that 
among edge-stable graph classes, almost bounded flip-width coincides with monadically stable class (and thus with monadically dependent classes, by Fact~\ref{fact:mstable}).
Moreover, this would imply a collapse result, confirming Conjecture~\ref{conj:collapse}.


To prove Theorem~\ref{thm:subpoly-mstab}, 
we use the following  result of~\cite[Thm. 1.4]{flippers}.

\begin{fact}[\cite{flippers}]\label{fact:mstab}
  Let $\CC$ be a hereditary, edge-stable class of graphs.
  If $\CC$ is not monadically stable then there are $r,k\ge 1$,
  a $k$-coloring $\wh\CC$ of $\CC$ and a quantifier-free formula $\phi(x,y)$
such that $\phi(\wh\CC)$ contains the exact $r$-subdivision of every graph.
\end{fact}

\begin{proof}[Proof of Theorem~\ref{thm:subpoly-mstab}]
  Let  $\CC$ be a hereditary, edge-stable graph class, such that 
  for every fixed $r\in\N$, we have that $\fw_r(G)= o(|G|^{1/2})$ holds for all $G\in\CC$. We prove that $\CC$ is monadically stable.
  Suppose otherwise.
  Let $r,k,\wh\CC$ and $\phi(x,y)$ be as in Fact~\ref{fact:mstab}.
  
  Pick a number $n\ge 1$.
  Let $K_n^{(r)}$ denote the exact $r$-subdivision of 
  the clique $K_n$.
  Then $\fw_{r+1}(K_n^{(r)})\ge \Omega_r(n)$ 
  by  Proposition~\ref{prop:subdivisions}.
  By Fact~\ref{fact:mstab}, there is some $G_n\in \wh\CC$ such that $\phi(G_n)$ 
  is isomorphic to $K_n^{(r)}$.
  In particular, $|V(G_n)|=|V(K_n^{(r)})|\le O(rn^2)$.
  

As $\phi$ is quantifier-free and involves $k$ colors, 
  by Theorem~\ref{thm:interpretations}, 
  we have $\fw_{r+1}(\phi(G_n))\le k\cdot \fw_{r+1}(G_n)$, and altogether:
  \begin{multline*}  
  \Omega_r(n)\le \fw_{r+1}(K_n^{(r)})\le \fw_{r+1}(\phi(G_n))\focs[\\]{}\le k\cdot \fw_{r+1}(G_n) = o((n^2)^{1/2})\le o(n),
  \end{multline*}
which is impossible. Hence, $\CC$ is monadically stable.
\end{proof}

\subsection{Classes of ordered graphs of almost bounded flip-width}
Recall that flip-width is defined 
for arbitrary binary relational structures (see \focs[\cite{flip-width-arxiv}]{Appendix~\ref{app:bin-struct}}),
and there is  a variant of flip-width
tailored to ordered graphs, defined in Section~\ref{sec:tww}. It follows from Lemma~\ref{lem:ordered flip-width},
that a class $\CC$ of ordered graphs 
has almost bounded flip-width if and only if 
for every $r$, we have $\fw_r^<(\CC)<\infty$.

In this section, we prove that for hereditary classes of ordered graphs, almost bounded flip-width
coincides with bounded twin-width,
as well as with bounded flip-width (by Theorem~\ref{thm:tww1}).
Moreover, in the setting of ordered graphs, a very strong form of the collapse predicted by Conjecture~\ref{conj:collapse}, holds: from $o(n^{1/2})$ all the way down to $O(1)$ (as opposed to just $n^{o(1)}$, as predicted by the conjecture).
This total collapse is, ultimately, a consequence of the Marcus-Tardos theorem/Stanley-Wilf conjecture.

\begin{theorem}\label{thm:ordered-abfw}
  The following conditions are equivalent for 
  a hereditary class $\CC$ of ordered graphs:
  \begin{enumerate}
    \item $\CC$ is monadically dependent,
    \item $\CC$ has bounded twin-width,
    \item $\CC$ has bounded flip-width,
    \item $\CC$ has almost bounded  
    flip-width,
    \item $\fw_1^<(G)= o(|G|^{1/2})$ for all $G\in\CC$.
  \end{enumerate}
\end{theorem}

The equivalence 1$\leftrightarrow$2 is proved in \cite{tww4}.
The equivalence  2$\leftrightarrow$3 is by Theorem~\ref{thm:tww1}. 
The implications 3$\rightarrow$4 and 4$\rightarrow$5 are immediate. 
We prove the implication 5$\rightarrow$2.
For this, 
we use the following consequence of the main result of \cite{tww4},
which itself is a manifestation of the Marcus-Tardos theorem/Stanley-Wilf conjecture~\cite{marcus-tardos}.

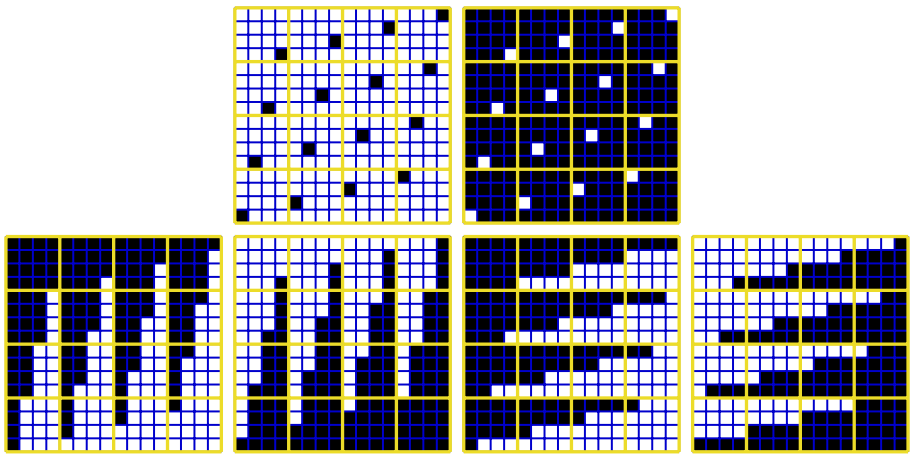
\begin{figure*}[h!]
  \centering
  \begin{tikzpicture}[scale=.179]

 \foreach \symb/\b/\xsh/\ysh in {{==}/0/0/-13,
 {>}/0/-10/-30,
 {<}/0/0/-30,
 {<}/1/10/-30,
 {>}/1/20/-30
 }{
 \begin{scope}[xshift=\xsh*1.7 cm,yshift=\ysh cm]     
  \foreach \i in {0,...,15}{
    \foreach \j in {0,...,15}{
      \pgfmathsetmacro{\ip}{\i+1}
      \pgfmathsetmacro{\jp}{\j+1}
      \pgfmathsetmacro{\col}{ifthenelse(\i == mod(\j,4)*4+floor(\j/4),"black",ifthenelse(\b==1,ifthenelse(\i \symb mod(\j,4)*4+floor(\j/4),"black","white"),ifthenelse(\j \symb mod(\i,4)*4+floor(\i/4),"black","white")))}
      \fill[\col] (\i,\j) -- (\i,\jp) -- (\ip,\jp) -- (\ip,\j) -- cycle;
    }
  }
  \draw[line width=0.75pt, scale=1, color=black!20!blue] (0, 0) grid (16, 16);
  \draw[line width=1.25pt, scale=4, color=black!10!yellow] (0, 0) grid (4, 4);
 \end{scope}
 }
 \foreach \symb/\b/\xsh/\ysh in {{!=}/0/10/-13}{
 \begin{scope}[xshift=\xsh*1.7 cm,yshift=\ysh cm]     
  \foreach \i in {0,...,15}{
    \foreach \j in {0,...,15}{
      \pgfmathsetmacro{\ip}{\i+1}
      \pgfmathsetmacro{\jp}{\j+1}
      \pgfmathsetmacro{\col}{ifthenelse(\i == mod(\j,4)*4+floor(\j/4),"white","black")}
      \fill[\col] (\i,\j) -- (\i,\jp) -- (\ip,\jp) -- (\ip,\j) -- cycle;
    }
  }
  \draw[line width=0.75pt, scale=1, color=black!20!blue] (0, 0) grid (16, 16);
  \draw[line width=1.25pt, scale=4, color=black!10!yellow] (0, 0) grid (4, 4);
 \end{scope}
 
}
  \end{tikzpicture}
  \caption{An $s$-pattern  $A,B$ of order $n^2$ for $n=4$ and $s\in\set{=,\neq,\le_l,\ge_l,\le_r,\ge_r}$,
  represented as incidence matrices of $G[A,B]$.
  The origin is at the lower-left corner and black entries denote incidence.}
  \label{fig:patterns}
\end{figure*}\newcommand{\lex}{\textup{lex}}
Denote $[n]:=\set{1,\ldots,n}$.
Fix a symbol $s\in \set{=,\neq,\le_l,\ge_l,\le_r,\ge_r}$
and let $G=(V,E,<)$ be an ordered graph and $A,B\subset V(G)$.
Say that $A$ and $B$ form a \emph{$s$-pattern} 
of order $n^2$ if 
 $|A|=|B|=n^2$, 
 and the following condition holds (see Fig.~\ref{fig:patterns}).
Let $\alpha\from [n]^2\to A$ 
and $\beta\from [n]^2\to B$ be order-preserving bijections, where $[n]^2$ is ordered by the lexicographic order $<_{\lex}$, while $A$ and $B$ are ordered by $<$.
Then
 for $(i,j),(i',j')\in [n]^2$, we have that 
  $\alpha(i,j)\beta(i',j')\in E$ 
  if and only if the following condition holds:
\begin{itemize}
 \item $(i,j)=(j',i')$, if $s$ is $=$,
\item $(i,j)\neq (j',i')$, if $s$ is $\neq $,
 \item $(j,i)\le_\lex (i',j')$, if $s$ is $\le_l$,
 \item $(j,i)\ge_\lex (i',j')$, if $s$ is $\ge_l$,
  \item $(i,j)\le_\lex (j',i')$, if $s$ is $\le_r$,
  \item $(i,j)\ge_\lex (j',i')$, if $s$ is $\ge_r$.
\end{itemize}

\begin{fact}[\cite{tww4}]\label{fact:tww-patterns}
  Let $\CC$ be a class of ordered graphs of unbounded twin-width. 
  Then there is $s\in \set{=,\neq,\le_l,\ge_l,\le_r,\ge_r}$ such that for every $n\ge 1$,
  there is some graph $G\in \CC$ 
and sets $A$ and $B$ that form an $s$-pattern of order $n^2$.
\end{fact}
\begin{proof}
  Follows from \cite[Lemma 40]{tww4-arxiv}, by considering the $(s,\sigma)$-matching (in their notation), where 
  $\sigma$ 
  is the permutation of $[n]^2$ (ordered lexicographically)
  such that 
  $\sigma((i,j))=(j,i)$ for $(i,j)\in [n]^2$.
\end{proof}

Recall the definition of an $n$-rich division in an ordered graph $G$, and that 
the existence of such a division implies $\fw_1^<(G)\ge n$ 
(see Lemma~\ref{lem:rich} and definition before it).
\begin{lemma}\label{lem:pat-division}Let $n$ be even, and $G$ be an ordered graph containing an $s$-pattern of order $n^2$.
  Then $G$ 
  has an $\frac n 2$-rich division. In particular, $\fw_1^<(G)\ge n/2$.
\end{lemma}
\begin{proof}
Let $A$ and $B$ form an $s$-pattern of order $n^2$.
Let $\cal L=\set{I_1,\ldots,I_{n}}$ be the partition of $V(G)$ into $n$ intervals, with respect to the order of $G$ restricted to $A$, each containing exactly $n$ elements of $A$. Similarly, let $\cal R=\set{J_1,\ldots,J_{n}}$ be the partition of $V(G)$ into $n$ intervals, each containing exactly $n$ elements of $B$. We show that $\cal L$ and $\cal R$ form an $\frac n 2$-rich division of $G$. 

Let  $I_i$ be an interval of $\cal L$ and pick a family $\cal R'\subset \cal R$ of $n/2$ intervals of $\cal R$. Let $\set{j_1,\ldots,j_{n/2}}$
be the indices of the $n/2$ intervals $J_j$ in $\cal R'$ and let $J$ be their union.
We show that there are $n/2$ vertices in $A\cap I_i$ 
with pairwise distinct neighborhoods in $B\cap J$.
This (by symmetry) implies that $\cal L$ and $\cal R$ form and $\frac n 2$-rich division of $G$.

Let $\alpha$ and $\beta$ be as in the definition of an $s$-pattern of order $n^2$.
Then, for $p,q\in [n/2]$ we have that  $\alpha(i,j_p)\in A\cap I_i$ and 
$\beta(j_q,i)\in B\cap J$. Furthermore, 
$\alpha(i,j_p)\beta(j_q,i)\in E(G)$ 
 if and only if:
\begin{itemize}
  \item $p=q$, if $s$ is $=$,
  \item $p\neq q$, if $s$ is $\neq$,
  \item $p\le q$, if $s$ is $\le_l$ or $\le_r$,
  \item $p\ge q$, if $s$ is $\ge_l$ or $\ge_r$.
\end{itemize}

In any case, the vertices
$\alpha(i,j_1),\ldots,\alpha(i,j_{n/2})\in A\cap I_i$
have distinct neighborhoods in $B\cap J$.
This (and a symmetric argument, exchanging the roles of $\cal L$ and $\cal R$) demonstrates that $\cal L$ and $\cal R$ form an $\frac n 2$-rich division.
By Lemma~\ref{lem:rich}, we have that $\fw_1^<(G)\ge n/2$.
\end{proof}


\begin{proof}[Proof of Theorem~\ref{thm:ordered-abfw}]
  As argued, it remains to prove the implication 5$\rightarrow$2.
  Let $\CC$ be a hereditary class of ordered graphs with unbounded twin-width.
  We prove that there are infinitely many ordered graphs $G\in \CC$ 
  with $\fw_1^<(G)\ge \Omega(|G|)^{1/2}$,
  which contradicts condition~5.
  
  Let $s\in \set{=,\neq,\le_l,\ge_l,\le_r,\ge_r}$ be as in Fact~\ref{fact:tww-patterns}.
 
Fix an even $n\ge 2$. By Fact~\ref{fact:tww-patterns}
there is an ordered graph $G_n\in \CC$ 
and sets $A,B\subset V(G_n)$ of size $n^2$,
that form an $s$-pattern of order $n^2$ in $G$.
As $\CC$ is hereditary, we may assume that $V(G_n)=A\cup B$, so in particular $|G_n|\le 2n^2$.
By Lemma~\ref{lem:pat-division},   
$\fw_1^<(G_n)\ge n/2$.
Altogether,
we have that 
$$\fw_1^<(G_{n})\ge \frac 1 2(|G_n|/2)^{1/2}=\Omega(|G_n|^{1/2}).$$
Since $G_n\in\CC$ for all even $n\ge 2$,
it cannot be that 
$\fw_1^<(G)\le o(|G|)^{1/2}$ holds for all $G\in\CC$.
\end{proof}

\paragraph{Summary}
In Section~\ref{sec:subpoly} we have defined graph classes of almost bounded flip-width.
We conjecture that they coincide with monadically dependent graph classes.
We provide the following evidence towards this conjecture.
We have shown that, when restricted to weakly sparse graph classes, almost bounded flip-width coincides with nowhere denseness, and  with monadic dependence.
And when restricted to edge-stable graph classes,
almost bounded flip-width generalizes 
structurally nowhere denseness, and is generalized by monadic stability. As it is conjectured that structurally nowhere dense classes coincide with monadically stable classes, this would imply that 
for edge-stable graph classes,
almost bounded flip-width coincides with monadic stability, and  with monadic dependence.
We have shown that for classes of ordered graphs,
almost bounded flip-width coincides with bounded flip-width, and with bounded twin-width, and therefore, with monadic dependence.

In all the special cases studied above -- of weakly sparse graph classes, of edge-stable graph classes, and of classes of ordered graphs -- the model-checking problem is known to be fixed-parameter tractable \cite{GroheKS17,dreierMS,tww4} (in the edge-stable case, this assumes that monadic stability coincides with structural nowhere denseness).
We therefore obtain the following corollary.

\begin{corollary}\label{cor:abf-mc}
  Let $\CC$ be a class of almost bounded flip-width that is 
  either weakly sparse, or is a class of ordered graphs, 
  or is structurally nowhere dense. Then the model checking problem for first-order logic is fpt on $\CC$.
\end{corollary} 
We remark that, up to best of our knowledge, all hereditary graph classes for 
which the model-checking problem is currently known to be fixed-parameter tractable -- apart from one family of examples --
have almost bounded flip-width. 
Those includes the listed cases above, as well as some classes 
of bounded twin-width \cite{tww8,tww-tournaments}.
However, we do not know whether classes of \emph{structurally bounded local clique-width} \cite{boundedLocalCliquewidth} have almost bounded flip-width (they are monadically dependent, and have tractable model-checking).

\subsection{Almost bounded twin-width}

Say that a graph  class $\CC$ has \emph{almost bounded twin-width}
if for every $\varepsilon>0$ we have \[\tww(G)\le O_{\CC,\varepsilon}(n^{\varepsilon})\qquad\text{for every $n$-vertex graph $G\in\CC$.}\]
We prove the following.
\begin{theorem}\label{thm:abtww-to-abfw}
  Every hereditary class of almost bounded twin-width has alomost bounded flip-width.
\end{theorem}

We first show that classes of 
 almost bounded twin-width have bounded VC-dimension, and in fact, bounded 2-VC-dimension.

\begin{lemma}\label{lem:2-vc-dim-tww}
  Every graph of twin-width $d$ has 2-VC-dimension at most  $O(d^2).$
\end{lemma}

\begin{proof}
This follows by inspecting the proof of 
~\cite[Prop.6.2]{tww2-journal} (in the case $k=1$) which implies that if the $1$-subdivision of $K_m$ has twin-width at most $d$, then $m-1\le (d+1)^2$. Exactly the same proof shows that in fact, if $G$ has 2-VC-dimension $m$ and twin-width at most $d$, then $m-1\le (d+1)^2$, which yields the conclusion.
\end{proof}
\begin{corollary}\label{cor:tww-2vc}
  Let $\CC$ be a hereditary class of almost bounded twin-width.
  Then $\CC$ has bounded 2-VC-dimension.
\end{corollary}
\begin{proof}
  Suppose $\CC$ has unbounded 2-VC-dimension.
  As $\CC$ is hereditary, it follows that for every $m\in\N$ there is a graph $G_m\in\CC$
  such that:
  \begin{itemize}
    \item $G_m$ has 2-VC-dimension (at least) $m$,
    \item $G_m$ has at most $m(m+1)/2$ vertices.
  \end{itemize}
  By Lemma~\ref{lem:2-vc-dim-tww},
  $$\tww(G_m)\ge \Omega(\sqrt m)\ge \Omega(\sqrt[4]{|V(G_m)|}).$$
   It follows that $\CC$ does not have almost bounded twin-width.
\end{proof}
\begin{corollary}\label{cor:abtww-ss}
  Let $\CC$ be a hereditary class of almost bounded twin-width. Then 
  there is a constant $c$ such that 
  for every $G\in\CC$ and $n\in \N$
  we have $\pi_G(n)\le n^c$.
\end{corollary}
\begin{proof}
  As the 2-VC-dimension of a graph $G$ is at least as large as the VC-dimension,
  from Corollary~\ref{cor:tww-2vc} it follows that $\CC$ has bounded VC-dimension.
  The conclusion follows from the Sauer-Shelah-Perles Lemma (Lem.~\ref{lem:sauer-shelah-perles}).
\end{proof}

We are ready to prove Theorem~\ref{thm:abtww-to-abfw}.
\begin{proof}[Proof of Thm.~\ref{thm:abtww-to-abfw}]
  Let $\CC$ be a hereditary class of almost bounded twin-width.
Let $c$ be the constant from Corollary~\ref{cor:abtww-ss},
so that $\pi_G(n)\le n^c$ for every $G\in\CC$ and $n\in\N$.
Fix $\eps>0$ and $r\in\N$.
As $\CC$ has almost bounded twin-width,
for very $n$-vertex graph $G\in\CC$
we have that: 
$$\tww(G)\le O_{\CC,\eps,r}(n^{\eps/(rc)}).$$
By Theorem~\ref{thm:btww} we have that:
$$
\fw_r(G)\le \pi_G(\tww(G)^{O(r)})\le O(\tww(G))^{O(rc)}\le O_{\CC,\eps,r}(n^ {O(\eps)}).$$
The conclusion follows by rescaling $\eps$ proportionally to the constant in the $O(\eps)$ notation (which is a universal constant  originating from the exponent $O(r)$ from Theorem~\ref{thm:btww}).
\end{proof}

\section{Discussion}\label{sec:discussion}
In this section, we discuss some conjectured relationships between various 
notions defined in this paper, and other known notions.
We speculate on possible routes towards proving some conjectures.

\subsection{Obstructions to small flip-width}
The results of this paper 
strongly suggest that the flip-width parameters are the sought dense analogues 
of generalized coloring numbers, that classes of bounded flip-width are the sought dense analogues of classes with bounded expansion,
and almost bounded flip-width is a dense counterpart of nowhere dense classes.
What is currently missing to complete this picture, is an analogue 
of the key result of Sparsity theory, which is a min-max theorem that relates, on one hand,
 explicit descriptions of winning strategies of the robber,  
 and on the other hand, explicit descriptions of winning strategies 
for the cops in the Cops and Robber game (see Section~\ref{sec:copwidth}).
In the sparse case, the former are
 obstructions to having small $r$-admissibility numbers,
 or (by Fact~\ref{fact:adm-nabla}),  ${\le}r$-subdivisions of graphs with large minimum degree.
 Finding an analogous notion in the dense case seems to be a major challenge. It seems plausible that finding such notions is related to the question of efficiently approximating 
 the flip-width parameters. This goal can be formalized as follows.
 \begin{goal}\label{goal:apx}
   Obtain an fpt approximation algorithm for radius-$r$ flip-width:
   an algorithm running in time $f(r,k)\cdot n^c$, 
   for some constant $c$
   and function $f\from\N\times\N\to\N$, which given an $n$-vertex graph $G$ and numbers $r,k$, 
    either concludes that $\fw_r(G)>k$,
   or that $\fw_r(G)< f(r,k)$.
 \end{goal}
 A first step in this direction is achieved by Theorem~\ref{thm:apx}, which achieves an XP approximation algorithm, rather than an fpt algorithm.

 A related, but less concrete goal is the following.
\begin{goal}\label{goal:weak obstructions}
  Describe explicit \emph{forbidden weak obstructions at depth $r$ and density $\ell$}, such that 
  there are functions $f,g\from \N\times \N\to\N$ such that for all $r,\ell\ge 1$,
every graph $G$ with $\fw_r(G)\ge f(r,\ell)$ contains 
such an obstruction of density $\ell$ as an induced subgraph,
and conversely, no graph $G$ which contains such an obstruction of density $g(r,\ell)$ as an induced subgraph satisfies $\fw_r(G)\ge \ell$.
\end{goal}

One attempt at formalizing what we mean by ``explicit'',
is by requiring that the following holds:
for every fixed $r$ there is a formula $\phi_r(x,y)$ (possibly involving colors) and a function $f_r\from\N\to\N$ 
such that if $G$ is is a forbidden weak obstruction at depth $r$ 
and density at least $f_r(\ell)$ 
then there is some $k$-coloring of $G$ such that $\phi_r(G)$ 
contains an $r$-subdivision of a graph with minimum degree at 
least $\ell$ as an induced subgraph.

The following question suggests a path towards achieving Goal~\ref{goal:weak obstructions}.

\begin{question}\label{q:hideouts}
  Fix $r\ge1$. Is it the case that $\fw_r(\CC)<\infty$ if and only if 
there is a $k\in\N$ such that no graph $G\in\CC$ 
contains a $(r,k,k)$-hideout (see Definition~\ref{def:hideout})?
\end{question}
Hideouts are not induced subgraphs, as posited in Goal~\ref{goal:weak obstructions},
but this could serve as a starting point.

\medskip
The following notions are defined in~\cite{stabletww-arxiv}.
A \emph{transduction ideal} is a property of graph classes that is preserved under first-order transductions:
if a class $\CC$ has the considered property, and $\DD$ transduces in $\CC$, then $\DD$ has the property, too.
The \emph{dense analogue of bounded expansion}
is the largest transduction ideal such that all weakly 
sparse classes that belong to it have bounded expansion.

By Corollary~\ref{cor:transductions}, classes of bounded flip-width form a transduction ideal, and by Theorem~\ref{thm:wsparse},
they are contained in the dense analogue of bounded expansion.
We conjecture that bounded flip-width is exactly the dense analogue of bounded expansion, according to the above definition.
This is equivalently stated as follows.

\begin{conjecture}\label{conj:dense analogue}
  Let $\CC$ be a class that does not have bounded flip-width. Then $\CC$ transduces a weakly sparse class $\DD$ which does not have bounded expansion.
\end{conjecture}

A road towards proving Conjecture~\ref{conj:dense analogue}
is through Goal~\ref{goal:weak obstructions}.
Indeed, if $\CC$ contains obstructions at a fixed depth $r\ge 1$ and of arbitrarily large density $\ell$,
then according to the requirement given below Goal~\ref{goal:weak obstructions}, $\CC$ transduces a class $\DD$ which contains $r$-subdivisions of arbitrarily dense graphs, and therefore $\DD$ does not have bounded expansion.

\subsection{Model checking}%
Recall that Conjecture~\ref{conj:mnip-mc}
predicts that for a hereditary class $\CC$,
first-order model-checking is fixed-parameter tractable (fpt) on $\CC$
if and only if $\CC$ is monadically dependent.
Both implications are open.
We now discuss approaching the backward implication,
giving the upper bound: that model-checking is fpt on every monadically dependent graph class.

The question whether the model checking problem is fpt on every class $\CC$ of graphs that has bounded twin-width, remains open.
(The result of~\cite{tww1} says that this is the case, if the input graph $G$ is given together with its contraction sequence. 
In some special cases \cite{tww4,tww8,tww-tournaments}, such a contraction sequence can be efficiently computed.)
As bounded twin-width implies bounded flip-width, we do not know whether the model checking problem  is fpt on every class $\CC$ of graphs that has bounded flip-width. However, we believe that approximating the radius-$r$ flip-width of a given graph $G$, might be easier than approximating its twin-width.
This is indicated for example by Theorem~\ref{thm:apx}, which gives an XP approximation algorithm for flip-width, while no such algorithm for twin-width is known.
Still, an analogue of the model-checking result of \cite{tww1} is missing, and is posed as the following goal.

\begin{goal}\label{goal:mc}
  Devise an efficient representation of winning strategies for the flipper such that for a fixed formula $\phi$ there is a number $r$, such that 
  given an efficient representation 
of a winning strategy for the flipper in the radius-$r$ flipper game on a given graph $G$, one can efficiently check whether $\phi$ holds in $G$.
\end{goal}

 In particular, Goal~\ref{goal:mc} combined with Goal~\ref{goal:apx} could allow to solve the model-checking problem on classes of bounded flip-width
 (and also on classes of bounded twin-width, overcoming  
the lack of an fpt approximation algorithm for twin-width).

An extension of Goal~\ref{goal:mc} to monadically dependent classes,
with the hope of confirming Conjecture~\ref{conj:mnip-mc},
could lead 
through Conjecture~\ref{conj:fw-mnip},
which characterizes the hereditary monadically dependent classes as those with almost bounded flip-width.
This, however, remains very speculative, and most likely requiring further  insights, on top of the ones needed to achieve Goal~\ref{goal:mc} and to solve Conjecture~\ref{conj:fw-mnip}.


\subsection{Stable classes of bounded flip-width}

It is conjectured \cite[Conjecture 6]{stabletww-arxiv} that the edge-stable classes in the dense analogue of bounded expansion have structurally bounded expansion, i.e., are transductions of classes with bounded expansion. 
As classes of bounded flip-width are contained in the dense analogue of bounded expansion, this would imply the following.

\begin{conjecture}\label{conj:stable}
  Every edge-stable class with bounded flip-width
  has structurally bounded expansion.
\end{conjecture}
Conjecture~\ref{conj:stable} would generalize the main result of \cite{stable-tww-lics,stabletww-arxiv},
which states that every edge-stable class of bounded twin-width has structurally bounded expansion.
The conjunction of Conjecture~\ref{conj:dense analogue} 
and Conjecture~\ref{conj:stable} implies the following duality statement for 
edge-stable graph classes.
\begin{conjecture}\label{conj:duality-stable}
  For every edge-stable graph class $\CC$, 
either $\CC$ is a transduction of a class with bounded expansion,
or $\CC$ transduces a weakly sparse class with unbounded expansion.
\end{conjecture}

 Conjecture~\ref{conj:dense analogue} 
and Conjecture~\ref{conj:stable} imply Conjecture~\ref{conj:duality-stable}:
  If $\CC$ is not monadically stable, then $\CC$ is not monadically dependent (by Fact~\ref{fact:mstable}), and hence transduces every weakly sparse class of graphs.
  If $\CC$ is monadically stable and has unbounded flip-width, then $\CC$ transduces a weakly sparse class that does not have bounded expansion by Conjecture~\ref{conj:dense analogue}.
  Finally, if $\CC$ is monadically stable and has bounded flip-width, then $\CC$ transduces in a class of bounded expansion, by Conjecture~\ref{conj:stable}.

\subsection{Restrictions of the flipper game}
Restricted classes of bounded flip-width 
can be defined by imposing additional constraints on the
flipper game. We consider the following restrictions \anonym{\footnote{I thank Rose McCarty and Pierre Ohlmann for suggesting the positional and bounded depth variants.}}
on the strategy of the flipper in the flipper game.
We say that the strategy of the flipper is:
\begin{description}
  \item[blind] if their move does not depend on the current position of  the runner 
  (but may depend on the graph and on the round number);
  \item[positional] if their move
  is only based on the current position of the runner (and not on the past moves, nor the round number); 
  \item [of bounded depth] if 
  there is a bound $\ell$ such that the flipper wins within $\ell$ rounds;
  \item [branching-blind] if 
  in each round move, the flipper proposes a partition $\cal P$ of the vertex set,
  and bases his move only on the part of $\cal P$ which is occupied by the runner, and furthermore, ensure that the runner will remain in that part until the end of the game.
\end{description}
Each of those variants of the game 
defines a variant of the flip-width parameter, and 
the related classes. Thus,
a class $\CC$ has bounded $X$ flip-width,
for $X\in\set{\textit{blind, branching blind, positional}}$,
if for every radius $r$, there is a $k$ 
such that the flipper 
wins the flipper game of radius $r$ and width $k$ on every graph $G\in\CC$, using a strategy with property $X$. Similarly, $\CC$ has \emph{bounded flip-depth} 
if for every $r\in\N$ there are $k,\ell$ 
 such that for every $G\in\CC$, the flipper has a  strategy
 in the flipper game of radius $r$ and width $k$, that wins in at most $\ell$ rounds.

 We may also consider the limit case of the above games, for $r=\infty$, thus 
 defining classes with \emph{blind $\infty$-flip-width},
 \emph{positional $\infty$-flip-width}, \emph{$\infty$-flip-depth},
 \emph{branching-blind $\infty$-flip-width}, analogously as above.

 It turns out that these notions relate to notions and conjectures
 that have been studied earlier, as we discuss below.
 We start with the following observation.

 \begin{proposition}\label{prop:strategies-transductions}
  All the above properties are preserved under first-order transductions,
  and under $\cmso$ transductions for the $\infty$ variants.
 \end{proposition}
 Proposition~\ref{prop:strategies-transductions} follows by observing that in the proof of Theorem~\ref{thm:interpretations} and Theorem~\ref{thm:interpCMSO}, the strategy on $\phi(G)$ is obtained by transferring a winning strategy from $G$ (see Section~\ref{sec:transfer}). And transferring a strategy with one of the listed properties, results in a strategy with the same property.

 \paragraph{Infinite radius}
 In the limit case of radius $\infty$, the picture 
 is quite well-understood, thanks to the following.
 \begin{theorem}\label{thm:infty-classes}
  Let $\CC$ be a class of graphs. Then:
  \begin{enumerate}
    \item $\CC$ has branching-blind $\infty$-flip-width if and only if 
    $\CC$ has bounded clique-width;
    \item $\CC$ has blind $\infty$-flip-width if and only if 
    $\CC$ has bounded linear clique-width;
    \item  $\CC$ has bounded positional $\infty$-flip-width, if and only if 
    $\CC$ has bounded $\infty$-flip-depth, 
    if and only if $\CC$ has bounded shrubdepth.
  \end{enumerate}
 \end{theorem}
A class $\CC$ has \emph{bounded shrubdepth}~\cite{shrubdepth-mfcs}
if and only if it transduces in a class 
of trees of depth bounded by a constant.


   
\begin{proof}[Sketch of proof]
  (1).
  For the forward implication,  observe that the strategy presented in the proof of Theorem~\ref{thm:cw} (see \focs[\cite{flip-width-arxiv}]{Appendix~\ref{app:cw}}) is branching blind. 
  For the backward implication, note that if $\CC$ has unbounded clique-width then 
  by Theorem~\ref{thm:cw} it has unbounded flip-width,
  so in particular, it has unbounded branching-blind flip-width.

(2).
For the forward implication, observe that the strategy presented in Example~\ref{ex:half-graphs},
for half-graphs,
is blind. Therefore, the class of half-graphs has bounded blind flip-width.
It is well-known that every class $\CC$ of bounded linear clique-width is a $\cmso$ transduction of the class of half-graphs (or of the class of finite total orders). By Proposition~\ref{prop:strategies-transductions}, $\CC$ has bounded blind flip-width.

Conversely, if a class $\CC$  has unbounded linear clique-width, then it $\cmso$ transduces the class of all trees (this follows from 
\cite{pathwidth2Conn,
pathwidth2ConnDT} and from \cite{courcelle-oum-vertex-minors-seese}).
One can prove\sz{do it} that 
the class of trees does not have bounded blind $\infty$-flip-width.
Again by Proposition~\ref{prop:strategies-transductions}, this implies that $\CC$ has unbounded blind $\infty$-flip-width.

(3). First observe that in a tree of depth at most $d$, the flipper has a positional winning strategy in the radius-$\infty$ flipper game, that wins in at most $d$ rounds (this is essentially the same as the strategy in Example~\ref{ex:trees}).
Hence, every class $\DD$ of trees of depth bounded by a constant 
has bounded positional $\infty$-flip-width, 
and bounded $\infty$-flip-depth.
By Proposition~\ref{prop:strategies-transductions},
the same holds for every class that transduces in $\DD$, hence for all classes of bounded shrubdepth.

Conversely, if a class $\CC$ has unbounded shrubdepth, then it $\cmso$-transduces the class 
of all paths, by \cite{shrubdepth-cmso-paths}.
One can prove\sz{do it} that the class of all paths does not have bounded positional $\infty$-flip-width, and does not have bounded $\infty$-flip-depth.  Proposition~\ref{prop:strategies-transductions} implies that 
$\CC$ has unbounded positional $\infty$-flip-width, and 
unbounded $\infty$-flip-depth.
\end{proof}

\paragraph{Blind case}We now move the case of classes of bounded blind flip-depth. We start with some examples.
 \begin{example}
  Classes of bounded blind flip-width include 
  classes of bounded pathwidth, and more generally, classes of bounded linear rank-width (or linear clique-width), by Theorem~\ref{thm:infty-classes}.
 \end{example}

 \begin{example}
  On the other hand, the strategies described in Example~\ref{ex:trees} and in Example~\ref{ex:comp-trees}, for trees and for comparability graphs of trees, respectively, are not blind. One can show that 
 those classes do not have bounded blind flip-width.
 A more general statement is given below.
 \end{example}

 \begin{conjecture}\label{conj:blind}
  A class $\CC$ has bounded blind flip-width if and only if $\CC$ has bounded linear clique-width.
\end{conjecture}
As bounded linear clique-width is equivalent to bounded blind $\infty$-flip-width by Theorem~\ref{thm:infty-classes},
Conjecture~\ref{conj:blind} states that bounded blind flip-width collapses to bounded blind $\infty$-flip-width.
 
\paragraph{Branching-blind case}
We now discuss the branching-blind case.
By Theorem~\ref{thm:infty-classes},
every class of bounded clique-width has bounded branching-blind flip-width.


 \begin{conjecture}\label{conj:branching-blind}
  A class $\CC$ has bounded branching-blind flip-width if and only if $\CC$ has bounded clique-width.
\end{conjecture}
As bounded clique-width is equivalent to bounded branching-blind $\infty$-flip-width by Theorem~\ref{thm:infty-classes},
 Conjecture~\ref{conj:branching-blind} states that bounded branching-blind flip-width collapses to bounded branching-blind $\infty$-flip-width.

\paragraph{Positional case}
 Let us now look at some examples related to bounded positional flip-width.
 \begin{example}
  Every class with bounded expansion has bounded positional flip-width. This follows from the proof of the upper bound in  Theorem~\ref{thm:adm-wcol}, as the cops' strategy in the radius-$r$ Cops and Robber game 
  is to occupy (or isolate) the vertices that are $2r$-weakly reachable from the vertex that is currently occupied by the robber. Clearly, this is a positional strategy.
 \end{example}
By Proposition~\ref{prop:strategies-transductions}
we get the following.

\begin{corollary}
  Every class with structurally bounded expansion has bounded positional flip-width.
\end{corollary}

On the other hand, we have:
 \begin{example}
 Half-graphs have unbounded positional flip-width.
 \end{example}
It follows that every class $\CC$ with bounded positional flip-width
is edge-stable. Since $\CC$ is also monadically dependent by Corollary~\ref{cor:mnip}, and every monadically dependent, edge-stable class is monadically stable by Fact~\ref{fact:mstable}, we get the following.
\begin{corollary}
  Every class with bounded positional flip-width is monadically stable.
\end{corollary}
Conjecture~\ref{conj:stable} would therefore imply the following
characterization of classes with bounded positional flip-width.
\begin{conjecture}\label{conj:positional}
  A class $\CC$ has bounded positional flip-width if and only if $\CC$ has structurally bounded expansion.
\end{conjecture}

\paragraph{Bounded flip-depth}
Finally, we look at classes of bounded flip-depth.
\begin{example}
  Classes of bounded degree have bounded flip-depth, as seen in Example~\ref{ex:bd-deg}. Also, classes of bounded treedepth, or more generally, classes of bounded shrubdepth, have bounded flip-depth.
\end{example}
A class of examples combining the above is provided by the following notion.
Fix $d,h\ge 1$.
A \emph{hybrid tree} of depth $h$ and degree $d$
is a graph $G$ that can be obtained from a rooted tree $T$ of depth  $h$ by:
 \begin{itemize}
  \item first adding some edges connecting siblings in $T$, in such a way that every vertex is adjacent in $G$ to at most $d$ of its siblings in $T$,
 \item afterwards, taking a subgraph of the resulting graph.
 \end{itemize}

\begin{proposition}Fix $d,h\ge 1$.
  The class $\cal H_{d,h}$ of hybrid trees of depth $h$ and degree $d$
  has bounded flip-depth. 
\end{proposition}

The following is a consequence of Simon's factorisation theorem \cite{SIMON199065}. We omit the details.

\begin{theorem}
  Every class $\CC$ of bounded pathwidth, and more generally,
  every stable class of bounded linear clique-width, 
  transduces in $\cal H_{2,h}$, for some fixed $h$.
\end{theorem}

\begin{corollary}
  Every class of bounded pathwidth, and more generally, every stable class of bounded linear clique-width has bounded flip-depth.
\end{corollary}

We now pose two conjectures characterizing classes of bounded flip-depth in complementary ways.

Define the \emph{tree-rank} of a graph class $\CC$ as the largest number $d\in\N$ such that 
 $\CC$ transduces the class of all forests of depth $d$, or $\infty$ if the largest such number does not exist.
Note that classes of rank $d\in\N$ are in particular monadically dependent, in fact, monadically stable, since the class of half-graphs 
transduces every class of forests of bounded depth.





Next, the following conjecture relates the above notion to bounded flip-depth.

 \begin{conjecture}\label{conj:flip-depth}
  A class $\CC$ has bounded flip-depth if and only if $\CC$ 
  has rank ${<}d$, for some $d\in\N$.
 \end{conjecture}

\medskip
We finish with the following conjecture, relating  classes of bounded flip-width with the complexity of the model-checking problem for first-order logic.

Say that a graph class $\CC$ has  \emph{non-uniform elementary-fpt} first-order model-checking 
if there are numbers $h,c\ge 1$ such that for every first-order formula $\phi$ of length $k$ there is an algorithm
which determines if a given $n$-vertex graph $G\in\CC$ satisfies $\phi$ 
in time at most 
\[\underbrace{2^{2^{.^{.{^{.^{2^k}}}}}}}_{{\rm height\,} h}\cdot n^c.\]
In the \emph{uniform} variant, there is a single algorithm,
which inputs $\phi$ and $G\in\CC$, 
and has the above running time.

It is known~\cite{model-theory-makes-formulas-large} that every class of bounded degree has uniform elementary-fpt model-checking.
Furthermore, the class of trees does not have uniform elementary-fpt 
model-checking, unless AW[$*$]=FPT, by a result of Frick and Grohe~\cite{FRICK20043} (see also~\cite{model-theory-makes-formulas-large}),
and the class of colored linear orders does not have uniform elementary-fpt model-checking unless P=NP, also by Frick and Grohe.

\anonym[We pose the following.]{Together with Micha{\l} Pilipczuk, we pose the following.}
\begin{conjecture}
  A class $\CC$ of graphs has non-uniform elementary-fpt first-order model-checking if and only if $\CC$ does not transduce the class of forests of depth $d$, for some $d\ge 1$.
\end{conjecture}
Together with Conjecture~\ref{conj:flip-depth},
this would characterize classes with non-uniform elementary-fpt first-order model-checking precisely as those with bounded flip-depth.

\clearpage
\begin{appendices}
\section{Variants of the cop-width parameter}

\subsection{A variant without announced moves}
\label{sec:copwidth'}
Fix parameters $k,r\in\N$,
and consider the variant of the Cops and Robber game in which 
there are $k$ cops, and a robber 
with speed $r$. In each round of this game, first the cops pick a set $A$ of  $k$ chosen vertices of the graph,
and then the robber may either stay in his last position $v$ -- but only if $v\not\in A$ -- or moves to any vertex $u$ via a path $v=v_0,\ldots,v_i=u$ of length $1\le i\le r$ such that $v_1,\ldots,v_i\notin A$. If they cannot do so, the cops win the game, and if the robber can evade the cops forever, then the robber wins.
Denote the smallest number $k$ for which the cops have a winning strategy on a graph $G$ by $\copw_r'(G)$. 

Those parameters essentially appear in the work~\cite{richerby-thilikos-lazy-fugitive, otherthilikos}. 
The paper~\cite{richerby-thilikos-lazy-fugitive} considers a variant of the game in which the robber is \emph{lazy}, that is, does not move unless
a cop is placed at his location, whereas the~\cite{otherthilikos} considers a variant where the cops occupy edges instead of vertices, and the robber never remains put. Analogues of the next notion and lemma also appear in those papers.

Call a set $U$ of vertices of a graph $G$ 
a $(k,r)$-\emph{hideout}\sz{hideouts}
if for every $v\in U$ and set $A\subset V(G)\setminus\set{v}$ 
with $|A|<k$, there is some path from $v$ to $U\setminus \set v$ 
of length at most $r$ in $G\setminus A$.

\begin{lemma}\label{lem:copw'}
  Fix numbers $k,r\in\N$ and a graph $G$.
  The following conditions are equivalent:
  \begin{enumerate}
    \item $\copw_r'(G)\le k$,
    \item $G$ has no $(k,r)$-hideout,
    \item there is a total order on $V(G)$ such that 
    for every $v\in V$ there is some set $A\subset V(G)\setminus \set{v}$
with  $|A|< k$ such that there is no path of length ${\le}r$ from $v$ to any vertex $w<v$
in $G\setminus A$.
  \end{enumerate}
\end{lemma}
\begin{proof}
  (1$\rightarrow$2) We show that if $U$ is a $(k,r)$-hideout in $G$,
  then there is a winning strategy for the robber 
in the game of radius $r$ and width $k$ corresponding to the $\copw_r'$ parameter.
As a first move, the robber picks an arbitrary vertex $v\in U$,
and we show that they may always remain in the set $U$.
In each round, when the cops place the cops on a set $A$ of at most $k$ vertices, 
then robber stays put if his current position $v$ is not in $A$,
and otherwise, as $|A\setminus \set v|< k$,  the robber moves to any vertex $u\in U\setminus A$ that is connected 
by a path of length at most $r$ from $v$ in $G\setminus A$.

(2$\rightarrow$3) Suppose that $G$ has no $(k,r)$-hideout.
Start with $U=V(G)$ and $\tup w$ being the empty sequence.
As long as $U$ is nonempty, 
pick any vertex  $u\in U$ such that there is some set $A\subset V(G)\setminus\set{v}$ with $|A|< k$ such that there is no path of length ${\le}r$ from $v$
to any vertex $w \in U$ in $G\setminus A$. Such a vertex exists, since $U$ 
is not a $(k,s)$-hideout.
Remove $u$ from $U$ and prepend it to $\tup w$, and repeat.
Once $U$ becomes empty, the sequence $\tup w$ gives a total order on $V(G)$ satisfying the required condition.

(3$\rightarrow$1) We show how to turn a total order as in condition (3)
into a winning strategy for the cops in the Cops and Robber game of radius $r$ and width $w$.
When the robber is occupying a vertex $v$,
the cops pick any set $A\subset V(G)\setminus\set v$ with $|A|< k$ such that there is no path of lenght ${\le}r$ from $v$ to any vertex $w<v$,
and place the cops on all the vertices of $A\cup\set{v}$.
Then the robber needs to move right in the order, so eventually they will lose.
\end{proof}

We now relate the parameter $\copw'_r$ to generalized coloring numbers.
The \emph{$r$-strong coloring number} of a graph $G$, denoted $\scol_r(G)$ is the smallest number $k$ with the following property. 
There is a total order $<$ on $V(G)$ such that every for vertex $v$,
there are at most $k$ vertices $w<v$ that can be reached from $v$ by a path $\pi$ of length at most $r$,
such that $u>w$ for all internal vertices $u$ of $\pi$.
Clearly, $\adm_r(G)\le\scol_r(G)\le \wcol_r(G)$.

\begin{lemma}\label{lem:adm-copw'-scol}
  For every $r\in\N$ and graph $G$
  the following inequalities hold:
  \[\adm_r(G)+1 \le \copw_r'(G)\le \scol_r(G)+1.\]
\end{lemma}
\begin{proof}  
  We prove the first inequality, by 
  showing that  $\adm_r(G)\ge k-1$ implies $\copw_r'(G)>k$. Indeed, by $\adm_r(G)\ge k-1$ there is a set $U\subset V(G)$
  such that for every $v\in U$ there are $k$ paths of length at most $r$ from $v$ to $U\setminus v$,
  which are vertex-disjoint apart from $v$.
  In particular, no set $A\subset V(G)\setminus \set v$ with $|A|<k$ 
  hits all of those $k$ paths. Hence, $U$ is a $(k,r)$-hideout,
  and $\copw_r'(G)>k$ by Lemma~\ref{lem:copw'}.

  For the second inequality, suppose $\scol_r(G)\le k$ and let $<$ be a total order witnessing it. Then condition 3 of Lemma~\ref{lem:copw'} holds.
\end{proof}

Also note that $\copw_r'(G)\le \copw_r(G)$, since a winning strategy in the game corresponding to $\copw_r$ is also a winning strategy in the game corresponding to $\copw_r'$.

By Lemma~\ref{lem:adm-copw'-scol}, we have the following.
\begin{corollary}\label{cor:oh,thilikos}
  The following conditions are equivalent for a graph class $\CC$:
  \begin{enumerate}
    \item $\CC$ has bounded expansion,
    \item $\copw_r(\CC)<\infty$, for every $r\in\N$,
    \item $\copw_r'(\CC)<\infty$, for every $r\in\N$.
  \end{enumerate}
\end{corollary}

\subsection{Isolation game}
Consider the following variant of the flipper game.
The \emph{isolation-width} game with radius $r\in\N\cup\set{\infty}$ 
and width $k\in\N$, $k\ge 1$, is played on a graph $G$.
In round $i$ of the game we have a set $S_i\subset V(G)$
with $|S_i|\le k$, which are the new positions of the cops declared by the cops, and the current position $v_i\in V(G)$ of the robber. Initially, $S_0=\emptyset$ and $v_0$ is a vertex of $G$ chosen by the robber. In round $i>0$,
the cops announce a set $S_i\subset V(G)$ 
of next positions of the cops
with $|S_i|\le k$, that will be put into effect momentarily.
The robber, knowing $S_i$, moves to a new vertex $v_i$ by following a path of length at most $r$ from $v_{i-1}$ to $v_i$ that avoids 
the \emph{previous} cop positions $S_{i-1}$. The game terminates when $v_i\in S_i$.
Write $\iw_r(G)$ for the smallest number $k$ 
such that the cops have a winning strategy in the isolation game with radius $r$ and width $k$.

\begin{lemma}\label{lem:cw-iw}
  \[\iw_r(G)\le \copw_r(G)\le 2\iw_r(G).\]
\end{lemma}
\begin{proof}[Proof sketch]
  A winning strategy in the Cops and Robber game of radius $r$ and width $k$ can be translated into a winning strategy in the isolation game of radius $r$ and width~$k$: when 
  in the Cops and Robber game the cops are directed to their new set of  positions $X$, in the isolation game the cops define $X$ as their next positions.
  In particular, $|X|\le k$, and a response of the robber in the isolation game is a valid response in the Cops and Robber game.
  
  Conversely, a winning strategy in the isolation game of radius $r$ and width $k$ can be translated into a winning strategy in the Cops and Robber game of radius $r$ and width~$2k$: when in the isolation game
  the cops declare the new set $X$ of cops positions and $Y$ is the previous set of positions, in the Cops and Robber game the cops define $X\cup Y$ as the next positions of the cops.
  In particular, $|X\cup Y|\le 2k$, and a response of the robber in the Cops and Robber game is a valid response in the isolation game.
\end{proof}

\section{Flip-width}\label{app:variants}

\subsection{Bipartite variants}\label{app:bfw}
Let $G$ be a bipartite graph.
We define the parameter $\bfw_r(G)$ analogously to $\fw_r(G)$,
but the flips played by the flipper are now \emph{bipartite} flips:
flips between two subsets of opposite parts of $G$.
We only consider partitions of $V(G)$ that refine the bipartition of $V(G)$, and measure its size by the maximum,
over the two parts of $G$, of the number of parts of $\cal P$
that are contained in it.
By definition, we have the following.
\begin{lemma}\label{lem:fw-bfw}
  Let $G$ be a bipartite graph and $r\in\N\cup\infty$.
  Then $\fw_r(G)\le 2\bfw_r(G)$.
\end{lemma}

Let $G$ be a graph and $X,Y\subset V(G)$,
  and let $G[X,Y]$ be the bipartite graph semi-induced by $X$ and $Y$ in $G$,
  with parts $X$ and $Y$ and edges $xy$ such that $x\in X,y\in Y,xy\in E(G)$.
\begin{lemma}\label{lem:fw1-bip}
  Let $G$ be a graph and $X,Y\subset V(G)$,
Then $\bfw_r(G[X,Y])\le \fw_r(G)$.
\end{lemma}
\begin{proof}
  We convert a winning strategy for the flipper 
  on $G$ to a winning strategy on $H:=G[X,Y]$.
  If the flipper plays a $k$-partite flip $G'$ of $G$,
  and $\cal P$ is the corresponding partition of $V(G)$,
  then consider the bi-partition $\cal Q$ of $X\uplus Y$
   $\cal Q:=\setof{P\cap X,P\cap Y}{P\in\cal P}$,
   and the $\cal Q$-flip $H'$ of $H$,
   in which two parts of $R,S\in \cal Q$ are flipped if and only if  the two unique parts of $A,B\in \cal P$ such that $A\cap X=R$ and $B\cap Y=S$, 
   are flipped in the $\cal P$-flip producing $G'$ from $G$.
   The key property of this construction is that any path in $H'$, starting at some vertex $u$ and ending at a vertex $v$, determines a path in $G'$ of the same length,
   starting at (a copy of) $u$ and ending at (a copy of) $v$.
   In particular, if the runner moves to a vertex $v'$ 
   along a path of length at most $r$ 
   in the previous flip of $H$, then this induces 
   a path of length at most $r$ in the previous flip of $G$.
   Hence, if the flipper wins in $G$, then they also win in $H$.
\end{proof}

\subsection{Excluding a $K_{t,t}$}
\begin{lemma}\label{lem:Ktt-complexity}
  Fix $t> 1$ and a graph $G$ that excludes $K_{t,t}$ as a subgraph. Let 
  $S\subset V(G)$ and let $\cal P_S$ be the partition of $V(G)$
  that partitions $S$ into singletons and vertices in $v\in V(G)\setminus S$ according to $N(v)\cap S$.
  Then $|\cal P_S|\le |S|^t$ if $t\ge 3$,
  and $|\cal P_S|\le O(|S|^t)$ if $t\ge 2$.
\end{lemma} 
\begin{proof}
  Fix a set $S\subset V(G)$ and let $s=|S|$. 

  To every vertex $v\in V(G)$ with $|N(v)\cap S|\ge t$ 
  assign any set $A\subset N(v)\cap S$ with $|A|=t$.
  Then every set $A\subset S$ with $|A|=t$
is assigned to at most $t-1$ vertices $v$,
  since otherwise we have a $K_{t,t}$ as a subgraph of $G$.
  It follows that $\setof{N(v)\cap S}{v\in V(G), |N(v)\cap S|\ge t}$ has at most $(t-1)\cdot {s\choose t}$ elements.
  On the other hand, 
  $\setof{N(v)\cap S}{v\in V(G), |N(v)\cap S|< t}$
  has at most $s^{t-1}$ elements. 
Altogether, 
$\setof{N(v)\cap S}{v\in V(G)}\uplus S$,
which is in bijection with the partition $\cal P_S$ of $V(G)$,
has at most $s^t$ elements for $t\ge 3$ and $O(s^t)$ elements for $t=2$.
\end{proof}

\subsection{2VC-dimension}\label{app:2vc}
\begin{corollary*}[\ref{cor:2vcdim}]
  If $G$ is the exact $1$-subdivision of an $n$-clique,
  then $\fw_2(G)>(n-1)/4$. Furthermore, for every graph $G$, $\TVCdim(G)\le 8\fw_2(G)+2$.
\end{corollary*}
\begin{proof}
The first part follows immediately from Proposition~\ref{prop:subdivisions}, for $r=2$ and $k=(n-1)/4$.
We prove the second part.
 Let $V=V(G)$, and consider the bipartite graph $G[V,V]$.
 Suppose $\TVCdim(G)\ge k$.
 Then $G[V,V]$ contains the $1$-subdivision of $K_k$ as an induced subgraph. By the first part of Corollary~\ref{cor:2vcdim}, we have that $\fw_2(G[V,V])>(k-1)/4$.
 By Lemma~\ref{lem:fw-bfw}, we have $\bfw_r(G[V,V])\ge \fw_r(G[V,V])/2>(k-1)/8$,
 and by Lemma~\ref{lem:fw1-bip}, we have $\fw_r(G)\ge \bfw_r(G[V,V])> (k-1)/8$.
\end{proof}

\subsection{Flip-width of binary structures}
\label{app:bin-struct}
We extend the definition of flip-width to structures equipped with one or more binary relations.
To this end, we extend the notion of flips as follows.
Let $R\subset V\times V$ be a binary relation on a set $V$,
and let $(A,B)$ be a pair of subsets of $V$.
The relation $R'\subset V\times V$ obtained from $R$ by \emph{flipping} 
the pair $(A,B)$ (now the order of the pair matters)
is defined as $R':=R\triangle (A\times B)$,
where $\triangle$ is the symmetric difference.
We will apply such flips in the context of binary relational structures, as defined below.

Fix a binary relational signature $\Sigma$,
that is, a signature consisting of unary and binary relation symbols only
(see Section~\ref{sec:logic}).
Let $B$ be a $\Sigma$-structure 
and $\cal P$ be a partition of $V(G)$.
A \emph{$\cal P$-flip} is an operation  
which is specified by a $\Sigma$-structure $F$ with vertex set $\cal P$.
Applying this operation to $B$
results in the $\Sigma$-structure $B'$
with $V(B')=V(B)$ 
and relations \[R_{B'}\quad:=\quad R_B\quad\triangle \bigcup_{(P,Q)\in R_F}P\times Q,\] for each binary relation symbol $R\in\Sigma$.
The unary relation symbols are interpreted in $B'$ in the same way as in $B$.
By slight abuse of language, we sometimes call the structure $B'$ 
a $\cal P$-flip of $B$.

For $k\ge 1$, a \emph{$k$-flip} of $B$ is a $\Sigma$-structure $B'$ which is a $\cal P$-flip of $B$, for some partition $\cal P$ of $V(B)$ with $|\cal P|\le k$.

Fix $r\in\N\cup\set{\infty}$.
We now define the \emph{flipper game} of radius $r$ and width $k$ on a $\Sigma$-structure $B$ similarly as in the case of graphs, with the following differences:
in each round, the flipper announces a $k$-flip $B'$ of $B$, whereas the runner moves 
along a path of length at most $r$ in the Gaifman graph of the $k$-flip of $B$ that was announced in the previous round (and $B$ in the first round).
The radius-$r$ flip-width of $B$, denoted $\fw_r(B)$ is the smallest $k$ such that the flipper wins the flipper game of radius $r$ and width $k$ on $B$.

Graphs are viewed as structures
over the signature $\Sigma$ consisting of a single binary relation symbol $E$, interpreted in a given graph $G$ as the (symmetric, irreflexive) adjacency relation.
Note that in principle, applying a $\cal P$-flip $F$ to a graph $G$ can result in a $\Sigma$-structure $G'$ 
in which the binary relation is no longer symmetric.
This is the case when the binary relation of $F$ is not symmetric. However, for the flipper, it never pays off to apply such flips, since the runner moves in the Gaifman graph of the resulting structure, so the directions of the edges are of no relevance to the runner. Hence, $\fw_r(G)$ is the same when $G$ is regarded as a graph, or as a binary structure.

\begin{example}Let $B=(V,<)$ be a totally ordered set,
  viewed as a structure over the signature $\Sigma=\set{<}$.
  Let $\cal P$ be a partition of $V$ into $k$ 
  sets that are intervals with respect to $<$,
  and let $I_1,\ldots,I_k$ denote those intervals in increasing order.
  Consider the $\cal P$-flip of $B$ specified by 
  the $\Sigma$-structure $F$ with vertices $\cal P$, with $<_F=\setof{(I_i,I_j)}{1\le i<j\le k}$.
  Applying the flip $F$ 
  results in the $\Sigma$-structure $B'=(V,<')$,
  where $a<'b$ if and only if $a<b$ and $a$ and $b$ belong to the same part of $\cal P$.

  Now consider the $\cal P$-flip of $B$ specified by 
  the $\Sigma$-structure $F'$ with vertices $\cal P$, with $<_F=\setof{(I_i,I_j)}{1\le i \le j\le k}$
  (so we also flip pairs $(I_i,I_i)$).
  Applying the flip $F'$ 
  results in the $\Sigma$-structure $B'=(V,<')$,
  where $a<'b$ if and only if $a\ge b$ and $a$ and $b$ belong to the same part of $\cal P$ (so in each part, the order is reversed and becomes reflexive).
\end{example}

\begin{example}
  The radius-$\infty$ flip-width of a totally ordered set $B=(\set{1,\ldots,n},<)$ 
  is at most three. The strategy is similar as in Example~\ref{ex:half-graphs}, but now in the $i$th round, the flipper applies the $3$-flip $F$ of $B$ 
  as in the previous example, for the partition $\cal P$ of $V$ into the intervals $\set{1,\ldots,i},\set{i},\set{i+1,\ldots,n}$, removing all relations between distinct intervals.
\end{example}

\subsection{Definable flip-width}\label{app:dfw}
We prove Lemma~\ref{lem:decide-dfw}.\sz{macro}
\begin{lemma*}[\ref{lem:decide-dfw}]
  There is an algorithm that, given a graph $G$ 
  and numbers $k\in\N$ and $r\in\N\cup\set{\infty}$, determines whether 
  $\fwd_r(G)\le k$ in time $n^{O(k)}\cdot 2^{O(2^k)}$.
\end{lemma*}
\begin{proof}
  Fix $k,r$, and a graph $G$.
  A \emph{configuration} in the definable flipper 
  game with radius $r$
  consists of:
  \begin{itemize}
    \item a set $S\subset V(G)$ of size at most $k$,
    specifying the partition $\cal P$ played by the flipper,
    \item a $\cal P$-flip of $G$,
    \item the current position of the robber.
  \end{itemize}
  The set of all configurations has 
  size $O(n^{k+1}\cdot 2^{4^k})$,
  and the winner of the game can be computed using a fixpoint computation 
  running in time polynomial in $n^{k+1}\cdot 2^{4^k}$.
\end{proof}

\begin{lemma*}[\ref{lem:fw-dfw-vc}]
  Fix $r\in\N\cup\set{\infty}$.
  For every graph $G$ we have:
  \begin{align}
    \fw_r(G)&\le O(\dfw_r(G)^{\VCdim(G)}).  
  \end{align}
\end{lemma*}
\begin{proof}
  Let $G$ be a graph and let $d=\VCdim(G)$
  and let $k=\dfw_r(G).$
  Consider a set $S\subset V(G)$  with $|S|\le k$.
  Then the set system $(S,\setof{N(v)\cap S}{v\in V(G)})$
  has VC-dimension at most $d$. By the Sauer-Shelah-Perles lemma,
  we have that $\setof{N(v)\cap S}{v\in V(G)}$ has $O(|S|^d)=O(k^d)$ elements.

  It follows that every $k$-definable flip of $G$ is a $O(k^d)$-flip of $G$.
  Since $\dfw_r(G)\le k$ it follows that $\fw_r(G)=O(k^d)$.
\end{proof}

\section{Modular partition and substitution closure}\label{app:modular-and-subst}
A set of vertices $X\subset V(G)$ in a graph $X$ 
is a \emph{module} if all vertices in $X$ have the same neighbors 
outside of $X$. Note that in a modular partition, all the parts of the partition are modules.
\begin{lemma*}[\ref{lem:modular-partition}]
  Let $G$ be a graph and $\cal P$ be its modular partition.
  Then 
  \[\fw_r(G)\le \max\left(\fw_r(G/\cal P),\max_{A\in\cal P}{\fw_r(G[A])+2}\right).\]
\end{lemma*}
\begin{proof}[Proof sketch for Lemma~\ref{lem:modular-partition}]
  The strategy of the flipper is as follows.
  First, use a strategy of width $\ell=\fw_r(G/\cal P)$ on 
  $G/\cal P$. Each $\ell$-flip $G'$ of $G/\cal P$ induces an $\ell$-flip 
  $\wh G$ of $G$. Playing according to this strategy, the runner eventually reaches a vertex $A\in V(G/\cal P)=\cal P$ that is isolated in the current flip $G'$ of $G/\cal P$.
  The corresponding play in $G$ leads to a 
  $\ell$-flip $\wh G$ of $G$ such that 
  the current vertex $v$ of the runner is in $A$,
  and there are no edges joining $A$ and $V(G)-A$ in $\wh G$.
  
  Since $A$ is a module in $G$, the graph $G_A$ obtained from $G$ by removing all edges with one endpoint in $A$ and one endpoint in $V(G)\setminus A$, can be obtained from $G$ by flipping $A$ and $N(A)\setminus A$. 
  The flipper now announces
  the graph $G_A$, and the robber is still confined to $A$.
  
  Now, we use a winning the strategy of the flipper in the graph $G[A]$, of width $k=\fw_r(G[A])$. Whenever in the game on $G[A]$ the flipper announces a $k$-flip $G'$ of $G[A]$, in the game on $G$ the flipper  announces the graph $\wh G$  such that 
  $\wh G[A]=G'[A]$, $\wh G[V(G)-A]=G[V(G)-A]$, and there are no edges joining $A$ and $V(G)\setminus A$ in $\wh G$.
  The graph $\wh G$ is a $(k+2)$-flip of $G$, where the partition partitions $A$ into $k$ parts, according to the partition of $G[A]$ used in the $k$-flip $G'$ of $G[A]$, and partitions $V(G)\setminus A$ into two parts, $N(A)\setminus A$ and $V(G)\setminus (N(A)\cup A)$.
  It follows that playing according to this strategy, the flipper wins, once the flipper wins in the game on $G[A]$.
\end{proof}

\begin{lemma*}[\ref{lem:substitution}]
  Fix $r\in\N\cup\set{\infty}$
  For every $r\in\N\cup\set{\infty}$ 
  and graph class $\CC$, we have 
  \[\fw_r(\CC^*)\le \fw_r(\CC)+2.\]
  In particular, if $\CC$ has bounded flip-width, then $\CC^*$ has bounded flip-width.
\end{lemma*}
The idea is to use the winning strategy on the class $\CC$,
and confine the runner to $L(w)$,
for deeper and deeper nodes of $w$,
similarly as in the strategy given Example~\ref{ex:trees}
for the Cops and Robber game on trees.

\begin{proof}[Proof sketch for Lemma~\ref{lem:substitution}]
  Suppose $\fw_r(\CC)\le k$. Let $G\in \CC^*$.
Note that $G$ has a modular partition $\cal P$ 
such that $G/\cal P\in \CC^*$ and each part $A$ of $\cal P$ induces a graph $G[A]\in \CC^*$,
or is a singleton. By repeatedly applying the argument as in the proof of Lemma~\ref{lem:modular-partition}, we conclude that $\fw_r(G)\le k+2$.
\end{proof}

\section{Flip-width with infinite radius}
\label{app:cw}
The following yields the upper bound in Theorem~\ref{thm:cw}.
\begin{lemma}For every graph $G$ with $\rw(G)=k$,
  \[\fw_\infty(G)\le O(2^k).\]
\end{lemma}
\begin{proof}
We prove that every class $\CC$ of bounded clique-width has bounded flip-width at radius-$\infty$. The other implication was shown in Section~\ref{sec:cw}.

  Fix a number $k$ and let $\CC$ be a class of graphs of rank-width at most $k$.
Then for every $G\in \CC$ there is 
a rooted binary tree $T$ with leaves $V(G)$, such that for every node $w$ of $T$
the set $L_w$ of leaves that are descendants of $w$ induces at most $2^k$ distinct neighborhoods over its complement $V(G)\setminus L_w$,
and conversely, $V(G)-L_w$ induces at most $2^k$ distinct neighborhoods in $L_w$.

We present a strategy for the flipper, for radius $r=\infty$.
The flipper keeps track of a node $w_i$ of $T$,
maintaining the invariant that in round $i$, $w_i$ is at distance exactly $i$ from the root, and $A_i\subset L_{w_i}$,
where $A_i$ denotes the set of possible next moves of the runner.

Initially, $w_0$ is the root, and $A_0=V(G)$, so the invariant is satisfied.

We now describe the flipper's strategy that maintains the invariant.
Suppose we are in round $i$. 
If $w_i$ is a leaf then the flipper wins, since $A_i\subset \set{w_i}$ by the invariant.
Otherwise, partition $V(G)$ into three parts:
\begin{itemize}
  \item $B_0$ -- the leaves of $T$ below the left child of $w_{i-1}$,
  \item $B_1$ -- the leaves of $T$ below the right child of $w_{i-1}$,
  \item $B_2$ -- the remaining leaves of $T$.
\end{itemize}
For $j=0,1,2$, let $\cal Q_j$ be the partition of the set $B_j$ into equivalence classes of the relation of having equal neighborhoods in $V(G)\setminus B_j$. Then $|\cal Q_j|\le \Oof(2^k)$, for $j=0,1,2$.
Define the partition $\cal P_i:= \cal Q_0 \cup\cal Q_1\cup \cal Q_2$.
Then there is a $\cal P_i$-flip $G_i$ of $G$ 
such that for any edge $uv\in E(G_i)$, the vertices $u$ and $v$ belong to the same part of the partition $\set{B_0,B_1,B_2}$.
The flipper plays the partition $\cal P_i$ and the flip $G_i$.

Suppose the runner responds by moving to a vertex $c_i\in A_{i-1}$.
Since $A_{i-1}$ consists of descendants of $w_{i-1}$ by the invariant,
it follows that either $c_i\in B_0$ or $c_i\in B_1$.
Therefore, the set $A_i$, which is defined as the connected component 
of $c_i$ in $G_i$, is either contained in $B_0$, or in $B_1$.
In the first case, define $w_i$ as the left child of $w_{i-1}$, 
and in the latter case, as the right child of $w_{i-1}$.
Thus, the invariant is maintained.

Since $T$ is a finite tree, after at most $|T|$ rounds $w_i$ is a leaf,
and then flipper wins.
\end{proof}

We now prove Lemma~\ref{lem:well-linked-hideout},
completing the proof of the lower bound in Theorem~\ref{thm:cw}.
First, a lemma.

\begin{lemma}\label{lem:rank-flip}
  Let $G$ be a graph 
  and $A\uplus B$ a bipartition of $V(G)$ 
  with $\rk_G(A,B)>k$.
  Then for every $k$-flip $G'$ of $G$ there is some  edge $ab\in E(G')$ with $a\in A$ and $b\in B$.
\end{lemma}
\begin{proof}Let $G'$ be a $k$-flip of $G$.
  Denote by $H$ the graph with vertices $V(G)$ and edges $E(G)\triangle E(G')$.
  Since $G'$ is a $k$-flip of $G$,
  it follows that $H$ is a $k$-flip of the edgeless graph,
  and thus $\rk_H(A,B)\le k$.
  Since $E(G)=E(G')\triangle E(H)$, we have that
  \begin{align}\label{eq:ranks-sum}
    \rk_G(A,B)\le \rk_{G'}(A,B)+\rk_{H}(A,B),
  \end{align}
  as the rank of the sum of two matrices 
  is at most the sum of the ranks.
  If there are no edges $ab\in E(G')$ with $a\in A,b\in B$, then
  we have $\rk_{G'}(A,B)=0$, and therefore $\rk_G(A,B)\le k$ by  \eqref{eq:ranks-sum}.
\end{proof}

\begin{lemma*}[\ref{lem:well-linked-hideout}]
  Fix a graph $G$ and number $k\in\N$.
  Every well-linked set $U$ with $|U|>3k$ is a $(\infty,k,k)$-hideout. 
\end{lemma*}
\begin{proof}
  Let $G'$ be a $k$-flip of $G$.
    For a vertex $v\in U$,
  let $U(v)\subset U$ denote the set of vertices 
  $w\in U$ that are reachable from $v$ by a path in $G'$
  (including $v$).
We show that there are at most $k$ vertices $v\in U$ 
with $|U(v)|\le k$, proving that $U$ is a $(\infty,k,k)$-hideout.

Suppose otherwise: that there are  $k+1$ such distinct vertices $v_1,\ldots,v_{k+1}\in U$.
Let $i$ be the smallest number
such that $|U(v_1)\cup\cdots\cup U(v_i)|> k$. Then $2\le i\le k+1$.
Let $U'=U(v_1)\cup\cdots\cup U(v_i)$.
Then $k< |U'|\le 2k$, where the second inequality holds because 
$|U(v_i)|\le k$ and $|U(v_1)\cup\cdots\cup U(v_{i-1})|\le k$ by minimality of $i$.
Since $|U|>3k$, we have that $|U-U'|>k$ and $|U'|>k$.
Let $A\subset V(G)$ be the union of the connected components of $G'$ that intersect $U'$,
and let $B=V(G)-A$. Then there is no edge in $G'$ 
with one endpoint in $A$ and one endpoint in $B$.
As  $A\cap U=U'$ and $B\cap U=U-U'$ have both size greater than $k$, we have that $\rk_G(A,B)>k$ as $U$ is well-linked.
By Lemma~\ref{lem:rank-flip}, there is an edge of $G'$ with one endpoint in $A$ and one endpoint in $B$, a contradiction.
\end{proof}

\section{VC-dimension}
\subsection{Bounding VC-dimension in terms of radius-one flip-width}
\label{app:fw-and-VC}

By a similar argument as in Lemma~\ref{lem:near-twins}, we prove the following (see Appendix~\ref{app:bfw} for the definition of  bipartite flip-width, $\bfw$).
  \begin{lemma*}[\ref{lem:near-twins-two-sets}]
    Let $b,k\in \N$ and let $G$ be a bipartite graph with $\bfw_1(G)\le k$.
    Then $G$ contains more than $b$ mutual $2 b k$-near-twins
    contained in a single part of $G$.
  \end{lemma*}

\newcommand{\Z}{\mathbb Z}
To prove Theorem~\ref{thm:vcdim}, we show that every bipartite graph of sufficiently large VC-dimension contains an induced subgraph with no $k$-near-twins. First we construct such bipartite graphs.
For a number $m\in\N$ let $\Z_2^m$ denote the $m$-dimensional vector 
space over the two-element field $\Z_2$,
and for $v,w\in\Z_2^m$ denote $v\cdot w=\sum_{i=1}^mv_i\cdot w_i \mod 2$.

\begin{lemma}\label{lem:form}
  Fix a number $m\in\N$
  and consider the bipartite graph $H$ whose parts are two 
  copies of $\Z_2^m$, 
  with edges $vw$ such that $v\cdot w\neq 0$.
  Then $H$ has no $(2^{m-1}-1)$-near-twins in either part.
\end{lemma}
\begin{proof}Denote the two copies of $\Z_2^m$ with $V$ and $V^*$,
  and write $v^*$ for the copy of a vector $v\in V$
  in $V^*$.
  We claim that for all $u,v\in V$ we have $|N_H(u)\triangle N_H(v)|=2^{m-1}$.
  Indeed, 
  \begin{multline*}
  N_H(u)\triangle N_H(v)=\setof{w^*}{w\in V,w\cdot u\neq w\cdot v}\\=\setof{w^*}{w\in V,w\cdot (u-v)\neq 0}.
  \end{multline*}
  Since $\setof{w\in V}{w \cdot (u-v)\neq 0}$ has $2^{m-1}$ elements, so does $N_H(u)\triangle N_H(v)$.
  
  By a dual argument, $|N_H(u^*)\triangle N_H(v^*)|=2^{m-1}$ 
  for all $u^*,v^*\in Y$.
  Therefore, $H_m$ has no pair of $(2^{m-1}-1)$-near-twins in either part. 
\end{proof}

\begin{lemma*}[\ref{lem:vector-space-twins}]
  Let $G$ be a graph with $\VCdim(G)\ge 2^m$, for some $m$.  Then 
  there are two sets $X,Y$ such that the bipartite graph $G[X,Y]$ 
  contains no pair of $(2^{m-1}-1)$-near-twins in either of the parts $X, Y$.
\end{lemma*}
\begin{proof}
  Let $X$ be a subset of $V(G)$ of size $2^m$ that is shattered by $\setof{N(v)}{v\in V(G)}$. 

  Arbitrarily identify the elements of $X$ with the elements of $\Z_2^m$. 
  Denote $v^\perp:=\setof{u\in X}{u\cdot v=0}\subset X$.
  Since $X$ is shattered in $G$,
  for every $v\in X$ there is a vertex $v^*\in V(G)$
  such that $N(v^*)\cap X=v^\perp$. 
  Denote $Y:=\setof{v^*}{v\in X}$.
  The function $v\mapsto v^*$
  is then a bijection between $X$ and $Y$,
  and for $v,w\in X$, $v$ is adjacent to $w^*$ in $G$
  if and only if $v\cdot w\neq 0$. Therefore, $G[X,Y]$ is isomorphic to 
  the graph from Lemma~\ref{lem:form}, and the conclusion follows.
\end{proof}

\begin{proof}[Proof of Theorem~\ref{thm:vcdim}]
  We show that if $\VCdim(G)\ge 2^m$, then $\fw_1(G)> 2^{m-2}$.
  Assume $\VCdim(G)\ge 2^m$, and let $G[X,Y]$ be as in Lemma~\ref{lem:vector-space-twins}.
  By Lemma~\ref{lem:near-twins-two-sets} (setting $b=1$), 
  this implies that $\bfw_1(H)>2^{m-2}$.
  By Lemma~\ref{lem:fw1-bip}, $\fw_1(G)>2^{m-2}$.
\end{proof}

\subsection{Duality result}\label{app:duality}
We reprove a duality result which is a corollary 
of the $(p,q)$-theorem of Alon-Kleitman-Matou\v{s}ek \cite[Theorem 4]{Matousek:p-q-theorem}.
We present a self-contained proof, to analyze the bounds.
Our presentation is based on a proof by 
Simon \cite[Cor 6.3; Sec. 6.3]{simon_book}.

Say that a binary relation $E\subset A\times B$ 
has a \emph{duality of order $k$}, where $k\in\N$,
if at least one of two cases holds:
\begin{itemize}
  \item there is a set $B'\subset B$, $|B'|\le k$,
  such that for all $a\in A$ there is some $b\in B'$ with $(a,b)\in E$, or 
  \item there is a set $A'\subset A$, $|A'|\le k$,
  such that for all $b\in B$ there is some $a\in A'$ with $(a,b)\notin E$.
\end{itemize}

Recall that the VC-dimension of a binary relation $E$ is the maximum of the VC-dimension of two set systems:
\[(B,\setof{\vec E(a)}{a\in A})\qquad\text{and}\qquad
(A,\setof{\cev E(b)}{b\in B}).\]

Here's the duality result.
\begin{theorem}\label{thm:duality}
  Let $E\subset A\times B$ be a binary relation,
  with $A,B$ finite.
  Then $E$ has a duality of order $O(d)$,
  where $d=\VCdim(E)$.
\end{theorem}

The result follows from a combination of two classic
results: the Vapnik-Chervonenkis theorem, 
and von Neumann's minimax theorem.
We first state a corollary of von Neumann's minimax theorem (also of Farkas' lemma, the Hahn-Banach theorem, or of the strong duality theorem for linear programming).

\begin{theorem}[Minimax Theorem]\label{thm:minimax}
  Let $E\subset A\times B$ be a binary relation with $A,B$ finite, and let $\alpha\in \R$.
  Then exactly one of the following two cases holds:
  \begin{enumerate}
    \item there is some probability distribution $\nu$ on $B$ such that
    $\nu(\vec E(a))\ge \alpha$
     holds for all $a\in A$,  
    \item there is some probability distribution $\mu$ on $A$ such that
    $\mu(\cev E(b))< \alpha$
    holds for all $b\in B$.
    \end{enumerate}
\end{theorem}

To formulate the VC-theorem, 
we introduce the following notions.

For a multiset $S$ of elements of a set $V$
and a set $F\subset V$,
denote by 
$\textup{Av}_{S}(F)$
the proportion of elements of $S$ (counted with multiplicities) that belong to $F$,
that is 
\[\textup{Av}_{S}(F)=\frac{|S\cap F|}{|S|}.\]

Let $(V,\cal F)$ be a set system with $V$ finite,
$\mu$ a probability distribution on $V$,
and $\eps>0$ a real.
We say that a set $S$ is an \emph{$\eps$-approximation} of $\mu$ on $\cal F$,
if for every $F\in\cal F$ 
we have 
\[|\textup{Av}_{S}(F)-\mu(F)|\le \eps.\]

\begin{theorem}[VC-theorem]
  Let $(V,\cal F)$ be a set system with $V$ finite,
  and let $d$ be its VC-dimension.
For every $\eps>0$ and
 probability distribution $\mu$ on $V$,
there exists an $\eps$-approximation $S$ of $\mu$ on $\cal F$ with \[|S|\le O(d)\cdot \frac{1}{\eps^2} \log\left(\frac {1}\eps\right).\]
\end{theorem}

\begin{proof}[Proof of Theorem~\ref{thm:duality}]
  Set $\alpha:=\tfrac 1 2$ and $\eps:=\tfrac 1 3$;
  the point is that $0<\alpha\pm\eps<1$.
  Fix $d\in\N$, and let $k$ be the number from 
  the VC-theorem, with $k\le 
  O(d)\frac{1}{\eps^2}\log(\frac {1}\eps)=O(d)$.
  By the Minimax Theorem applied to $\alpha=\tfrac 1 2$,
  one of the two cases below holds.

  \textbf{Case 1:} \emph{There 
  is some probability distribution $\nu$ on $B$ 
  such that $\nu(\vec E(a))\ge \tfrac 1 2$ holds for all $a\in A$.}

  \noindent By the VC-theorem applied to $\varepsilon=\tfrac 1 3$, there is a multiset 
  $B'\subset B$ with $|B'|\le k$
  such that 
  \[|\textup{Av}_{B'}(\vec E(a))-\nu(\vec E(a))|\le\tfrac 1 3\qquad\text{for all }a\in A.\]
  Pick $a\in A$.
Since $\nu(\vec E(a))\ge \tfrac 1 2$,
we have that $\textup{Av}_{B'}(\vec E(a))>0$,
so $B'\cap \vec E(a)$ is nonempty.
Therefore, the first condition in the definition of a duality is satisfied.

\textbf{Case 2:} \emph{There is some probability distribution 
$\mu$ on $A$ such that $\mu(\cev E(b))<\tfrac 1 2$ holds for all $b\in B$. }

\noindent  By the VC-theorem again, there is a multiset $A'\subset A$ with $|A'|\le k$ such that 
\[|\textup{Av}_{A'}(\cev E(b))-\mu(\cev E(b))|\le \tfrac 1 3\qquad\text{for all }b\in B.\]  
Pick $b\in B$.
  Since $\mu(\cev E(b))<\tfrac 1 2$,
  we have that $\textup{Av}_{A'}(\cev E(b))<1$, 
  so $A'$ is not contained in $\cev E(b)$. Therefore, the second condition in the definition of a duality is satisfied.
\end{proof}

\subsection{Definability result}\label{app:inc}
The following result is proved 
in \cite[Thm. 3.5]{boundedLocalCliquewidth}, up to a slight variation of the assumptions (see proof below).
\begin{lemma*}[\ref{lem:incr}]
  Fix $k,d\in\N$.
  Let $V$ be a set equipped with:
  \begin{itemize}  
    \item a binary relation $E\subset V\times V$ of VC-dimension at most $d$,
    \item a pseudometric $\dist\from V\times V\to \R_{\ge0}\cup\set{\infty}$
     (that is, a function satisfying the triangle inequality),
    \item 
    and a partition $\cal P$ of size at most $k$,
  \end{itemize}
  such that $E(u,v)$ depends only on the $\cal P$-class of $u$ and the $\cal P$-class of $v$ whenever $\dist(u,v)>1$.
  Then there is a set $S\subset V$ of size $\Oof(dk^2)$, such that $E(u,v)$ depends only on the $S$-types of $u$ and of $v$, whenever $\dist(u,v)>5$.
\end{lemma*}
\begin{proof}
  \cite[Thm. 3.5]{boundedLocalCliquewidth} prove the same statement, except that 
  the assumption ${\VCdim(E)\le d}$  is replaced with:
  \emph{for every $A,B\subset V$, the relation $E\cap (A\times B)$
   has a duality of order $d$}.
  By Theorem~\ref{thm:duality},
  this implies the formulation above.
\end{proof}

\section{Exact subdivisions}\label{app:exact}

We prove Proposition~\ref{prop:subdivisions}, which is repeated below.
\begin{proposition*}[\ref{prop:subdivisions}]
  Fix $r\ge 2,k\ge 1$. 
  Let $G$ be the exact $(r-1)$-subdivision of some graph $H$ 
  with minimum degree at least $2rk$.
  Then $\fw_{r}(G)>k$.
\end{proposition*}

We first prove two lemmas.

\begin{lemma}\label{lem:bipartite-flips}
  Fix $k,\ell,m\ge 1$.
  Let $H$ be a bipartite graph with bipartition $(L,R)$, in 
  which every vertex in $L$ has degree at least $\ell$,
  and any two distinct  vertices in $L$ have at most $m$ common neighbors.
  Let $H'$ be a $k$-flip of~$H$.
  Then there is a set $X\subset L$ with $|X|\ge |L|-k$, and a bijection $\pi\from X\to X$,
  such that every $v\in X$
  is adjacent in $H'$ to at least $\lceil\frac{\ell-m}2\rceil$ vertices in $N_H(\pi(v))$.
\end{lemma}
\begin{proof}
  Let $\cal P$ be the partition of $V(G)$ with $|\cal P|\le k$ 
  such that $H'$ is a $\cal P$-flip of $H$.
  
  First consider the case when all the vertices of $L$ are in one part $A$ of $\cal P$. Let $W$ be the   
  union of the parts $B$ of $\cal P$ that are not flipped with $A$ in the flip that produces $H'$ from $H$.
  Let $X_1$ consist of those vertices $v\in L$ such that $|N_H(v)\cap W|\ge \frac{\ell-m}2$,
  and let $X_2$ consist of the remaining vertices in $L$.\szfuture{pic}
  
  Every  vertex in $X_1$ is adjacent in $H'$ to at least $\frac{\ell-m}2$ vertices in $N_H(v)$ (namely, to $N_H(v)\cap W$), so we can set $\pi(v)=v$ for all $v\in X_1$.

  If $v$ and $v'$ are two distinct vertices in $X_2$, then $N_{H'}(v)\supseteq N_H(v')\cap W$ and 
  $|N_H(v')\cap W|\ge \frac{\ell-m}2$.
  If $|X_2|\le 1$ then set $X:=X_1=L\setminus X_2$, and let $\pi\from X\to X$ be the identity on $X$.
  If $|X_2|>2$, set $X:=L=X_1\cup X_2$, and let $\pi\from X\to X$ be a permutation 
  that maps every vertex in $X_1$ to itself, and acts as a cyclic permutation on the vertices in $X_2$.
In any case, $|X|\ge |L|-1$, and every $v\in X$
is adjacent in $H'$ to at least $\lceil\frac{\ell-m}2\rceil$ neighbors of $\pi(v)$.

In the general case, partition $L$ as $L=L_1\uplus\ldots\uplus L_s$, for some $s\le k$, following the partition $\cal P$ restricted to $L$. 
For each $1\le i\le s$, let $H_i=H[L_i,R]$, and $H_i'=H'[L_i,R]$;
then $H_i'$ is a $k$-flip of $H_i$, and they fall into the special case considered above.
Hence, for each $1\le i\le s$ there are set $X_i\subset L_i$ with $|L_i|\ge |X_i|-1$ and a bijection $\pi_i\from X_i\to X_i$. Set $X=X_1\uplus\ldots\uplus X_s$, and let $\pi\from X\to X$ be such that $\pi(x)=\pi_i(x)$ for $x\in X_i$. Then $X$ and $\pi$ satisfy the required condition.
\end{proof}

Applying Lemma~\ref{lem:bipartite-flips} in the case $\ell=1$ and $m=0$, we get the following.
\begin{corollary}\label{lem:matching-flips}
  Fix $k\ge 1$.
  Let $M$ be a matching between two sets $L$ and $R$, and let $M'$ be a $k$-flip of $M$.
  Then $M'$ contains, as a subgraph, a matching between all but $k$ vertices of $L$, and a set of vertices of $R$ of the same size.
\end{corollary}

For $r,\ell\ge 1$, let $G_{r,\ell}$ denote the 
 union of 
$\ell$ paths, each of length $r$ (and with $r+1$ vertices), and in each path, call one of the vertices 
of degree one a \emph{source}, and the other one a \emph{target}.
By an easy induction on $r\ge 1$, Corollary~\ref{lem:matching-flips} gives the following.
\begin{lemma}\label{lem:path-flips}Fix $k,r,\ell\ge 1$.
  If $G'$ is a $k$-flip of $G_{r,\ell}$, then 
  at least $\ell-rk$ target vertices of $G_{r,\ell}$
are joined by a path of length $r$ in $G'$ with some 
source vertex.
\end{lemma}

\begin{proof}[Proof of Proposition~\ref{prop:subdivisions}]Fix $r\ge 2$ and let $G$ be an exact $(r-1)$-subdivision of some graph $H$ 
  with minimum degree at least $2rk$. We aim to prove that $\fw_r(G)>k$.
  Let $P\subset V(G)$ denote the set of \emph{principal vertices} of $G$, that is, vertices
   of degree larger than two;
  those vertices correspond to the vertices of $H$. 
  We show that $P$
  forms a $(r,k,k)$-hideout in $G$. By Lemma~\ref{lem:hideouts}, this implies that $\fw_{r}(G)>k$. 

  \szfuture{pic}
  Let $R=N_G(P)=\bigcup_{v\in P}N_G(v)$ denote the set of neighbors of the principal vertices in $G$. Note that $R$ and $P$ are disjoint, as $r\ge 2$.
  Then $G[P,R]$ is a 
  bipartite graph in which every vertex in $P$ has 
  at least $2rk$ neighbors in $R$, 
  and any two vertices in $P$ have 
at most one common neighbor in $R$ 
(in fact, no common neighbors if $r\ge 2$)\sz{simplify lemmas?}.
For each $v\in P$, there is an induced subgraph $G_v$ of 
$G$ that is isomorphic to $G_{r-1,\ell}$, for some $\ell\ge 2rk$ with the source vertices equal to $N_G(v)$, and where each target vertex is a principal vertex.

  Let $G'$ be a $k$-flip of $G$.
  Call a principal vertex $v\in P$ \emph{good}
if there is some $w\in P$ such that 
  if $v$ is adjacent in $G'$ to at least $rk$ elements of $N_G(w)$,
  and call $v$ \emph{bad} otherwise. 
We show that 
(1) there are at most $k$ bad vertices,
and (2) every good vertex has more than $k$ principal vertices in its $r$-neighborhood in $G'$.
It follows that the principal vertices form a $(r,k,k)$-hideout in $G$.

(1) Apply Lemma~\ref{lem:bipartite-flips} to the bipartite graph $G[P,R]$
and  to  $G'[P,R]$, which is a $k$-flip of $F$.
Since every vertex in $P$ has degree at least $2rk$,
and any two distinct vertices in $L$ have at most one common neighbor, 
and $\lceil(2rk-1)/2\rceil=rk$, there are at most $k$ bad principal vertices by Lemma~\ref{lem:bipartite-flips}.

(2) Let $v\in P$ be a good principal vertex,
and let $w\in P$ be such that some set $A$ of $rk$ vertices of $N_G(w)$ are adjacent to 
$v$ in $G'$.
Consider the subgraph $K$ of $G_w$, consisting of vertex-disjoint paths of length $r-1$ in $G$ joining the vertices of $A$ 
with $rk$ principal vertices.
In particular, $K$ is isomorphic to $G_{r-1,rk}$.
Let $K'= G'[V(K)]$; then $K'$ is a $k$-flip of $K$.
By Lemma~\ref{lem:path-flips}, at least $rk-(r-1)k=k$
principal vertices are joined by a path of length $r-1$
in $K'$ with some vertex in $A$.
Those $k$ principal vertices are therefore at distance 
at most $r$ in $G'$ from $v$.
\end{proof}

\section{Bounded twin-width}\label{app:tww}

\subsection{Bounding flip-width in terms of twin-width}
We prove:
\begin{theorem*}[\ref{thm:btww}]
  Fix $r\in\N$. For every graph $G$ of twin-width $d$ we have:
  \begin{align}\label{eq:tww-bound}
    \fw_r(G)\le \pi_G(d^{O(r)})\le 2^d\cdot d^{O(r)}.
  \end{align}   
   In particular, every class of bounded twin-width has bounded flip-width.
  \end{theorem*}
The second inequality in \eqref{eq:tww-bound} is immediate by the following result, proved in~\cite{tww-neighborhood-complexity}.
\begin{theorem}\label{thm:tww-nbd-complexity}
  Let $G$ be a graph of twin-width $d\ge 1$, and let $A\subset V(G)$.
  Then we have:
   \[|\setof{N(v)\cap A}{v\in V(G)}|\quad \le\quad  2^{d+O(\log d)}\cdot |A|.\]
\end{theorem}
This result improves a previous result~\cite{tww-polynomial-kernels,tww-distality},
which gave a doubly-exponential dependency on $d$.

  \begin{proof}
    Fix an uncontraction sequence $\cal P_1,\cal P_2,\ldots,\cal P_n$ of $G$ of red-degree $d$.
    For $v\in V(G)$ and $1\le i\le n$ let 
    $B_r^i(v)\subset\cal P_i$  denote the ball of radius $r$ in the red graph of $\cal P_i$, around the part of $\cal P_i$ containing $v$. In particular, 
    $|B_r^i(v)|\le 2d^{r}\le\Oof_{r,d}(1)$.
  
    We describe a strategy of the flipper that guarantees that the following invariant holds after round $i$, for $i=1,2,\ldots,n$:
    \begin{align}\label{inv:tww}
      A_{i}\subset \bigcup B_r^i(c_{i}), 
    \end{align}     
    where $c_i$ is the runner's position in round $i$, and 
    $A_i\subset V(G)$ is the set of possible positions the runner can move to in round $i+1$ of the game,
    that is, $A_i$ is the ball of radius $r$ around $c_i$ in the graph $G_i$
    announced by the flipper in round $i$, with $G_1=G$
    and $c_1$ being the initial vertex of the runner.
    
    Note that for $i=n$, the inclusion~\eqref{inv:tww} implies that $|A_n|=1$ as in $\cal P_n$ there are no red edges and each part is a singleton.
     Therefore, the flipper wins the game after $n$ rounds according to this strategy. We need to show how the flipper can maintain the invariant \eqref{inv:tww},
     by playing $k$-flips of $G$ for some $k$ bounded in terms of $r$ and $d$.
  
    Before describing the strategy, we make an  observation that will be useful in the inductive reasoning. It gives a description of a ball in the red graph of $\cal P_{i-1}$, in terms of the red graph of $\cal P_i$.  
    Below, for $\cal F\subset \cal P_i$, the set $B_r^i(\cal F)\subset \cal P_i$ denotes the set of parts of $\cal P_i$ that are at distance at most $r$ from some part in $\cal F$ in the red graph of $\cal P_i$.
    \begin{claim}\label{cl:silly}
    Let $1< i\le n$
    and let $v\in V(G)$.
    Then there is a set $\cal F\subset \cal P_i$ with $|\cal F|\le d+3$,
    such that $B_r^{i-1}(v)\subset B_r^{i}(\cal F)$.
    \end{claim}
    \begin{proof}
      The family $\cal F$ consists of:
      \begin{itemize}
        \item the part of $\cal P_i$ that contains $v$,
        \item the parts $A,B$ of $\cal P_i$ such that $A\cup B$ is a part of $\cal P_{i-1}$,
        \item the parts in $\cal P_i$ that are not homogeneous towards $A\cup B$ in $G$.
      \end{itemize}
      It can be easily verified that $\cal F$ satisfies the statement of the claim.
    \end{proof}
    
    We now describe flipper's strategy.
    In the first round, we have $G_1=G$, and the invariant \eqref{inv:tww} is trivially satisfied since $\cal P_1$ 
    has just one part, and that part contains $A_1=B^r_{G}(c_1)$, regardless of the runner's choice of $c_1$.
  
    Suppose that the invariant~\eqref{inv:tww} is satisfied after round $i-1$, for some $1<i\le n$, so that $A_{i-1}\subset \bigcup B^{i-1}_r(c_{i-1})$.
  We describe how the flipper should play to maintain invariant \eqref{inv:tww} after round $i$.
  Apply Claim~\ref{cl:silly} to $v=c_{i-1}$, obtaining a family $\cal F\subset \cal P_i$ with $|\cal F|\le d+3$ 
  such that $B_r^{i-1}(c_{i-1})\subset B_r^i(\cal F)$. In particular, $A_{i-1}\subset \bigcup B_r^i(\cal F)$.
  
  Let $R=V(G)\setminus \bigcup B_{2r}^{i}(\cal F)$,
  and let $\cal R$ be the partition of $R$ 
  according to the equivalence relation of having the same neighborhood in 
  the set $\bigcup B_{2r-1}^{i}(\cal F)$. 
  Note that we are simultaneously considering balls around $\cal F$ with radii $r,2r-1$, and $2r$.

  \begin{claim}\label{cl:tww-types}
     We have $|\cal R|\le \pi_G(2(d+3)d^{2r-1})$,
     and every part $P\in \cal R$ is homogeneous towards all parts in $B_{2r-1}^i(\cal F)$.
  \end{claim}
  \begin{proof}
    The ball of radius $2r-1$
    in the red graph of $\cal P_i$ around 
    a part $P$ consists of at most 
    $$1+d+\ldots+d^{2r-1}\le 2d^{2r-1}$$
    parts.
    As $B_{2r-1}^{i}(\cal F)$ 
    is a union of at most $|\cal F|\le d+3$ many such balls, therefore,
    $$|B_{2r-1}^{i}(\cal F)|\le 2(d+3)d^{2r-1}.$$ 

  Pick a set $X$ which contains one representative of each part in $B_{2r-1}^{i}(\cal F)$. Then $|X|\le 2(d+3)d^{2r-1}$.
  Note that every part $P\in \cal P_i\setminus B_{2r}^{i}(\cal F)$ is homogeneous 
  towards every part $Q\in \cal B_{2r-1}^i(\cal F)$.
  Therefore, the neighborhood of a vertex $v\in R$ in $\bigcup B_{2r-1}^i(\cal F)$
  is completely determined by $N(v)\cap X$.
  By definition of $\pi_G$, there are at most 
  $\pi_G(2(d+3)d^{2r-1})$ different neighborhoods of vertices $v\in R$ in $X$, so $|\cal R|\le \pi_G(2(d+3)d^{2r-1})$.
  \end{proof}
  
  Let $\cal P_i'=B_{2r}^i(\cal F)\cup \cal R$.
  Then $\cal P_i'$ is a partition of $V(G)$.
We have $|B_{2r}^i(\cal F)|\le 2(d+3)\cdot  d^{2r}$,
  and it follows from Claim~\ref{cl:tww-types} that 
  \begin{align}\label{eq:flipper-bound}
    |\cal P_i'|\le |B_{2r}^i(\cal F)|+|\cal R|\le 2(d+3)\cdot  d^{2r}+ \pi_G(2(d+3)d^{2r-1})\le O(d)^{2r+1}+ \pi_G(O(d)^{2r})\\\le \pi_G(d^{O(r)}).
  \end{align}

  The flipper plays the $\cal P_i'$-flip $G_i$ of $G$
  obtained by flipping any pair $P,Q$ of distinct parts of $\cal P_i'$ such that
  the pair $P,Q$ is complete in $G$.
  
  Now, the runner makes his move, and picks a vertex $c_i\in A_{i-1}$,
  and we set $A_i:=B_r^{G_i}(c_i)$.
  We now prove that the invariant \eqref{inv:tww}
  holds.
  
  \begin{claim}\label{claim:tww-inv}
  $B_r^{G_i}(c_i)\subset \bigcup B_{r}^{i}(c_i)$.
  \end{claim}
  \begin{proof}
  Note that $c_i\in A_{i-1}\subset \bigcup B_{r}^i(\cal F)$.\szfuture{figure}
  
    Let $v\in B_r^{G_i}(c_i)$ be a vertex at distance $k$ from $c_i$ in $G_i$,
    for some $0\le k\le r$. We show by induction on $k$ that if $A$ 
    is the part of $\cal P_i$ that contains $v$, then $A\in B_k^i(c_i)$.
    For $k=r$, this immediately yields the claim.
  
  For $k=0$ the statement holds trivially, so suppose that $k>0$ and the statement holds for $k-1$.
  Let $w$ be a vertex that is a neighbor of $v$ and is at distance $k-1$ from $c_i$ in $G_i$,
  and let $B$ be the part of $\cal P_i$ that contains $w$.
  By inductive assumption, $B\in B_{k-1}^i(c_i)$, and since $c_i\in \bigcup B_r^i(\cal F)$,
   it follows that $B\in B_{r+k}^i(\cal F)$. As $k<r$, we have that $B\in B_{2r-1}^i(\cal F)$.
  
   Since $A$ and $B$ are connected by an edge in $G_i$, 
   it must be the case that $A$ and $B$ are not homogeneously connected in $G$.
   As $B\in B_{2r-1}^i(\cal F)$, it cannot be that $A\in \cal R$ 
   (since all parts in $\cal R$ are homogeneous towards all parts in $B_{2r-1}^i(\cal F)$
   by Claim~\ref{cl:tww-types}), so $A\in B_{2r}^i(\cal F)$.
   As $A$ and $B$ are not homogeneously connected in $G$,
   $A$ is a neighbor of $B$ in the red graph of $\cal P_i$.
   As $B\in B_{k-1}^i(c_i)$, it follows that $A\in B_k^i(c_i)$, finishing the inductive step.
  \end{proof}
  
  The invariant \eqref{inv:tww} now follows from Claim~\ref{claim:tww-inv}.
  In particular, if the flipper continues playing this way,
  at the end of round $n$ we have that $|A_n|=1$, so flipper wins.

  As in every step, the flipper plays a $\cal P_i'$-flip,
  the 
  first inequality in \eqref{eq:tww-bound} follows by \eqref{eq:flipper-bound}.
  By Theorem~\ref{thm:tww-nbd-complexity}, 
  we have 
  $$\pi_G(d^{O(r)})\le 2^{d+O(\log d)}\cdot d^{O(r)}\le 2^d\cdot d^{O(r)},$$
  proving the second inequality in \eqref{eq:tww-bound}.
  \end{proof}

\subsection{Ordered flip-width}\label{app:ordered flip-width}
We prove Lemma~\ref{lem:ordered flip-width}, repeated below.
\begin{lemma*}[\ref{lem:ordered flip-width}]
  Fix $r\in \N\cup\set{\infty}$ and an ordered graph $G=(V,E,<)$.
  Then \[\sqrt{\fw_r(G)+1}\quad\le\quad fw_{r}^<(G)+1\quad\le\quad \fw_{3r+2}(G)+1.\]
 \end{lemma*}
 We first study the effects of applying flips 
 to a set equipped with a total order.
\begin{lemma}\label{lem:order-flips}
  Let $L=(V,<)$ be a total order and 
  let $L'$ be a $k$-flip of $L$ (as a binary structure).
  Then there is a set  $S\subset V$ with $|S|\le k$
  such that any two vertices of $V\setminus S$ 
  with no vertex of $S$ between them
  are at distance at most $2$ in the Gaifman graph of $L'$.
\end{lemma}
\begin{proof}
  Let $\cal P$ be a partition of $V$ with $|\cal P|\le k$
  such that $L'$ is a $\cal P$-flip of $L$.
  Note that the parts of $\cal P$ need not be convex in the order $<$; this is the main challenge here.
  Let $S=\setof{\max(A)}{A\in\cal P}$ be the set of $<$-maximal elements 
  of each part of $\cal P$.
  We claim that $S$ satisfies the required condition.

  Observe first that each part $A\in\cal P$ forms a clique 
  in the Gaifman graph of $L'$. 
  Indeed, the relation $<'$ of $L'$ restricted to $A$ 
   either coincides with $<$, if $(A,A)$ is not flipped
   in the $\cal P$-flip producing $L'$ from $L$,
or coincides with $\ge$, if $(A,A)$ is flipped. In any case, $<'$ is a total  relation on $A$.

Let $a,b\in V\setminus S$ be vertices belonging to different parts $A,B$ of $\cal P$, with no element of $S$ between $a$ and $b$.
We show that $a$ and $b$ are either adjacent, or have a common neighbor in the Gaifman graph of $L'$.
By symmetry, suppose that $a<b$.
If the pair $(A,B)$ is not flipped 
in the $\cal P$-flip producing $L'$ from $L$,
then $a<'b$ in $L'$, so $a$ and $b$ are adjacent in the Gaifman graph of $L'$.

Suppose that the pair $(A,B)$ is flipped. Let $m_a:=\max(A)$. Then $a<m_a$ since $m_a\in S$ and $a\notin S$.  As $m_a$ is not between $a$ and $b$,
we have that $b<m_a$. 
Since $(A,B)$ is flipped, we have that 
$m_a<'b$ in $L'$. Since $a$ and $m_a$ are in the same part of $A$,
they are adjacent in the Gaifman graph of $L'$.
 Therefore, $m_a$ is a common neighbor of $a$ and of $b$
in the Gaifman graph of $L'$.
\end{proof}
\begin{corollary}\label{cor:order-flip}
  Let $G=(V,E,<)$ be an ordered graph, and let $G'=(V,E',<')$ be a $k$-flip of $G$,
  in the sense of binary structures. 
  Then there is a $k$-cut-flip $G''=(V,E',S)$ of $G$,
  such that for all $u,v\in V$, if $u$ and $v$
   are connected by a path $\pi$ of total weight at most $r$ in $G''$,
   then $u$ and $v$ are within distance at most $3r+2$ 
   in the Gaifman graph of $G'$.
\end{corollary}
\begin{proof}Let $L=(V,<)$ be the total order underlying $G$, and 
  let $G'=(V,E',<')$ be a $k$-flip of $G$.
   Then $L':=(V,<')$ is a $k$-flip of $L$.
   Apply Lemma~\ref{lem:order-flips}, yielding a set $S$.
    Let $G'':=(V,E',S)$.
   Then $G''$ is a $k$-cut-flip of $G$.
  We check that it satisfies the condition.

  Every path of total weight $0$ 
  can be replaced by a path of length $\le 2$ in the Gaifman graph of $G'$, 
  by Lemma~\ref{lem:order-flips}.
A path of total weight $r$ decomposes into 
at most $r+1$ paths of total weight $0$ and at most $r$ edges of weight $1$.
By replacing the paths of total weight $0$ as above, we get a path of length 
at most $3r+2$.
\end{proof}

\begin{proof}[Proof of Lemma~\ref{lem:ordered flip-width}]
  Fix $G=(V,E,<)$ and $r$.

  To prove the second inequality, 
we show that if $\fw_{3r+2}(G)\le k$ (in the sense of binary structures),
then $\fw_{r}^<(G)\le k$ (in the sense of ordered graphs).
This is done by transferring the strategy (cf. Section~\ref{sec:transfer}),
using Corollary~\ref{cor:order-flip}. 

More precisely, suppose $\fw_{3r+2}(G)\le k$, so the runner wins the  flipper game on $G$,
as a binary structure, of radius $3r+2$ and width $k$.
The flipper copy their winning strategy when playing the radius-$r$ ordered flipper game on $G$, as follows:
Whenever the flipper announces a $k$-flip $G'$ of $G$
in the flipper game, 
then in the ordered flipper game,  the flipper announces the $k$-cut-flip
$G''=(V,E',S)$ of $G$, as given by Corollary~\ref{cor:order-flip}.
By an argument analogous to Lemma~\ref{lem:strategy-transfer},
Corollary~\ref{cor:order-flip} shows that this way, the flipper wins the ordered flipper game on $G$ as an ordered graph, so $\fw_r^<(G)\le k$.

\medskip
  For the first inequality,
  we show that if $(V,E',S)$ is a $k$-cut-flip of $G$,
  then there is a $(k^2+2k)$-flip $G'=(V,E',<')$ of $G$
  such for all distinct $u,v\in V$, if 
  there is no $s\in S$ with $u\le s\le v$ then $u<'v$ or $v<'u$.
  Namely, consider the partition $\cal P$ of $V$ with $|\cal P|\le k$
  such that $(V,E')$ is a $\cal P$-flip of $(V,E)$,
  let $\cal Q$ be the partition of $V$ 
  that partitions $S$ into singletons 
  and $V-S$ into maximal $<$-intervals that are disjoint with $S$,  
  and let $\cal R$ be the common refinement of $\cal Q$ and $\cal R$.
  Then $|\cal Q|\le |\cal P|\cdot (|S|+1)+|S|\le k(k+1)+k=k^2+2k$.
  Define $<'$ so that $u<'v$ if and only if $u$ and $v$ are in the same part of $\cal Q$ and $u<v$.
Now, $G'=(V,E',<')$ is a $\cal Q$-flip of $G$, and has the desired properties.

It follows that a winning strategy for the flipper in the ordered flipper game on $G$ with radius $r$ and width $k$ can be transferred into a winning strategy for the flipper in the flipper game on $G$ with radius $r$ and width $k^2+2k=(k+1)^2-1$, 
by replacing each $k$-cut-flip played by the flipper by the $(k^2+2k)$-flip $G'$ as above.
Hence, $\fw_r(G)\le (\fw_r^<(G)+1)^2-1$, proving the first inequality.
\end{proof}

  \section{Closure under transductions}
\label{app:transductions}
In this section, we show that bounded flip-width is preserved by first-order transductions,
and that bounded $\infty$-flip-width is preserved by $\cmso$ transductions.

We work with $c$-colored graphs $G$.
We consider formulas $\phi(x,y)$ of first-order logic, or $\cmso$.
The same arguments work to other logics with suitable locality properties.

\begin{theorem*}[\ref{thm:interpretations}]
  There is a computable function $T_q\from\N\to\N$ with the following property.
  Fix a radius $r\ge 1$ and a first-order formula $\phi(x,y)$ of quantifier rank $q$ in the signature of $c$-colored graphs, for some $c\ge 0$.
  Set $r':=2^q\cdot r$.
  Then 
for every $c$-colored graph $G$ we have 
\begin{align*}
  \fw_r(\phi(G))\le T_q(\fw_{r'}(G)\cdot c).  
\end{align*}
  In particular, if $\CC$ has bounded flip-width, then $\phi(\CC)$ has bounded flip-width.
\end{theorem*}


Say that a formula $\phi(x,y)$ is \emph{$r$-local}, where $r\in\N$,
if the following condition holds:
there is a finite set $T_\phi$ of \emph{local types} such that for every colored graph $G$
each vertex $v$ of $G$ can be labelled by an element $\ltp(v)\in T_\phi$
in such a way that 
for any pair of vertices $(a,b)$ with distance larger than $r$ in $G$,
whether or not $\phi(a,b)$ holds in $G$ depends only on $\ltp(a)$ and  $\ltp(b)$. More precisely, there is
a binary relation $\Phi\subset T_\phi\times T_\phi$ (which may depend on $G$) such that
for all vertices $a,b$ with $\dist(a,b)>r$,
\[G\models \phi(a,b)\iff (\ltp(a),\ltp(b))\in \Phi.\]

We say that $\phi(x,y)$ is $\infty$-local if the above condition holds,
where instead of $\dist(a,b)>r$ we require that $a$ and $b$ are in different connected components of the graph.

\begin{fact}\label{fact:fo-locality}\szfuture{cite}
  Fix $c\ge 0$ and consider the signature of $c$-colored graphs.
  Every formula $\phi(x,y)$ of first-order logic
  is $2^q$-local, where $q$ is the quantifier rank of $\phi$.
  The number $|T_\phi|$ of local types  is at most
  \[T_q(k):=\underbrace{2^{2^{.^{.{^{.^{2^m}}}}}}}_{{\rm height\,} q}\]
  where $m$ is the number of $c$-colored graphs with vertex set $\set{1,\ldots,q+1}$.
\end{fact}

\begin{fact}\label{fact:cmso-locality}\szfuture{cite}
  In the setting of the previous fact, 
  every formula $\phi(x,y)$ of $\cmso$ 
  is $\infty$-local.
  The number $|T_\phi|$ is again bounded by a number $T_q'(k)$
  that is non-elementary in $q$.
\end{fact}

Theorem~\ref{thm:interpretations} follows easily from the next lemma,
using Lemma~\ref{lem:strategy-transfer}.

\begin{lemma}\label{lem:response}
  Fix $k,c\ge 1$, a first-order formula $\phi(x,y)$ of quantifier rank $q$,   and let $s=2^q$ and $\ell=T_q(k\cdot c)$,
  where $T_q$ is the function from Fact~\ref{fact:fo-locality}.
  Let $G$ be a $c$-colored graph.
  For every $k$-flip $G'$ of the uncolored graph underlying $G$
  there is an $\ell$-flip $\phi(G)'$ of $\phi(G)$
  such that:
  \begin{align}\label{eq:interp}
    \dist_{G'}(u,v)\le s\qquad{\text {for all $uv\in E(\phi(G)')$}.}    
  \end{align}  
\end{lemma}
\begin{proof}
  Let $\cal P=\set{A_1,\ldots,A_s}$ be the partition of $V(G)$ with $s\le k$,
such that $G'$ is a $\cal P$-flip of $G$.
Color the vertices of $G'$ using $k\cdot c$ colors $[k]\times [c]$ so that 
a vertex $v$ has color $(i,j)$ if and only if $v\in A_i$
and $v$ has color $j$ in $G$.
 Below, we treat $G'$ as a relational structure 
 equipped with the edge relation of $G'$, 
 and $k\cdot c$ unary predicates marking the colors of $G'$.
 In particular, for $i=1,\ldots,s$ we can write a quantifier-free formula 
 $A_i(x)$ such that $A_i(v)$ holds in $G'$ if and only if 
 $v\in A_i$.

We now write a formula $\psi(x,y)$ such that for all $a,b\in V(G)$ we have
\begin{align}\label{eq:iff}
  G'\models \psi(a,b)\iff G\models\phi(a,b).  
\end{align}
The formula $\psi(x,y)$ is obtained from $\phi(x,y)$
by replacing each atom $E(z,t)$ by the quantifier-free
formula 
\begin{multline*}
\eps(z,t):=E(z,t)\triangle\alpha(z,t):=(E(z,t)\land\neg\alpha(z,t))\\\lor(\alpha(z,t)\land\neg(E(z,t))),
\end{multline*}
where $\alpha(z,t)$ is the disjunction of formulas 
$A_i(z)\land A_j(t)$, for all pairs $i,j\in[k]$
such that the parts $A_i$ and $A_j$ are flipped in the $\cal P$-flip $G'$ of $G$. In particular, $\psi(x,y)$ is a formula of quantifier rank $q$
over the signature of graphs colored with $k\cdot c$ colors.

Hence, by Fact~\ref{fact:fo-locality}, there is a labelling $\ltp\from V(G)\to T_\psi$, for some set of local types $T_\psi$ with $|T_\psi|\le T_q(k\cdot c)=\ell$, and a binary relation $\Phi\subset T_\psi\times T_\psi$,
such that 
\[G'\models \psi(a,b) \quad \iff\quad  (\ltp(a),\ltp(b))\in \Phi\]
for all pairs $a,b\in V(G)$ with distance larger than $s=2^q$ in 
$G'$.
With~\eqref{eq:iff}
this implies that 
\begin{align}\label{eq:flip}
  G\models \phi(a,b) \quad \iff\quad  (\ltp(a),\ltp(b))\in \Phi,  
\end{align}
for all pairs $a,b\in V(G)$ with distance larger than $s$ in $G'$.

Let $\cal Q$ be the partition of $V(G)$ defined by $\ltp$, 
with $\cal Q=\setof{\ltp^{-1}(p)}{p\in T_\psi}\setminus\set{\emptyset}$.
In particular, $|\cal Q|\le \ell$. 
Construct the $\cal Q$-flip $\phi(G)'$ of $\phi(G)$
by flipping two parts $\ltp^{-1}(p),\ltp^{-1}(q)$ 
of $\cal Q$ whenever $(p,q)\in \Phi$ or $(q,p)\in \Phi$.
In particular, if $a$ and $b$ are 
adjacent in $\phi(G)'$, then it must be the case that $\dist_{G'}(a,b)\le s$ by~\eqref{eq:flip}.
\end{proof}



\begin{proof}[Proof of Theorem~\ref{thm:interpretations}]
  Fix a formula $\phi(x,y)$ of quantifier rank $q$,
  and a $c$-colored graph $H$. Let $k=\fw_{sr}(G_0)$,
  where $H_0$ is the uncolored graph underlying $H$.
  Set $s:=2^q$
  and $\ell:=T_q(k\cdot c)$. 
  Denote $G:=\phi(H)$.
  Then Lemma~\ref{lem:response} says that the assumptions of Lemma~\ref{lem:strategy-transfer} hold. 
  Therefore, $\fw_{r}(\phi(H))\le \ell=T_q(k\cdot c)$, as required.
\end{proof}

We now consider the case of $\cmso$-transductions.
\begin{theorem*}[\ref{thm:interpCMSO}]
  Let $\CC$ be a class of bounded $\infty$-flip-width and let $\phi(x,y)$ be a formula of $\cmso$.
  Then $\phi(\CC)$ has bounded $\infty$-flip-width.
\end{theorem*}

The proof of Theorem~\ref{thm:interpCMSO} is the same as the proof of Theorem~\ref{thm:interpretations}, with the difference that the use of Lemma~\ref{lem:response} is replaced by Lemma~\ref{lem:responseCMSO} below.
\begin{lemma}\label{lem:responseCMSO}
  Fix $k\ge 1$, a $\cmso$ formula $\phi(x,y)$ of quantifier rank $q$, and let $\ell=T_q'(k)$, where $T_q'(k)$ is the function from Fact~\ref{fact:cmso-locality}. 
  For every $k$-flip $G'$ of a graph $G$
  there is an $\ell$-flip $\phi(G)'$ of $\phi(G)$
  such that 
  \begin{align}
    \dist_{G'}(u,v)<\infty\qquad\text{for $uv\in E(\phi(G)')$}.
    \label{eq:interpCMSO}
  \end{align}  
  \end{lemma}In turn, the proof of Lemma~\ref{lem:responseCMSO}
  is the same as the proof of Lemma~\ref{lem:response},
  with the difference that the use of Fact~\ref{fact:fo-locality}
  is replaced by the use of Fact~\ref{fact:cmso-locality}.

\section{Structurally nowhere dense classes have almost bounded flip-width}\label{app:snd-subpoly}
In this Appendix, we prove Theorem~\ref{thm:quasi-bushes-vc}
and Lemma~\ref{lem:fw-qf-vc}.
Theorem~\ref{thm:quasi-bushes-vc} is proved in Section~\ref{app:quasi-bushes-vc}.
Lemma~\ref{lem:fw-qf-vc} is proved in Section~\ref{sec:qff}.
Theorem~\ref{thm:snd-are-subpoly} is an immediate consequence of those two results, as argued in Section~\ref{sec:subpoly}.

\subsection{Proof of Theorem~\ref{thm:quasi-bushes-vc}}
\label{app:quasi-bushes-vc}
In Section~\ref{app:quasi-bushes-vc} we prove Theorem~\ref{thm:quasi-bushes-vc}, restated below.
\begin{theorem*}[\ref{thm:quasi-bushes-vc}]
  Let $\CC$ be a structurally nowhere dense graph class.
  There is a signature $\Sigma$ consisting of unary and binary relation symbols and a unary function symbol,
   a class $\BB$ of $\Sigma$-structures
  such that the class of Gaifman graphs of the structures in $\BB$ is almost nowhere dense, and a symmetric quantifier-free formula $\phi(x,y)$, such that every $G\in\CC$  is an induced subgraph of $\phi(B)$ for some $B\in \BB$ with $|B|=O(|G|)$. Moreover,  $\VCdim(B)<\infty$.
\end{theorem*}

The proof of Theorem~\ref{thm:quasi-bushes} is based on the proof of \cite[Theorem 3]{bushes-lics} (see also \cite{bushes-arxiv}),  stated below.

A \emph{quasi-bush} $B$ is a rooted tree $T$ 
equipped with:
\begin{itemize}\item 
  a set $D$ of directed edges from the leaves of $T$ to inner nodes of $T$, called \emph{pointers};  every leaf has a pointer to the root of $T$, 
\item a labelling function $\lambda\from \leaves(T)\to \Lambda$, where $\Lambda$ is a finite set of labels,
\item a labelling function $\lambda^D\from D\to 2^{\Lambda}$.
\end{itemize}
A quasi-bush $B$ defines a directed graph $G(B)$
whose vertices are the leaves of $T$ and directed edges 
$(u,v)$ such that $u,v$ are distinct leaves of $T$,
and the closest ancestor $w$ of $u$ such that $(v,w)\in D$ satisfies $\lambda(u)\in \lambda^D((v,w))$. In particular, for $G(B)$ to be equal to an undirected graph, we require that the directed edge relation of $G(B)$ is symmetric.

Say that a class $\BB$ of quasi-bushes is almost nowhere dense if the class of underlying graphs 
(where we keep the edges of the tree $T$ and turn the  pointers in $D$ into undirected edges) form an almost nowhere dense graph class.

\begin{theorem}\textup{\cite[Theorem 3]{bushes-lics}}\label{thm:quasi-bushes}
  Let $\DD$ be a structurally nowhere dense graph class.
  Then there are $d,\ell\in\N$, and an almost nowhere dense class $\BB$ of quasi-bushes, each of depth at most $d$ and using at most $\ell$ labels, such that for every $G\in\DD$ there is some quasi-bush $B\in\BB$ with $G(B)=G$.
\end{theorem}

We now prove Theorem~\ref{thm:quasi-bushes-vc}.
\begin{proof}[Proof of Theorem~\ref{thm:quasi-bushes-vc}]
  Fix $d,\ell$ as in Theorem~\ref{thm:quasi-bushes}.
  Every quasi-bush $B\in\cal B$ may be viewed as a $\Sigma$-structure,
  over a fixed signature $\Sigma$ consisting of:
  \begin{itemize}
    \item a unary function, interpreted in $B$ as the parent function of the tree, and mapping the root to itself,
    \item $\ell$ unary relation symbols, interpreted in $B$ as the labels of the leaves of $T$ according to the function $\lambda\from \leaves(T)\to \Lambda$,
     where $|\Lambda|\le \ell$,
    \item $2^\ell$ binary relation symbols $D_M$, for $M\subset \Lambda$, where $D_M(u,v)$ holds for a leaf $u$ and inner node $v$ and if and only if $\lambda^D((u,v))=M$.
  \end{itemize}

  It is straightforward to construct a quantifier-free formula 
 $\gamma_0(x,y)$
 such that for every $B\in\BB$ and leaves $u,v$ of $B$ we have
$B\models \gamma_0(u,v)$ if and only if 
the lowest ancestor $w$ of $u$ such that $(v,w)\in D$ 
satisfies $\lambda(u)\in\lambda^D((v,w))$.
Let $\gamma(x,y)$ be the symmetric formula $\gamma_0(x,y)\lor\gamma_0(y,x)$.
Then for every quasi-bush $B\in\BB$ such that $G(B)$ is
 an undirected graph $G$,
  we have that $G$ is the subgraph of $\gamma(B)$ induced by the leaves of $B$.

By Theorem~\ref{thm:quasi-bushes}, for every $G\in\DD$ there 
is a quasi-bush $B\in\BB$ such that $G(B)=G$, and hence, $G$ is the subgraph of $\gamma(B)$ induced by the leaves of $B$. Since there is a tree $T$ with leaves $V(G)$, depth at most $d$, and $V(T)=V(B)$, it follows that $|B|\le d\cdot |G|=O(|G|)$. The class $\BB$ is almost nowhere dense. This proves the statement of Theorem~\ref{thm:quasi-bushes-vc}, apart from the `moreover' part.

\medskip
It remains to argue that the class $\BB$, viewed as $\Sigma$-structures as described above, has bounded VC-dimension. Therefore, we need to show that each of the binary relations $D_M$, for $M\subset \Lambda$, has VC-dimension bounded by a constant independent of $B\in\BB$ and of $M\subset \Lambda$.
To do this, we inspect how the set of pointers $D$ and labeling function $\lambda^D$ are defined in the construction in \cite{bushes-arxiv}.
The key property of the construction, from which the bounds on the VC-dimension follows,
is encapsulated in the claim below.

Let $\BB$ be as constructed in the proof of \cite[Theorem 3]{bushes-arxiv}.
\begin{claim}\label{claim:tuples}  
  There is a nowhere dense graph class $\CC$
  and numbers $s,q\ge 1$ with the following property.
  For every quasi-bush $B\in\BB$  
  there is a graph $G\in\CC$ with $\leaves(B)=V(G)$,
   a 
  function $\beta\from V(B)\to V(G)^s$,
  and for each  $M\in 2^{\Lambda}$
  a formula 
  $\psi_M(x_0,x_1,\ldots,x_s)$ of quantifier rank at most $q$, 
  such that 
  \begin{align}\label{eq:beta}
    (u,w)\in D_M\quad\iff\quad G\models \psi_M(u,\beta(w))\quad\text{for all $u,w\in V(B)$.}
  \end{align}  
\end{claim}
We first show how
Claim~\ref{claim:tuples} implies 
the bound on the VC-dimension of $\BB$.

Let $\Delta$ be the set of all formulas $\psi(x_0,\ldots,x_s)$ of quantifier rank $q$,
where $q$ and $s$ are as in Claim~\ref{claim:tuples}.
Since $\Delta$ is finite (up to equivalence), 
it follows from Fact~\ref{fact:nd-nip} that 
there is a bound $k$ depending only on $q,s$ and the class $\CC$
such that for all $G\in\CC$ and every formula $\psi\in\Delta$,
the binary relation $R_{G}^{\psi}\subset V(G)\times V(G)^s$ 
(as considered in Fact~\ref{fact:nd-nip})
has VC-dimension at most $k$. In particular, $R_G^{\psi_M}$ has 
VC-dimension at most $k$, for all $G\in\CC$ and $\psi_M$ as in Claim~\ref{claim:tuples}. 
Then \eqref{eq:beta} is restated as follows:
\[(u,w)\in D_M\quad\iff\quad (u,\beta(w))\in R_G^{\psi_M}\text{\quad for all $u,w\in V(B)$.}\]
It follows \szfuture{argue?} that $D_M$ has VC-dimension bounded 
by the VC-dimension of the binary relation $R_{G}^{\psi_M}$, so at most $k=O_{q,s}(1)$.
Since $q$ and $s$ are independent of $B\in\BB$, the `moreover' part of Theorem~\ref{thm:quasi-bushes-vc} follows.

\medskip
It therefore remains to 
analyse the construction from \cite{bushes-arxiv}, and 
argue that Claim~\ref{claim:tuples} holds.

First, Theorem~\ref{thm:quasi-bushes} is proved (in \cite{bushes-arxiv})
in the special case when $\DD=\phi(\CC)$ for some nowhere dense class $\CC$
of colored graphs and first-order formula $\phi(x,y)$ involving color predicates.
This is done in \cite[Theorem 28]{bushes-arxiv}.
In general, in Theorem~\ref{thm:quasi-bushes},
 $\DD$ is contained in the hereditary closure 
of $\phi(\CC)$, rather than in $\phi(\CC)$ itself.

The general case is reduced to the special 
case at the end of Section 5 in~\cite{bushes-arxiv},
as follows.
It is shown that there is a coloring $\wh\CC$ 
of $\CC$ (that is, a class of $k$-colored graphs for some $k\ge 1$,
where each underlying graph belongs to $\CC$), a formula $\wh\phi(x,y)$, and constants $c,d>0$,
such that for every $G\in \DD$ there is 
some colored graph $\wh H\in \wh \CC$ 
such that $\wh\phi(\wh H)$ 
contains $G$ as a subgraph induced by some set $A\subset V(\wh H)$, and moreover 
$|\wh H|\le c|G|^d$.
A quasi-bush $B$ is constructed for the graph $\wh\phi(\wh H)$
using the special case of Theorem~\ref{thm:quasi-bushes}
applied to the class $\wh\phi(\wh\CC)$.
To get a quasi-bush $B'$ for $G=\wh\phi(\wh H)[A]$, the nodes of $B$ that have no descendants in $A\subset \leaves(B)$ are removed. It is then argued that the class of quasi-bushes $B'$
obtained in this way, for each $G\in\DD$, satisfies the conditions of Theorem~\ref{thm:quasi-bushes}. For us here, it only matters that 
$B'$ is a quasi-bush obtained from restricting the quasi-bush $B$ as obtained in the special case of Theorem~\ref{thm:quasi-bushes}. It is immediate that 
the VC-dimension of $B'$ is bounded by the VC-dimension of $B$,
so it is enough to argue that the quasi-bushes constructed in the special case of Theorem~\ref{thm:quasi-bushes} have bounded VC-dimension.

\medskip
We may therefore focus on analyzing the proof of the special case of Theorem~\ref{thm:quasi-bushes} (stated as Theorem 28 in \cite{bushes-arxiv}) where it is assumed 
that $\DD=\phi(\CC)$ for some nowhere dense class $\CC$ of colored graphs and first-order formula $\phi(x,y)$. We argue that Claim~\ref{claim:tuples}
holds in this case.

\paragraph{Types}
Before analysing the construction, 
we recall the following notion. Fix the signature $\Sigma$ 
consisting of the edge relation symbol $E$ and unary predicates corresponding to the colors of the graphs in $\CC$.
Fix $q,m\ge 0$.
For a colored graph $H$ and $m$-tuple $\tup v\in V(H)^m$, 
the \emph{quantifier rank $q$ type} of $\tup v$,
denoted $\tp^q_H(\tup v)$,
is the of all first-order formulas $\phi(x_1,\ldots,x_m)$ 
over the signature $\Sigma$
of quantifier rank $q$ such that $H\models \phi(\tup v)$.
Let $\Gamma_q^m$ denote the set of all possible quantifier rank $q$ types of $m$ tuples:
\[\Gamma_q^m:=\setof{\tp^q_H(\tup v)}{H\text{ -- colored graph}, \tup v\in V(H)^m}.\]

The following fact is well known, and follows from the observation that up to equivalence, there are finitely many formulas $\phi(x_1,\ldots,x_m)$ of quantifier rank $q$.
\begin{fact}\label{fact:types}
  The set $\Gamma_q^m$ is finite. 
  For every type $\tau\in\Gamma_q^m$ there is a first-order formula, denoted 
  $\tau(x_1,\ldots,x_m)$, such that for every colored graph $H$ and tuple 
  $\tup v\in V(H)^m$, we have:
\[\tp^q_H(\tup v)=\tau\quad\iff\quad H\models\tau(\tup v).\]
\end{fact}

\paragraph{Analysis of the proof}
We now go through the proof of 
Theorem 28 in \cite{bushes-arxiv},
that is, 
Theorem~\ref{thm:quasi-bushes} 
in the case where $\DD=\phi(\CC)$ for some class $\CC$ of colored graphs 
and first-order formula $\phi(x,y)$. We argue that Claim~\ref{claim:tuples} holds.

For each graph $G\in\CC$, 
a quasi-bush $B$ with $G(B)=\phi(G)$
is constructed as follows.
First, an \emph{$r$-separator quasi-bush $T$} for $G$
is constructed, for some number $r$ depending on the quantifier rank of $\phi(x,y)$. An $r$-separator quasi-bush is a tree $T$ with leaves $V(G)$,
equipped with:
\begin{itemize}
  \item  a set $D\subset V(T)\times V(T)$ of \emph{pointers}, where each pointer $(u,w)$
   points from some leaf $u$ of $T$ to some inner node $w$ of $T$ (and each leaf points to the root), and
  \item a function $\alpha$ mapping 
    each inner node $v$ of $T$ to a set $\alpha(v)\subset V(G)$ with the following property.
    For every two leaves $u,v$ of $T$, and node $w$ 
which is the lowest ancestor of $v$ such that $(v,w)\in D$,
the set $\alpha(w)$ is an $r$-separator between $u$ and $v$ in $G$,
that is, every path from $u$ to $v$ of length at most $r$
in $G$ passes through $\alpha(w)$.
\end{itemize}
Crucially (see Lemma 36 and second item in Lemma 33 in \cite{bushes-arxiv}), the size of the set $\alpha(w)$ is bounded by a constant $m$ (depending only on $\CC$ and $\phi$).
 Below, the set $\alpha(w)$ is treated as a tuple of length at most $m$,
by enumerating its elements according to any fixed order on $V(G)$.

Next, an $r$-separator quasi-bush $T$ is converted into a quasi-bush $B$,
by assigning a label $\lambda(v)\in\Lambda$ 
(where $\Lambda$ is some finite set) to each leaf $v$ of $T$, 
and a label $\lambda^D((u,w))\in 2^\Lambda$  to each pointer $(u,w)\in D$.
For $M\in 2^\Lambda$, let $D_M$ 
denote the set of pointers $(u,w)\in D$ with $\lambda^D((u,w))=M$. 

The following statement is immediate from the construction (Proof of Theorem 28 in \cite{bushes-arxiv}):
There is a number $q$ (depending only on $\phi$ and $\CC$), such that 
the label $\lambda^D((u,w))$ of a pointer $(u,w)\in D$
depends only on $\tp^q_G(u\alpha(w))$, where $\alpha(w)$ is viewed as a tuple.
Hence, for every label $M\in 2^\Lambda$,
whether or not a pointer $(u,w)\in D$ 
belongs to $D_M$, depends only on $\tp^q_G(u\alpha(w))$.
That means that for each $M\in 2^\Lambda$ there is a set $\Phi_M\subset \Gamma^{m+1}_q$ 
such that for all $(u,w)\in D$ we have 
\[(u,w)\in D_M \quad\iff \quad \tp^q_G(u\alpha(w))\in \Phi_M.\]
Let $\psi_M^0(x,\tup y)$ denote the disjunction 
of all formulas $\tau(x,\tup y)$ representing the types $\tau\in \Phi_M$ (as described in Fact~\ref{fact:types}).
We conclude that the following claim holds.

\begin{claim}\label{claim:dependson}
  There is a number $q$ depending only on $\CC$ and $\phi$ such that the following holds.
  For every $M\in 2^\Lambda$
  there is a formula $\psi_M^0(x,\tup y)$ of quantifier rank $q$ such that 
for every pointer $(u,w)\in D$
we have 
\[(u,w)\in D_M\quad\iff\quad G\models \psi_M^0(u,\alpha(w)).\]
\end{claim}
We also need to argue that the set $D$ can be defined by a first-order formula, as made precise below.
\begin{claim}\label{claim:D}  
There is a number $t$ depending only on $\phi$ and $\CC$,
a function $\delta\from V(T)\to V(G)^t$
and a first-order formula $\psi_D(x_0,\ldots,x_t)$, such that the following holds 
for all $u,w\in V(B)$:
\begin{align}\label{eq:psiD}
  (u,w)\in D\quad\iff\quad G\models \psi_D(u,\delta(w)).
\end{align}
\end{claim}

First, we show how Claim~\ref{claim:dependson} and Claim~\ref{claim:D}
imply Claim~\ref{claim:tuples}.
From the two claims it follows 
that for each $M\in 2^\Lambda$ and pair $(u,w)\in V(T)\times V(T)$
we have that 
\[(u,w)\in D_M \quad\iff\quad G\models \psi_D(u,\delta(w))\land \psi_M^0(u,\alpha(w)).\]
For a node $w$ of $T$, let $\beta(w)$ be the concatenation 
of the tuples $\alpha(w)$ and $\delta(w)$.
For $M\in 2^\Lambda$ define 
\[\psi_M(x,\tup y,\tup z):= \psi_D(x,\tup z)\land \psi_M^0(x,\tup y).\]
Then for each pair $(u,w)\in V(T)\times V(T)$
we have that 
\[(u,w)\in D_M \quad\iff\quad G\models \psi_M(u,\beta(w)).\]
This proves Claim~\ref{claim:tuples}, assuming 
Claim~\ref{claim:dependson} and Claim~\ref{claim:D}.

\medskip
Claim~\ref{claim:D} is argued below, by analysing the construction the  $r$-separator quasi-bush for $G$.
Definition 32 of \cite{bushes-arxiv} associates to each vertex $v$ of $G$
and number $k\ge 0$ two sets of vertices of $G$, denoted $M_r^k[v]$ and $S_r^k[v]$.
Those sets are treated as tuples according to some fixed enumeration of $V(G)$.
It is shown (see Lemma 36 of \cite{bushes-arxiv}) that those sets have size 
bounded by some constant $d$.
According to Definition 37 of \cite{bushes-arxiv},
the nodes $w$ of the $r$-separator quasi-bush $T$ are sets of the form $M^k_r[v]$, for all $v\in V(G)$ and all $k\le d$.
And the pointers $D$ of $T$ 
are defined so that $(u,M^k_r[v])\in D$,
for $v\in V(G)$ and $k\ge 1$,
if and only if $S^{k-1}_r[v]$ does not $r$-separate $u$ and $v$ in $G$.

For a node $w=M^k_r[v]$ of $T$,
define $\delta(w)$ as the concatentation of the following tuples:
\begin{itemize}
  \item $v$ (where $v$ is arbitrarily chosen so that $M^k_r[v]=w$),
  \item $S^{k-1}_r[v]$, padded to a tuple of length $d$.
\end{itemize}
Let $\psi_D(x,y,\tup z)$ with $|\tup z|=d$ 
be a first-order formula expressing 
``$\tup z$ does not $r$-separate $x$ and $y$.''
Then, by definition of $D$, we have that 
for every pair $u,w\in V(T)$ 
\eqref{eq:psiD} holds.
This proves Claim~\ref{claim:D}.
\end{proof}

\subsection{Quantifier-free interpretations with function symbols}
\label{sec:qff}
In this section, we prove Lemma~\ref{lem:fw-qf-vc}, which is repeated below.
\begin{lemma*}[\ref{lem:fw-qf-vc}]
  Let $\Sigma$ be a signature consisting of unary and binary relation symbols, and unary function symbols. Fix $k,r\ge 0$,
  and a symmetric quantifier-free $\Sigma$-formula  $\phi(x,y)$. There are numbers $p=O_\phi(k)$ and 
  $r'=O_\phi(r)$ such that the following holds.
  Let $B$ be a $\Sigma$-structure of VC-dimension at most $k$ and $G_B$ be its Gaifman graph. Then 
   \[\fw_r(\phi(B))\le O(\copw_{r'}(G_B))^{p}.\]
\end{lemma*}

In Section~\ref{sec:qff}, 
fix a signature $\Sigma$ consisting of unary relation symbols, binary relation symbols, and unary function symbols. All considered formulas are over this signature, and are quantifier-free.

The \emph{depth} of a term $t(x)$ is the nesting of function symbols occurring in $t$,
where the term $x$ has depth $0$, $f(x)$ has depth $1$, etc. 
The depth of a quantifier-free formula 
is the maximal depth of a term occurring in it.
Note that there are $O_d(1)$ terms and atomic formulas of depth~$d$.

We first prove the following lemma.
    For a set $S\subset V(G)$ and two vertices $u,v\in V(B)$,
    let $\dist_S(u,v)$ denote the distance between $u$ 
    and $v$ in the subgraph of the Gaifman graph of $B$
    obtained by isolating $S$, that is, removing the edges incident to vertices in $S$.

    \begin{lemma}\label{lem:local-types-over-S}
      Fix $k,d\ge 0$
      and a quantifier-free $\Sigma$-formula $\phi(x,y)$ of depth at most~$d$.
      Then there is a number $m=O_d(k)$ such that 
      for every $\Sigma$-structure $B$ 
      of VC-dimension at most $k$ and set $S\subset V(B)$
      there is a set $T$ of labels with $|T|\le O_\phi(|S|^m)$,
       a binary relation $\Phi\subset T\times T$,
and a function $\lambda\from V(B)\to T$,
  such that for all vertices $u,v\in V(G)$ 
  with $\dist_{S}(u,v)>2d+1$ we have
  \[B\models \phi(u,v) \quad\iff\quad (\lambda(u),\lambda(v))\in \Phi.\]
  \end{lemma}
  \begin{proof}
    
    Fix $d\ge 0$.
For a vertex $v\in V(B)$
and set of vertices $S\subset V(G)$,
define 
  the \emph{atomic $S$-type of depth $d$} of $v$,
  denoted $\atp^d(v/S)$,
as
the set of all pairs consisting of an atomic formula
$\alpha(x,y)$ of depth at most $d$ and 
an element $s\in S$, such that $\alpha(v,s)$ holds in $B$. 
For a set $S\subset V(G)$, define
\[T^d(S):=\setof{\atp^d(v/S)}{v\in V(G)}.\]


\begin{claim}\label{claim:ntypes-qf}
  Fix $d,k\ge 0$. There is a number $m=O_d(k)$ such that for every structure $B$ with  $k=\VCdim(B)$ and set $S\subset V(B)$ we have 
  \[|T^d(S)|=O(|S|)^m.\] 
\end{claim}
\begin{proof}
We first prove the claim in the case $d=0$.
We have that $\atp^0(v/S)$ is determined by the following data:
\begin{itemize}
  \item the set of 
  atomic formulas $\alpha(x)$ of depth $0$ 
  such that $\alpha(v)$ holds in $G$, 
  
  \item the sets $R(v;S)$ and 
  and $R(S;v)$, for each binary relation symbol $R\in \Sigma$,
  \item the set of elements  $s\in S$ such that $s=v$; this set is either empty, or a singleton.
\end{itemize}
There are $O(1)$ formulas of depth $0$,
and for each binary relation symbol $R\in\Sigma$,
we have 
\[|\setof{R(v;S)}{v\in V(B)}|\le O(|S|^k)\]
and 
\[|\setof{R(S;v)}{v\in V(B)}|\le O(|S|^k)\]
by the Sauer-Shelah-Perles lemma and the assumption that 
$\VCdim(B)\le k$.
We get that $|T^d(S)|\le O(|S|^{2k|\Sigma|+1})=|S|^{O(k)}$,
since we consider $\Sigma$ as fixed.

We now consider the case $d>0$.
  Let $S^{(d)}$ 
  denote the set of vertices that can be obtained in $B$ by applying 
   a term  $t(x)$ of depth at most $d$ to a vertex $s\in S$:
   \[S^{(d)}:=\setof{t(s)}{s\in S, \text{$t(x)$ is a term of depth ${\le}d$}}.\]
   Then $|S^{(d)}|\le O_d(|S|)$, as there are $O_d(1)$ terms of depth at most $d$.

   For a vertex $v\in V(B)$,
$\atp^d(v/S)$ is uniquely determined by the tuple
\[(\atp^0(t(v)/S^{(d)}))_{t(x)}\]
where $t(x)$ ranges over all terms of depth at most $d$.
As there are $O_d(1)$ such terms $t(v)$,
and $|S^{(d)}|\le O_d(|S|)$
the conclusion follows from the case $d=0$ considered earlier.
\end{proof}

\begin{claim}\label{claim:depends-qf}Fix $d\ge 0$, and 
  let $\phi(x,y)$ be a quantifier-free formula of depth at most~$d$.
  Fix a $\Sigma$-structure $B$ and a set $S\subset V(B)$.
  For all vertices $u,v\in V(B)$ with 
       $\dist_{S}(u,v)>2d+1$, whether or not $\phi(u,v)$ holds in $B$, depends only on $\atp^d(u/S)$ and $\atp^d(v/S)$. More precisely, there is a binary relation $\Phi\subset T^d(S)\times T^d(S)$ 
       such that for all vertices $u,v\in V(B)$ with 
       $\dist_{S}(u,v)>2d+1$
       we have 
       \[B\models\phi(u,v)\iff (\atp^d(u/S),\atp^d(v/S))\in \Phi.\]
    \end{claim}
  \begin{proof} 
    It is enough to consider the case when $\phi(x,y)$ is an atomic formula, since if the statement holds for two formulas $\phi(x,y)$ and $\psi(x,y)$ of nesting depth at most $d$,
    then it also holds for $\neg\phi(x,y)$ and for $\phi(x,y)\lor\psi(x,y)$.

    Thus assume that $\phi(x,y)$ is of the form 
    \[\phi(x,y)\equiv R(t(x),t'(y)),\]
    where $R$ is either a binary relation symbol occurring in $\Sigma$, or is the equality relation, and $t(x)$ and $t'(y)$ are two terms of depth at most $d$.

    Fix two vertices $u,v\in V(B)$ with $\dist_B(u,v)>2d+1$.
    We show how to determine whether $\phi(u,v)$ holds in $B$,
  from the information contained in $\atp^d(u/S)$ and $\atp^d(v/S)$.
    
  Suppose first that there is a subterm $t_0(x)$
  of $t(x)$ such that $t_0(u)\in S$. 
  Note that whether this is the case can be determined 
from $\tp^d(u/S)$.

  Let $s=t_0(u)\in S$, and 
  let $t_1(z)$ be a term such that $t_1(t_0(x))=t(x)$. In particular, 
  $t(u)=t_1(s)$.
Then 
\[B\models \phi(u,v)\iff  B\models R(t_1(s),t'(v)).\]
Since $R(t_1(x),t'(y))$ is an atomic formula of depth at most $d$,
whether or not $R(t_1(s),t'(v))$  holds in $B$ 
is determined  by $\atp^d(v/S)$.
Hence, in this case, whether or not $B\models \phi(u,v)$, is determined by $\atp^d(v/S)$.

Similarly, if there is a subterm $t_0'(y)$ of $t'(y)$ 
such that $t_0'(u)\in S$, then whether or not $B\models \phi(u,v)$, is determined by $\atp^d(u/S)$.
Moreover, whether this case holds can be determined from $\tp^d(u/S)$.

We show that if neither of the two cases holds,
then $B\models\neg \phi(u,v)$.
First, note that $\dist_S(u,t(u))\le d$,
as witnessed by the path formed by $u,f_1(u), f_2(f_1(u)),\ldots,t(u)$, where $t(x)=f_d(\ldots(f_1(x))\ldots)$.
Similarly, $\dist_S(v,t'(v))\le d$.
Since $\dist_S(u,v)>2d+1$, by the triangle inequality we have
$\dist_S(t(u),t'(v))>1$. As  $t(u),t'(v)\notin S$,
it follows that $t(u)$ and $t'(v)$ are non-adjacent in the Gaifman graph of $G$. We conclude that $B\models \neg R(t(u),t'(v))$,
equivalently, $B\models \neg \phi(u,v)$.

The claim follows.
  \end{proof}
Lemma~\ref{lem:local-types-over-S}
follows immediately from Claim~\ref{claim:ntypes-qf} and Claim~\ref{claim:depends-qf}, by taking $T=T^d(S)$ and $\lambda(v)=\atp^d(v/S)$.
  \end{proof}

  From Lemma~\ref{lem:local-types-over-S} we get the following.
  \begin{corollary}\label{cor:flips-over-S}
    Fix $d,k\ge 0$
      and a symmetric quantifier-free $\Sigma$-formula $\phi(x,y)$ of depth at most $d$.
      There is a number $m=O_d(k)$ with the following property.
      For every $\Sigma$-structure $B$ 
      of VC-dimension at most $k$ and set $S\subset V(B)$
      there is a 
      $ O_\phi(|S|^{m})$-flip $\phi(B)'$ of $\phi(B)$       
  such that for every vertex $v\in V(B)$ 
  we have 
  \begin{align}\label{eq:flips-over-S}
    \dist_S(u,v)\le 2d+1\qquad\text{for $u,v\in E(\phi(B)')$,}
  \end{align}
where $\dist_S(\cdot,\cdot)$ denotes the distance 
in the Gaifman graph of $B$ with the vertices in $S$ isolated.
  \end{corollary}
  \begin{proof}
    Let $\lambda\from V(B)\to T$ be the labelling from Lemma~\ref{lem:local-types-over-S}. Let $\cal P$ be the partition of $V(B)$
    into parts $\lambda^{-1}(a)$, for $a\in T$. Then $|\cal P|\le |T|=O_\phi(|S|^m)$ for some $m=O_d(k)$.
    Define $\phi(B)'$ as the $\cal P$-flip of $\phi(B)$,
    obtained by flipping two parts $P,Q$ of $\cal P$ if and only if there are $u\in P$, $v\in Q$ such that $\dist_S(u,v)>2d+1$ and $B\models \phi(u,v)$. By construction, if $u$ and $v$ are adjacent in $\phi(B)'$, then $\dist_S(u,v)\le 2d+1$.
    The conclusion follows.
  \end{proof}

  Lemma~\ref{lem:fw-qf-vc} now follows along the same lines as Theorem~\ref{thm:interpretations}.
\begin{proof}[Proof of Lemma~\ref{lem:fw-qf-vc}]
  Let $m$ be as in Corollary~\ref{cor:flips-over-S}. 
  We fix a winning strategy 
  of the cops in the Cops and Robber game of radius $r':=r(2d+1)$ and width $\ell:=\copw_{r'}(G)$,
  and transfer this strategy to the flipper game of radius $r$ and width $O(\ell^m)$ on $\phi(B)$,
  so that whenever the cops  announce a new set $S\subset V(B)$ of vertices in the Cops and Robber game, 
  then in the flipper game the flipper announces
  the $O(|S|^m)$-flip $\phi(B)'$ of $\phi(B)$, as obtained by Corollary~\ref{cor:flips-over-S}. 
  It follows from~\eqref{eq:flips-over-S} and Lemma~\ref{lem:strategy-transfer}
  that this yields a winning strategy in the flipper game. Hence, $\fw_r(\phi(B))=O(\ell^m)$.
\end{proof}

\end{appendices}

\addcontentsline{toc}{section}{References}
\bibliographystyle{alphaurl}
\bibliography{bib}

\newcommand{\etalchar}[1]{$^{#1}$}
\begin{thebibliography}{BGOdM{\etalchar{+}}22}

\bibitem[AA14]{AdlerA14}
Hans Adler and Isolde Adler.
\newblock Interpreting nowhere dense graph classes as a classical notion of model theory.
\newblock {\em Eur. J. Comb.}, 36:322--330, 2014.
\newblock \href {https://doi.org/10.1016/j.ejc.2013.06.048} {\path{doi:10.1016/j.ejc.2013.06.048}}.

\bibitem[AAL21]{functionality}
Bogdan Alecu, Aistis Atminas, and Vadim Lozin.
\newblock Graph functionality.
\newblock {\em Journal of Combinatorial Theory, Series B}, 147:139--158, 2021.
\newblock \href {https://doi.org/10.1016/j.jctb.2020.11.002} {\path{doi:10.1016/j.jctb.2020.11.002}}.

\bibitem[ACLZ15]{implicit-representations}
A.~Atminas, A.~Collins, V.~Lozin, and V.~Zamaraev.
\newblock Implicit representations and factorial properties of graphs.
\newblock {\em Discrete Mathematics}, 338(2):164--179, 2015.
\newblock \href {https://doi.org/10.1016/j.disc.2014.09.008} {\path{doi:10.1016/j.disc.2014.09.008}}.

\bibitem[Adl08]{Adler2008-ADLAIT}
Hans Adler.
\newblock An introduction to theories without the independence property.
\newblock available online, 2008.

\bibitem[AF84]{cop-number}
M.~Aigner and M.~Fromme.
\newblock A game of cops and robbers.
\newblock {\em Discrete Applied Mathematics}, 8(1):1--12, 1984.
\newblock \href {https://doi.org/10.1016/0166-218X(84)90073-8} {\path{doi:10.1016/0166-218X(84)90073-8}}.

\bibitem[AM15]{https://doi.org/10.1002/jgt.21791}
Noga Alon and Abbas Mehrabian.
\newblock Chasing a fast robber on planar graphs and random graphs.
\newblock {\em Journal of Graph Theory}, 78(2):81--96, 2015.
\newblock \href {https://doi.org/10.1002/jgt.21791} {\path{doi:10.1002/jgt.21791}}.

\bibitem[BCK{\etalchar{+}}22]{tww8}
{\'{E}}douard Bonnet, Dibyayan Chakraborty, Eun~Jung Kim, Noleen K{\"{o}}hler, Raul Lopes, and St{\'{e}}phan Thomass{\'{e}}.
\newblock Twin-width {VIII:} delineation and win-wins.
\newblock In Holger Dell and Jesper Nederlof, editors, {\em 17th International Symposium on Parameterized and Exact Computation, {IPEC} 2022, September 7-9, 2022, Potsdam, Germany}, volume 249 of {\em LIPIcs}, pages 9:1--9:18. Schloss Dagstuhl - Leibniz-Zentrum f{\"{u}}r Informatik, 2022.
\newblock \href {https://doi.org/10.4230/LIPIcs.IPEC.2022.9} {\path{doi:10.4230/LIPIcs.IPEC.2022.9}}.

\bibitem[BDG{\etalchar{+}}22]{boundedLocalCliquewidth}
{\'E}douard Bonnet, Jan Dreier, Jakub Gajarsk{\'y}, Stephan Kreutzer, Nikolas M{\"a}hlmann, Pierre Simon, and Szymon Toru{\'n}czyk.
\newblock {Model Checking on Interpretations of Classes of Bounded Local Cliquewidth}.
\newblock In {\em {37th Annual ACM/IEEE Symposium on Logic in Computer Science, LICS 2022}}. ACM, 2022.
\newblock URL: \url{https://hal.archives-ouvertes.fr/hal-03751012}.

\bibitem[BFLP24]{tww-neighborhood-complexity}
{\'E}douard Bonnet, Florent Foucaud, Tuomo Lehtil{\"a}, and Aline Parreau.
\newblock Neighbourhood complexity of graphs of bounded twin-width.
\newblock {\em European Journal of Combinatorics}, 115:103772, 2024.
\newblock \href {https://doi.org/10.1016/j.ejc.2023.103772} {\path{doi:10.1016/j.ejc.2023.103772}}.

\bibitem[BGK{\etalchar{+}}21]{tww2}
{\'{E}}douard Bonnet, Colin Geniet, Eun~Jung Kim, St{\'{e}}phan Thomass{\'{e}}, and R{\'{e}}mi Watrigant.
\newblock Twin-width {II:} small classes.
\newblock In D{\'{a}}niel Marx, editor, {\em Proceedings of the 2021 {ACM-SIAM} Symposium on Discrete Algorithms, {SODA} 2021, Virtual Conference, January 10 - 13, 2021}, pages 1977--1996. {SIAM}, 2021.
\newblock \href {https://doi.org/10.1137/1.9781611976465.118} {\path{doi:10.1137/1.9781611976465.118}}.

\bibitem[BGK{\etalchar{+}}22]{tww2-journal}
Édouard Bonnet, Colin Geniet, Eun~Jung Kim, Stéphan Thomassé, and Rémi Watrigant.
\newblock {Twin-width II: Small classes}.
\newblock {\em Combinatorial Theory}, 2(2), 2022.
\newblock \href {https://doi.org/10.5070/c62257876} {\path{doi:10.5070/c62257876}}.

\bibitem[BGOdM{\etalchar{+}}21]{tww4-arxiv}
{\'E}douard Bonnet, Ugo Giocanti, Patrice Ossona~de Mendez, Pierre Simon, St{\'e}phan Thomass{\'e}, and Szymon Toru{\'n}czyk.
\newblock Twin-width iv: ordered graphs and matrices.
\newblock {\em arXiV}, 2021.
\newblock \href {https://arxiv.org/abs/2102.03117} {\path{arXiv:2102.03117}}.

\bibitem[BGOdM{\etalchar{+}}22]{tww4}
\'{E}douard Bonnet, Ugo Giocanti, Patrice Ossona~de Mendez, Pierre Simon, St\'{e}phan Thomass\'{e}, and Szymon Toru\'{n}czyk.
\newblock Twin-width iv: Ordered graphs and matrices.
\newblock In {\em Proceedings of the 54th Annual ACM SIGACT Symposium on Theory of Computing}, STOC 2022, pages 924--937, New York, NY, USA, 2022. Association for Computing Machinery.
\newblock \href {https://doi.org/10.1145/3519935.3520037} {\path{doi:10.1145/3519935.3520037}}.

\bibitem[BKR{\etalchar{+}}22]{tww-polynomial-kernels}
{\'E}douard Bonnet, Eun~Jung Kim, Amadeus Reinald, St{\'e}phan Thomass{\'e}, and R{\'e}mi Watrigant.
\newblock Twin-width and polynomial kernels.
\newblock {\em Algorithmica}, 84(11):3300--3337, 2022.
\newblock \href {https://doi.org/10.1007/s00453-022-00965-5} {\path{doi:10.1007/s00453-022-00965-5}}.

\bibitem[BKTW20]{tww1}
{\'{E}}douard Bonnet, Eun~Jung Kim, St{\'{e}}phan Thomass{\'{e}}, and R{\'{e}}mi Watrigant.
\newblock Twin-width {I:} tractable {FO} model checking.
\newblock In Sandy Irani, editor, {\em 61st {IEEE} Annual Symposium on Foundations of Computer Science, {FOCS} 2020, Durham, NC, USA, November 16-19, 2020}, pages 601--612. {IEEE}, 2020.
\newblock \href {https://doi.org/10.1109/FOCS46700.2020.00062} {\path{doi:10.1109/FOCS46700.2020.00062}}.

\bibitem[BL21]{Braunfeld2021CharacterizationsOM}
Samuel Braunfeld and Michael~C. Laskowski.
\newblock Characterizations of monadic nip.
\newblock {\em Transactions of the American Mathematical Society, Series B}, 2021.

\bibitem[BL22]{braunfeld2022existential}
Samuel Braunfeld and Michael~C Laskowski.
\newblock Existential characterizations of monadic {NIP}.
\newblock {\em arXiv preprint arXiv:2209.05120}, 2022.

\bibitem[BNdMS22a]{horizons}
Samuel Braunfeld, Jaroslav Ne{\v s}et{\v r}il, Patrice~Ossona de~Mendez, and Sebastian Siebertz.
\newblock Decomposition horizons: from graph sparsity to model-theoretic dividing lines, 2022.
\newblock \href {https://arxiv.org/abs/2209.11229} {\path{arXiv:2209.11229}}.

\bibitem[BNdMS22b]{transduction-quasiorder}
Samuel Braunfeld, Jaroslav Nesetril, Patrice~Ossona de~Mendez, and Sebastian Siebertz.
\newblock On the first-order transduction quasiorder of hereditary classes of graphs.
\newblock {\em arXiv}, abs/2208.14412, 2022.
\newblock \href {https://arxiv.org/abs/2208.14412} {\path{arXiv:2208.14412}}.

\bibitem[BS85]{BS1985monadic}
J.T. Baldwin and S.~Shelah.
\newblock Second-order quantifiers and the complexity of theories.
\newblock {\em Notre Dame Journal of Formal Logic}, 26(3):229--303, 1985.

\bibitem[BT15]{bousquet-thomassee}
Nicolas Bousquet and St{\'e}phan Thomass{\'e}.
\newblock Vc-dimension and erd{\"o}s--p{\'o}sa property.
\newblock {\em Discrete Mathematics}, 338(12):2302--2317, 2015.
\newblock \href {https://doi.org/10.1016/j.disc.2015.05.026} {\path{doi:10.1016/j.disc.2015.05.026}}.

\bibitem[CER93]{clique-width}
Bruno Courcelle, Joost Engelfriet, and Grzegorz Rozenberg.
\newblock Handle-rewriting hypergraph grammars.
\newblock {\em Journal of Computer and System Sciences}, 46(2):218--270, 1993.
\newblock \href {https://doi.org/10.1016/0022-0000(93)90004-G} {\path{doi:10.1016/0022-0000(93)90004-G}}.

\bibitem[CiO07]{courcelle-oum-vertex-minors-seese}
Bruno Courcelle and Sang il~Oum.
\newblock Vertex-minors, monadic second-order logic, and a conjecture by seese.
\newblock {\em Journal of Combinatorial Theory, Series B}, 97(1):91--126, 2007.
\newblock \href {https://doi.org/10.1016/j.jctb.2006.04.003} {\path{doi:10.1016/j.jctb.2006.04.003}}.

\bibitem[DGK{\etalchar{+}}22a]{bushes-lics}
Jan Dreier, Jakub Gajarsk{\'{y}}, Sandra Kiefer, Michal Pilipczuk, and Szymon Torunczyk.
\newblock Treelike decompositions for transductions of sparse graphs.
\newblock In Christel Baier and Dana Fisman, editors, {\em {LICS} '22: 37th Annual {ACM/IEEE} Symposium on Logic in Computer Science, Haifa, Israel, August 2 - 5, 2022}, pages 31:1--31:14. {ACM}, 2022.
\newblock \href {https://doi.org/10.1145/3531130.3533349} {\path{doi:10.1145/3531130.3533349}}.

\bibitem[DGK{\etalchar{+}}22b]{bushes-arxiv}
Jan Dreier, Jakub Gajarsk{\'{y}}, Sandra Kiefer, Michal Pilipczuk, and Szymon Torunczyk.
\newblock Treelike decompositions for transductions of sparse graphs.
\newblock {\em arXiv}, abs/2201.11082, 2022.
\newblock \href {https://arxiv.org/abs/2201.11082} {\path{arXiv:2201.11082}}.

\bibitem[DGKS07]{model-theory-makes-formulas-large}
Anuj Dawar, Martin Grohe, Stephan Kreutzer, and Nicole Schweikardt.
\newblock Model theory makes formulas large.
\newblock In Lars Arge, Christian Cachin, Tomasz Jurdzi{\'{n}}ski, and Andrzej Tarlecki, editors, {\em Automata, Languages and Programming}, pages 913--924, Berlin, Heidelberg, 2007. Springer Berlin Heidelberg.

\bibitem[DK06]{dingkotlov}
Guoli Ding and Andrei Kotlov.
\newblock On minimal rank over finite fields.
\newblock {\em ELA. The Electronic Journal of Linear Algebra [electronic only]}, 15:210--214, 2006.
\newblock URL: \url{http://dml.mathdoc.fr/item/05279950}.

\bibitem[DKT13]{DvorakKT13-journal}
Zdenek Dvor{\'{a}}k, Daniel Kr{\'{a}}l, and Robin Thomas.
\newblock Testing first-order properties for subclasses of sparse graphs.
\newblock {\em J. {ACM}}, 60(5):36:1--36:24, 2013.
\newblock \href {https://doi.org/10.1145/2499483} {\path{doi:10.1145/2499483}}.

\bibitem[DMS23]{dreierMS}
Jan Dreier, Nikolas M{\"{a}}hlmann, and Sebastian Siebertz.
\newblock First-order model checking on structurally sparse graph classes.
\newblock In Barna Saha and Rocco~A. Servedio, editors, {\em Proceedings of the 55th Annual {ACM} Symposium on Theory of Computing, {STOC} 2023, Orlando, FL, USA, June 20-23, 2023}, pages 567--580. {ACM}, 2023.
\newblock \href {https://doi.org/10.1145/3564246.3585186} {\path{doi:10.1145/3564246.3585186}}.

\bibitem[DT17]{pathwidth2ConnDT}
Thanh~N. Dang and Robin Thomas.
\newblock Minors of two-connected graphs of large path-width, 2017.

\bibitem[Dvo12]{dvorakTop}
Zden{\v e}k Dvo{\v r}{\'a}k.
\newblock A stronger structure theorem for excluded topological minors.
\newblock {\em ArXiv}, abs/1209.0129, 2012.

\bibitem[Dvo13]{dvorak-admissibility}
Zden{\v e}k Dvo{\v r}{\'a}k.
\newblock Constant-factor approximation of the domination number in sparse graphs.
\newblock {\em European Journal of Combinatorics}, 34(5):833--840, 2013.
\newblock \href {https://doi.org/10.1016/j.ejc.2012.12.004} {\path{doi:10.1016/j.ejc.2012.12.004}}.

\bibitem[Dvo18]{dvorakInducedSubdivisions}
Zden{\v e}k Dvo{\v r}{\'a}k.
\newblock Induced subdivisions and bounded expansion.
\newblock {\em Eur. J. Comb.}, 69(C):143--148, mar 2018.
\newblock \href {https://doi.org/10.1016/j.ejc.2017.10.004} {\path{doi:10.1016/j.ejc.2017.10.004}}.

\bibitem[FG04]{FRICK20043}
Markus Frick and Martin Grohe.
\newblock The complexity of first-order and monadic second-order logic revisited.
\newblock {\em Annals of Pure and Applied Logic}, 130(1):3--31, 2004.
\newblock Papers presented at the 2002 IEEE Symposium on Logic in Computer Science (LICS).
\newblock \href {https://doi.org/10.1016/j.apal.2004.01.007} {\path{doi:10.1016/j.apal.2004.01.007}}.

\bibitem[FGK{\etalchar{+}}10]{FOMIN20101167}
Fedor~V. Fomin, Petr~A. Golovach, Jan Kratochv{\'\i}l, Nicolas Nisse, and Karol Suchan.
\newblock Pursuing a fast robber on a graph.
\newblock {\em Theoretical Computer Science}, 411(7):1167--1181, 2010.
\newblock \href {https://doi.org/10.1016/j.tcs.2009.12.010} {\path{doi:10.1016/j.tcs.2009.12.010}}.

\bibitem[FKL12]{https://doi.org/10.1002/jgt.20591}
Alan Frieze, Michael Krivelevich, and Po-Shen Loh.
\newblock Variations on cops and robbers.
\newblock {\em Journal of Graph Theory}, 69(4):383--402, 2012.
\newblock \href {https://doi.org/10.1002/jgt.20591} {\path{doi:10.1002/jgt.20591}}.

\bibitem[FT08]{FOMIN2008236}
Fedor~V. Fomin and Dimitrios~M. Thilikos.
\newblock An annotated bibliography on guaranteed graph searching.
\newblock {\em Theoretical Computer Science}, 399(3):236--245, 2008.
\newblock Graph Searching.
\newblock \href {https://doi.org/10.1016/j.tcs.2008.02.040} {\path{doi:10.1016/j.tcs.2008.02.040}}.

\bibitem[GHN{\etalchar{+}}12]{shrubdepth-mfcs}
Robert Ganian, Petr Hlinen{\'{y}}, Jaroslav Nesetril, Jan Obdrz{\'{a}}lek, Patrice~Ossona de~Mendez, and Reshma Ramadurai.
\newblock When trees grow low: Shrubs and fast {MSO1}.
\newblock In Branislav Rovan, Vladimiro Sassone, and Peter Widmayer, editors, {\em Mathematical Foundations of Computer Science 2012 - 37th International Symposium, {MFCS} 2012, Bratislava, Slovakia, August 27-31, 2012. Proceedings}, volume 7464 of {\em Lecture Notes in Computer Science}, pages 419--430. Springer, 2012.
\newblock \href {https://doi.org/10.1007/978-3-642-32589-2\_38} {\path{doi:10.1007/978-3-642-32589-2\_38}}.

\bibitem[GHO{\etalchar{+}}20]{new-perspective-journal}
Jakub Gajarsk\'{y}, Petr Hlin\v{e}n\'{y}, Jan Obdr\v{z}\'{a}lek, Daniel Lokshtanov, and M.~S. Ramanujan.
\newblock A new perspective on fo model checking of dense graph classes.
\newblock {\em ACM Trans. Comput. Logic}, 21(4), 2020.
\newblock \href {https://doi.org/10.1145/3383206} {\path{doi:10.1145/3383206}}.

\bibitem[GKN{\etalchar{+}}18]{lsd-icalp}
Jakub Gajarsk{\'{y}}, Stephan Kreutzer, Jaroslav Nesetril, Patrice~Ossona de~Mendez, Michal Pilipczuk, Sebastian Siebertz, and Szymon Torunczyk.
\newblock First-order interpretations of bounded expansion classes.
\newblock In Ioannis Chatzigiannakis, Christos Kaklamanis, D{\'{a}}niel Marx, and Donald Sannella, editors, {\em 45th International Colloquium on Automata, Languages, and Programming, {ICALP} 2018, July 9-13, 2018, Prague, Czech Republic}, volume 107 of {\em LIPIcs}, pages 126:1--126:14. Schloss Dagstuhl - Leibniz-Zentrum f{\"{u}}r Informatik, 2018.
\newblock \href {https://doi.org/10.4230/LIPIcs.ICALP.2018.126} {\path{doi:10.4230/LIPIcs.ICALP.2018.126}}.

\bibitem[GKN{\etalchar{+}}20]{lsd-journal}
Jakub Gajarsk{\'{y}}, Stephan Kreutzer, Jaroslav Nesetril, Patrice~Ossona de~Mendez, Michal Pilipczuk, Sebastian Siebertz, and Szymon Torunczyk.
\newblock First-order interpretations of bounded expansion classes.
\newblock {\em {ACM} Trans. Comput. Log.}, 21(4):29:1--29:41, 2020.
\newblock \href {https://doi.org/10.1145/3382093} {\path{doi:10.1145/3382093}}.

\bibitem[GKR{\etalchar{+}}18]{doi:10.1137/18M1168753}
Martin Grohe, Stephan Kreutzer, Roman Rabinovich, Sebastian Siebertz, and Konstantinos Stavropoulos.
\newblock Coloring and covering nowhere dense graphs.
\newblock {\em SIAM Journal on Discrete Mathematics}, 32(4):2467--2481, 2018.
\newblock \href {https://doi.org/10.1137/18M1168753} {\path{doi:10.1137/18M1168753}}.

\bibitem[GKS17]{GroheKS17}
Martin Grohe, Stephan Kreutzer, and Sebastian Siebertz.
\newblock Deciding first-order properties of nowhere dense graphs.
\newblock {\em J. {ACM}}, 64(3):17:1--17:32, 2017.
\newblock \href {https://doi.org/10.1145/3051095} {\path{doi:10.1145/3051095}}.

\bibitem[GMM{\etalchar{+}}23]{flippers}
Jakub Gajarsk{\'{y}}, Nikolas M{\"{a}}hlmann, Rose McCarty, Pierre Ohlmann, Michal Pilipczuk, Wojciech Przybyszewski, Sebastian Siebertz, Marek Sokolowski, and Szymon Torunczyk.
\newblock Flipper games for monadically stable graph classes.
\newblock In Kousha Etessami, Uriel Feige, and Gabriele Puppis, editors, {\em 50th International Colloquium on Automata, Languages, and Programming, {ICALP} 2023, July 10-14, 2023, Paderborn, Germany}, volume 261 of {\em LIPIcs}, pages 128:1--128:16. Schloss Dagstuhl - Leibniz-Zentrum f{\"{u}}r Informatik, 2023.
\newblock \href {https://doi.org/10.4230/LIPIcs.ICALP.2023.128} {\path{doi:10.4230/LIPIcs.ICALP.2023.128}}.

\bibitem[GPPT22]{tww-types-icalp}
Jakub Gajarsk{\'{y}}, Michal Pilipczuk, Wojciech Przybyszewski, and Szymon Torunczyk.
\newblock Twin-width and types.
\newblock In Mikolaj Bojanczyk, Emanuela Merelli, and David~P. Woodruff, editors, {\em 49th International Colloquium on Automata, Languages, and Programming, {ICALP} 2022, July 4-8, 2022, Paris, France}, volume 229 of {\em LIPIcs}, pages 123:1--123:21. Schloss Dagstuhl - Leibniz-Zentrum f{\"{u}}r Informatik, 2022.
\newblock \href {https://doi.org/10.4230/LIPIcs.ICALP.2022.123} {\path{doi:10.4230/LIPIcs.ICALP.2022.123}}.

\bibitem[GPT21]{stabletww-arxiv}
Jakub Gajarsk{\'y}, Micha{\l} Pilipczuk, and Szymon Toru{\'n}czyk.
\newblock Stable graphs of bounded twin-width.
\newblock {\em arXiv}, 2021.
\newblock \href {https://arxiv.org/abs/2107.03711} {\path{arXiv:2107.03711}}.

\bibitem[GPT22]{stable-tww-lics}
Jakub Gajarsk\'{y}, Micha\l{} Pilipczuk, and Szymon Toru\'{n}czyk.
\newblock Stable graphs of bounded twin-width.
\newblock In {\em Proceedings of the 37th Annual ACM/IEEE Symposium on Logic in Computer Science}, LICS '22, New York, NY, USA, 2022. Association for Computing Machinery.
\newblock \href {https://doi.org/10.1145/3531130.3533356} {\path{doi:10.1145/3531130.3533356}}.

\bibitem[Gro07]{logic-graphs-algorithms}
Martin Grohe.
\newblock Logic, graphs, and algorithms.
\newblock {\em Electron. Colloquium Comput. Complex.}, {TR07-091}, 2007.
\newblock URL: \url{https://eccc.weizmann.ac.il/eccc-reports/2007/TR07-091/index.html}.

\bibitem[GT23]{tww-tournaments}
Colin Geniet and St{\'{e}}phan Thomass{\'{e}}.
\newblock First order logic and twin-width in tournaments.
\newblock In Inge~Li G{\o}rtz, Martin Farach{-}Colton, Simon~J. Puglisi, and Grzegorz Herman, editors, {\em 31st Annual European Symposium on Algorithms, {ESA} 2023, September 4-6, 2023, Amsterdam, The Netherlands}, volume 274 of {\em LIPIcs}, pages 53:1--53:14. Schloss Dagstuhl - Leibniz-Zentrum f{\"{u}}r Informatik, 2023.
\newblock \href {https://doi.org/10.4230/LIPIcs.ESA.2023.53} {\path{doi:10.4230/LIPIcs.ESA.2023.53}}.

\bibitem[GW00]{gurskiwanke}
Frank Gurski and Egon Wanke.
\newblock The tree-width of clique-width bounded graphs without kn,n.
\newblock In Ulrik Brandes and Dorothea Wagner, editors, {\em Graph-Theoretic Concepts in Computer Science}, pages 196--205, Berlin, Heidelberg, 2000. Springer Berlin Heidelberg.

\bibitem[HJMW20]{pathwidth2Conn}
Tony Huynh, Gwena\"{e}l Joret, Piotr Micek, and David Wood.
\newblock Seymour's conjecture on 2-connected graphs of large pathwidth.
\newblock {\em Combinatorica}, 40:839--868, 2020.
\newblock \href {https://doi.org/10.1007/s00493-020-3941-3} {\path{doi:10.1007/s00493-020-3941-3}}.

\bibitem[iOS06]{OumSeymour-approximating}
Sang il~Oum and Paul Seymour.
\newblock Approximating clique-width and branch-width.
\newblock {\em Journal of Combinatorial Theory, Series B}, 96(4):514--528, 2006.
\newblock \href {https://doi.org/10.1016/j.jctb.2005.10.006} {\path{doi:10.1016/j.jctb.2005.10.006}}.

\bibitem[KMOW21]{shrubdepth-cmso-paths}
O-joung Kwon, Rose McCarty, Sang-il Oum, and Paul Wollan.
\newblock Obstructions for bounded shrub-depth and rank-depth.
\newblock {\em Journal of Combinatorial Theory, Series B}, 149:76--91, 2021.
\newblock \href {https://doi.org/10.1016/j.jctb.2021.01.005} {\path{doi:10.1016/j.jctb.2021.01.005}}.

\bibitem[KO04]{kuhn-osthus}
Daniela K{\"u}hn and Deryk Osthus.
\newblock Induced subdivisions in ks,s-free graphs of large average degree.
\newblock {\em Combinatorica}, 24(2):287--304, 2004.
\newblock \href {https://doi.org/10.1007/s00493-004-0017-8} {\path{doi:10.1007/s00493-004-0017-8}}.

\bibitem[KPS17]{rankwidth-colorings}
O-joung Kwon, Micha{\l} Pilipczuk, and Sebastian Siebertz.
\newblock On low rank-width colorings.
\newblock In Hans~L. Bodlaender and Gerhard~J. Woeginger, editors, {\em Graph-Theoretic Concepts in Computer Science}, pages 372--385, Cham, 2017. Springer International Publishing.

\bibitem[Lam12]{lampis}
Michael Lampis.
\newblock Algorithmic meta-theorems for restrictions of treewidth.
\newblock {\em Algorithmica}, 64(1):19--37, 2012.
\newblock \href {https://doi.org/10.1007/s00453-011-9554-x} {\path{doi:10.1007/s00453-011-9554-x}}.

\bibitem[LPPT20]{otherthilikos}
Stratis Limnios, Christophe Paul, Joanny Perret, and Dimitrios~M. Thilikos.
\newblock Edge degeneracy: Algorithmic and structural results.
\newblock {\em Theoretical Computer Science}, 839:164--175, 2020.
\newblock \href {https://doi.org/10.1016/j.tcs.2020.06.006} {\path{doi:10.1016/j.tcs.2020.06.006}}.

\bibitem[Mat04]{Matousek:p-q-theorem}
Ji{\v{r}}{\'{\i}} Matou{\v{s}}ek.
\newblock Bounded {VC}-dimension implies a fractional {H}elly theorem.
\newblock {\em Discrete Comput. Geom.}, 31(2):251--255, 2004.

\bibitem[MT04]{marcus-tardos}
Adam Marcus and G{\'a}bor Tardos.
\newblock Excluded permutation matrices and the stanley--wilf conjecture.
\newblock {\em Journal of Combinatorial Theory, Series A}, 107(1):153--160, 2004.
\newblock \href {https://doi.org/10.1016/j.jcta.2004.04.002} {\path{doi:10.1016/j.jcta.2004.04.002}}.

\bibitem[NdM12]{sparsity-book}
Jaroslav Nesetril and Patrice~Ossona de~Mendez.
\newblock {\em Sparsity - Graphs, Structures, and Algorithms}, volume~28 of {\em Algorithms and combinatorics}.
\newblock Springer, 2012.
\newblock \href {https://doi.org/10.1007/978-3-642-27875-4} {\path{doi:10.1007/978-3-642-27875-4}}.

\bibitem[NdMP{\etalchar{+}}21]{rankwidth-meets-stability}
Jaroslav Ne\v{s}et\v{r}il, Patrice~Ossona de~Mendez, Micha\l{} Pilipczuk, Roman Rabinovich, and Sebastian Siebertz.
\newblock Rankwidth meets stability.
\newblock In {\em Proceedings of the Thirty-Second Annual ACM-SIAM Symposium on Discrete Algorithms}, SODA '21, pages 2014--2033, USA, 2021. Society for Industrial and Applied Mathematics.

\bibitem[NiO20]{nguyenoum}
Huy-Tung Nguyen and Sang il~Oum.
\newblock The average cut-rank of graphs.
\newblock {\em European Journal of Combinatorics}, 90:103183, 2020.
\newblock \href {https://doi.org/10.1016/j.ejc.2020.103183} {\path{doi:10.1016/j.ejc.2020.103183}}.

\bibitem[NO08]{grad-and-bounded-expansion-Nesetril}
Jaroslav Ne{\v s}et{\v r}il and Patrice {Ossona de Mendez}.
\newblock Grad and classes with bounded expansion i. decompositions.
\newblock {\em European Journal of Combinatorics}, 29(3):760--776, 2008.
\newblock \href {https://doi.org/10.1016/j.ejc.2006.07.013} {\path{doi:10.1016/j.ejc.2006.07.013}}.

\bibitem[NO11]{NesetrilM11a}
Jaroslav Ne\v{s}et\v{r}il and Patrice {Ossona de Mendez}.
\newblock On nowhere dense graphs.
\newblock {\em Eur. J. Comb.}, 32(4):600--617, 2011.
\newblock \href {https://doi.org/10.1016/j.ejc.2011.01.006} {\path{doi:10.1016/j.ejc.2011.01.006}}.

\bibitem[NORS21]{bounded-linear-rankwidth}
Jaroslav Ne{\v s}et{\v r}il, Patrice {Ossona de Mendez}, Roman Rabinovich, and Sebastian Siebertz.
\newblock Classes of graphs with low complexity: The case of classes with bounded linear rankwidth.
\newblock {\em European Journal of Combinatorics}, 91:103223, 2021.
\newblock Colorings and structural graph theory in context (a tribute to Xuding Zhu).
\newblock \href {https://doi.org/10.1016/j.ejc.2020.103223} {\path{doi:10.1016/j.ejc.2020.103223}}.

\bibitem[Prz23]{tww-distality}
Wojciech Przybyszewski.
\newblock Distal combinatorial tools for graphs of bounded twin-width.
\newblock In {\em {LICS}}, pages 1--13, 2023.
\newblock \href {https://doi.org/10.1109/LICS56636.2023.10175719} {\path{doi:10.1109/LICS56636.2023.10175719}}.

\bibitem[PST18]{number-of-types}
Michal Pilipczuk, Sebastian Siebertz, and Szymon Torunczyk.
\newblock On the number of types in sparse graphs.
\newblock In Anuj Dawar and Erich Gr{\"{a}}del, editors, {\em Proceedings of the 33rd Annual {ACM/IEEE} Symposium on Logic in Computer Science, {LICS} 2018, Oxford, UK, July 09-12, 2018}, pages 799--808. {ACM}, 2018.
\newblock \href {https://doi.org/10.1145/3209108.3209178} {\path{doi:10.1145/3209108.3209178}}.

\bibitem[PZ78]{podewski1978stable}
Klaus-Peter Podewski and Martin Ziegler.
\newblock Stable graphs.
\newblock {\em Fundamenta Mathematicae}, 100(2):101--107, 1978.

\bibitem[RS84]{robertson-seymour-tw}
Neil Robertson and P.D Seymour.
\newblock Graph minors. iii. planar tree-width.
\newblock {\em Journal of Combinatorial Theory, Series B}, 36(1):49--64, 1984.
\newblock \href {https://doi.org/10.1016/0095-8956(84)90013-3} {\path{doi:10.1016/0095-8956(84)90013-3}}.

\bibitem[RT08]{richerby-thilikos-lazy-fugitive}
David Richerby and Dimitrios~M. Thilikos.
\newblock Searching for a visible, lazy fugitive.
\newblock In Hajo Broersma, Thomas Erlebach, Tom Friedetzky, and Dani{\"{e}}l Paulusma, editors, {\em Graph-Theoretic Concepts in Computer Science, 34th International Workshop, {WG} 2008, Durham, UK, June 30 - July 2, 2008. Revised Papers}, volume 5344 of {\em Lecture Notes in Computer Science}, pages 348--359, 2008.
\newblock \href {https://doi.org/10.1007/978-3-540-92248-3\_31} {\path{doi:10.1007/978-3-540-92248-3\_31}}.

\bibitem[Sau72]{sauer}
N~Sauer.
\newblock On the density of families of sets.
\newblock {\em Journal of Combinatorial Theory, Series A}, 13(1):145--147, 1972.
\newblock \href {https://doi.org/10.1016/0097-3165(72)90019-2} {\path{doi:10.1016/0097-3165(72)90019-2}}.

\bibitem[See96]{seese_1996}
Detlef Seese.
\newblock Linear time computable problems and first-order descriptions.
\newblock {\em Mathematical Structures in Computer Science}, 6(6):505--526, 1996.
\newblock \href {https://doi.org/10.1017/S0960129500070079} {\path{doi:10.1017/S0960129500070079}}.

\bibitem[She72]{shelah-sauer-lemma}
Saharon Shelah.
\newblock {A combinatorial problem; stability and order for models and theories in infinitary languages.}
\newblock {\em Pacific Journal of Mathematics}, 41(1):247 -- 261, 1972.
\newblock \href {https://doi.org/pjm/1102968432} {\path{doi:pjm/1102968432}}.

\bibitem[She86]{Shelah1986}
Saharon Shelah.
\newblock {\em Monadic logic: Hanf Numbers}, pages 203--223.
\newblock Springer Berlin Heidelberg, Berlin, Heidelberg, 1986.
\newblock \href {https://doi.org/10.1007/BFb0098511} {\path{doi:10.1007/BFb0098511}}.

\bibitem[Sim90]{SIMON199065}
Imre Simon.
\newblock Factorization forests of finite height.
\newblock {\em Theoretical Computer Science}, 72(1):65--94, 1990.
\newblock \href {https://doi.org/10.1016/0304-3975(90)90047-L} {\path{doi:10.1016/0304-3975(90)90047-L}}.

\bibitem[Sim15]{simon_book}
Pierre Simon.
\newblock {\em A Guide to NIP Theories}.
\newblock Lecture Notes in Logic. Cambridge University Press, 2015.
\newblock \href {https://doi.org/10.1017/CBO9781107415133} {\path{doi:10.1017/CBO9781107415133}}.

\bibitem[SS71]{Shelah1971StabilityTF}
Saharon Shelah and Saharon Shelah.
\newblock Stability, the f.c.p., and superstability; model theoretic properties of formulas in first order theory.
\newblock {\em Annals of Mathematical Logic}, 3:271--362, 1971.

\bibitem[ST93]{seymour-thomas-cops}
P.D. Seymour and R.~Thomas.
\newblock Graph searching and a min-max theorem for tree-width.
\newblock {\em Journal of Combinatorial Theory, Series B}, 58(1):22--33, 1993.
\newblock \href {https://doi.org/10.1006/jctb.1993.1027} {\path{doi:10.1006/jctb.1993.1027}}.

\bibitem[war16]{warwick-problems}
Algorithms, logic and structure workshop in warwick, open problem session.
\newblock \url{https://warwick.ac.uk/fac/sci/maths/people/staff/daniel_kral/alglogstr/openproblems.pdf}, 2016.
\newblock [Online; accessed 23-Jan-2023].

\bibitem[Zhu09]{zhu2009colouring}
Xuding Zhu.
\newblock Colouring graphs with bounded generalized colouring number.
\newblock {\em Discrete Mathematics}, 309(18):5562--5568, 2009.

\end{thebibliography}
\end{document}